\documentclass[a4paper,12pt]{article}
\usepackage{amsmath,amsthm,amsfonts,amssymb,bbm}
\usepackage{graphicx,psfrag,tikz}
\usepackage{subcaption}
\usepackage{cite}
\usepackage{hyperref}
\hypersetup{
	colorlinks   = true, %Colours links instead of ugly boxes
	urlcolor     = blue, %Colour for external hyperlinks
	linkcolor    = red, %Colour of internal links
	citecolor   = blue %Colour of citations
}
\usepackage{bm}
\usepackage{enumitem}
\usepackage[hmargin=1in,vmargin=1in]{geometry}
\usepackage{mathtools}
\usepackage{stmaryrd}

\renewcommand{\P}{\mathbb{P}}
\newcommand{\E}{\mathbb E}

\newcommand{\N}{\mathbb N}
\newcommand{\R}{\mathbb R}
\newcommand{\Z}{\mathbb Z}

\DeclarePairedDelimiter\floor{\lfloor}{\rfloor}
\DeclarePairedDelimiter\ceil{\lceil}{\rceil}

\providecommand{\abs}[1]{\vert#1\vert}
\newcommand{\cA}{\mathcal{A}}
\newcommand{\cB}{\mathcal{B}}
\newcommand{\cC}{\mathcal{C}}
\newcommand{\cD}{\mathcal{D}}
\newcommand{\cE}{\mathcal{E}}

\newcommand{\cL}{\mathcal{L}}
\newcommand{\cM}{\mathcal{M}}

\newcommand{\cS}{\mathcal{S}}

\newcommand{\cW}{\mathcal{W}}
\newcommand{\cX}{\mathcal{X}}
\newcommand{\cY}{\mathcal{Y}}
\newcommand{\cZ}{\mathcal{Z}}

\newcommand{\h}[1]{\overline{#1}}
\newcommand{\qs}{\cX} %the space of queues inputs -  R^Z

\newtheorem{theorem}{Theorem}[section]
\newtheorem{lemma}[theorem]{Lemma}
\newtheorem{proposition}[theorem]{Proposition}
\newtheorem{corollary}[theorem]{Corollary}

\theoremstyle{remark}
\newtheorem{remark}[theorem]{Remark}

\def\xin{x_0} %the spatial interval
\def\nc{\mathfrak{D}} % new construction

\def\en{M}
\def\arr{a}

\def\serv{s}
\def\depa{d}

\def\md{\cD}

\def\arrv{\mathbf\arr}

\def\servv{\mathbf\serv}
\def\depav{\mathbf\depa}

\def\rhov{\bm{\rho}}
\def\denp{\Gamma}
\def\Mx{\mathcal{M}}
\def\Ind{\mathcal{E}_{ind}}
\def\cq{\Phi}
\def\cm{\Psi}
\numberwithin{equation}{section}
\numberwithin{figure}{section}

\author{Ofer Busani\thanks{University of Bristol, School of Mathematics, Fry Building, Woodland Rd., Bristol BS8 1UG, UK. O. Busani was supported by EPSRC's EP/R021449/1 Standard Grant. E-mail: {\tt o.busani@bristol.ac.uk}} 
}
\title{Diffusive scaling limit of the Busemann process in Last Passage Percolation}

\date{\today}

\begin{document}

\sloppy

\maketitle
\begin{abstract}
	In exponential last passage percolation, we consider the rescaled Busemann process $x\mapsto N^{-1/3}B^\rho_{0,[xN^{2/3}]e_1} \,\, (x\in\R)$, as a process parametrized by the scaled density $\rho=1/2+\frac{\mu}{4} N^{-1/3}$, and taking values in $C(\R)$. We show that these processes, as $N\rightarrow \infty$, have a c\`adl\`ag scaling limit $G=(G_\mu)_{\mu\in \R}$, parametrized by $\mu$ and taking values in $C(\R)$. The limiting process $G$, which can be thought of as the Busemann process under the KPZ scaling, can be described as an ensemble of "sticky" lines of Brownian regularity. We believe $G$ is the universal scaling limit of Busemann processes in the KPZ universality class. Our proof provides  insight into this limiting behaviour by highlighting a connection between the joint distribution of Busemann functions obtained by Fan and Sepp\"al\"ainen in \cite{FS18}, and a sorting algorithm of random walks introduced by O'Connell and Yor in \cite{OcY02}.  
\end{abstract}
	\tableofcontents
\section{Introduction}
In \textit{last passage percolation} (LPP), i.i.d weights are assigned to each vertex of $\Z^2$. If $x,y\in\Z^2$ such that $y$ is above and to the right of $x$, then one defines the last passage time $L(x,y)$ as the maximum total weight collected by an up-right path on the lattice, starting at $x$ and terminating at $y$.  A path on which the total weight is attained is called a \textit{geodesic} from $x$ to $y$. Lattice LPP belongs to a  large family of LPP models, which can be viewed as random directed metric spaces \cite{DV21}, which seem to share the same limiting behaviour of the fluctuations of their observables. Since the seminal work of Rost \cite{R81} where the first order of $L(\bm{0},y)$ was established for $y$ large in exponential LPP, much progress has been made in the study of these models \cite{BDJ99,PS02,MQR17,DOV18,JRAS19}. 

One of the aspects of LPPs and random metric-like models that has been studied in the past three decades is that of infinite geodesics. An infinite geodesic is an up-right path on the random environment, such that its restriction to between any two points on it is a geodesic. Questions such as the existence and uniqueness of  infinite geodesics as well as other finer results have been studied in \textit{first passage percolation} (FPP) by C. Newman and co-authors \cite{HN97,HN01,N95,LN96}, while results of that flavour for various  LPP models can be found in \cite{GRAS17,P16,BHS18,BSS19,BBS20a,JRAS19,SX20,SS21a,SS21b,MV21}. One of the main tools of studying infinite geodesics in random (directed) metrics models is the Busemann function - a random stationary real valued function defined on $\Z^2\times\Z^2$  which satisfies the cocycle property and a so-called Burke's property.  Originating in hyperbolic geometry, the Busemann functions were introduced first to FPP by Newman \cite{N95} and Hoffman \cite{H08} and later to LPP in \cite{GRAS17b}. They underlie many of the techniques used to study infinite geodesics \cite{BBS20a,SX20} and have deep connections to the set of exceptional directions at which uniqueness of the geodesics fail \cite{JRAS19}. They  are also important in the study of  point-to-point geodesics \cite{BBS20}, their coalescence \cite{BF20} and the regularity of the passage time profile around a point \cite{P16,P18,BBS20,BF20,P21}. Busemann fucntions are usually associated with a direction on the first quadrant of the lattice -- for $\rho\in(0,1)$ one defines the Busemann function $B^\rho$ by essentially setting $B^\rho(x,y)$ to be the difference in passage time between two  geodesics starting from $x$ and $y$ and going in a direction that is parametrised by $\rho\in(0,1)$. $B^\rho$ turns out to hold much (and in some models all) the information about infinite geodesics emanating from every point of the lattice, in the direction associated with $\rho$. It is possible, to construct Busemann functions of different intensity $\rho$ on one probability space to obtain the \textit{Busemann process} $\bm{B}=\{B^\rho\}_{\rho\in(0,1)}$. The process $\bm{B}$ turns out to be very rich, and in the case of Exponential LPP, it was shown to hold all the information on the exceptional directions of non uniqueness of infinite geodesics \cite{JRAS19}. For more on the magic of Busemann functions in the context of LPP see \cite{S18,RA18}. 

The random picture of passage times and geodesics has an interesting scaling limit achieved through an incremental progress over the past three decades  \cite{BDJ99,PS02,MQR17} (to mention a few), culminating in showing that, under the (1:2:3)-KPZ scaling, the microscopic metric-like space of LPP converges to an object dubbed the Directed Landscape(DL) \cite{DOV18}. The DL is a continuous function $\mathcal{L}: \R^2\times \R^2 \rightarrow \R$ holding much of the information on the  limiting fluctuations of LPP models in the KPZ class. The prelimiting geodesics, in turn, scale to  continuous functions on the plane - the geodesics of the DL. While much of the work in recent years have been to understand the scaling limit of the LPP picture on a compact set in the DL, the question of the scaling behavior of infinite geodesics to that of those in the DL is interesting, and first steps were taken very recently in  \cite{MV21} where the existence of infinite geodesics and Busemann functions was proven in the DL.    

In that context, it would be interesting to see what is the Busemann process in the DL \cite[Remark 1.3]{D21} i.e.\ does the  limit 
\begin{align}\label{Bpdl}
	B_{DL}^z(x,0;y,t)=\lim_{r\rightarrow\infty}\mathcal{L}(rz,-r;x,0)-\mathcal{L}(rz,-r;y,t) \qquad \forall z,x,y,t\in \R,
\end{align} 
exist and what are its properties. Note that the index $z$ in \eqref{Bpdl} plays the role of the direction while $(x,s)$ and $(y,t)$ are points on the plane. The existence of $B_{DL}^z$ for every $z\in\R$ was proven very recently in \cite[Section 3.8]{MV21}. Following the same progress as in that of the lattice LPP, it would be interesting to see what is the  distribution of the Busemann process $z\mapsto B_{DL}^z$, its properties, and its connection to infinite geodesics in the DL. From the co-cycle property of the Busemann functions, the study of $B_{DL}$ can be broken down into two steps: a) study the Busemann process on one time horizon, i.e. study the process $H_{DL}^z$ defined by $z\mapsto B_{DL}^z(0,0;x,0)$, b) push the process $H_{DL}^z$ through $\mathcal{L}$ to obtain $B_{DL}$ on different time horizons i.e.\
\begin{align}
	B_{DL}^z(0,0;x,t)=\sup_{y}H^z_{DL}(y)\mathcal{L}(y,0;x,t) \qquad z\in\R, t>0.
\end{align}

From the stationarity of the Busemann process $\bm{B}$ on the lattice and the (1:2:3)-KPZ limit, it is expected that $H_{DL}$ should emerge as the diffusive scaling limit of 
\begin{align*}
	H^\rho_x= B^{1/2+\rho}_{(0,0),xe_1} \qquad x\in\R,\rho\in(-1/2,1/2),
\end{align*}
that is,  we expect the following sequence
\begin{align}\label{Gin}
	G^N_\rho(x):= N^{-1/3}\big[H^{\frac{\rho}{4}N^{-1/3}\rho(z)}_{xN^{2/3}}-2xN^{2/3}\big],
\end{align}
to converge in some sense to an object parameterized by the intensity $\rho$ and taking values in $C(\R)$. This limit should, perhaps under some reparametrization, be equal in distribution to $H_{DL}$, and so, should be universal in the KPZ class.

\textit{The main motivation behind this work is the study of $B_{DL}$ and the geodesics associated to it in the DL. This work  establishes  the limit $G$ of the sequence in \eqref{Gin}, which we call the Stationary Horizon (SH), and studies its properties.} 

Only a few months after the first draft of this paper, Sepp\"al\"ainen and Sorensen obtained, among other results, the joint distribution of the unscaled Busemann process in Brownian LPP (BLPP), i.e.\ when the set of directions by which the process is parameterized is macroscopic, as opposed to our work where we scale the directions around a specific macroscopic direction. Interestingly, although not a proven result yet, it seems that the (unscaled) Busemann process in BLPP is essentially the SH. From that perspective, the results in \cite{SS21a,SS21b} would complement the results in this work. One possible explanation for the SH appearance as Busemann process of the BLPP without scaling the direction, is that the BLPP can be thought of as a lattice LPP after a diffusive scaling. It is then not surprising that the Busemann process in BLPP has a scaling invariance property in that direction.  

Regarding our proof and techniques, our starting point will be the results of Fan and Sepp\"al\"ainen in \cite{FS18} where the finite dimensional distribution of the Busemann process was obtained. Our main technique is a   "melonization" representation of the map presented in \cite{FS18},  which, as it turns out, is closely related  to the RSK picture as non stationary queues developed in \cite{Occ03,BBO05} by O'Connell and co-authors. Loosely speaking, melonization of a finite sequence of random walks (RWs) is a sorting process that results in a sequence of ordered RWs and is closely related to the dynamics of queues in tandem. The reason for the term melonization is that by Greene's Theorem the sum of the height of the $k$ top sorted RWs has the same distribution as that of the maximal passage weight collected by  $k$ non-intersecting paths, which arguably looks like a melon. Even though we do not study a possible connection between our "stationary RSK" and  $k$-paths maximal weight in this paper, we shall still refer to  our sorting process as melonization, to point out the relation to previous work in the standard RSK model.  Our representation does not seem to be completely new, however, our derivation of it is new, and its interpretation as melonization of  random walks is proved important in this work, and appears to be novel in the context of two sided RWs. We discuss this at length in Subsection \ref{subsec:mel}. 
%As our random walks are two-sided, their melonization orders any two RWs not only with respect to their paths past a certain fixed point ("future") but also with respect to their "past". This is due to the fact that in stationarity, the queues are not necessarily empty prior to a fixed point in time and carry a "waiting time" as opposed to the picture in \cite{Occ03}. 

This melonization dynamics is behind  one of the most salient features of the SH -  although $G$ is two dimensional as function of space $x$ and its drift $\mu$, it can be represented as an ensemble of "sticky" lines (Figure \ref{sfig:SH2}). Alternatively, we can (and will) think of $G$ as a right continuous with left limits (rcll) monotone process which is a jump process when restricted to any compact interval (Figure \ref{sfig:SH}). This stickiness phenomenon is in line with similar prelimiting and limiting behavior of LPP models studied in the literature, for the Difference Profile \cite{BGH19,BGH21,GH21}, the Busemann process in lattice LPP \cite{JRAS19,FS18} and the Shock Measures \cite{D21,MV21,DV21}. This in turn shows, that the rcll microscopic behaviour of the Busemann function carries over to its diffusive limit. It also implies that infinite geodesics associated with a continuum of directions, emanating from a single point in the DL form a tree (see \ref{sfig:dgeo2}). We defer this, and the connection between the SH and the infinite geodesics in the DL to a future work. 
\subsection{Notation}
 We denote by $\N=\{1,2,...\}$ the set of natural numbers. For $\lambda>0$, we write $X\sim \rm{Exp}(\lambda)$ if $X$ is a random variable s.t. $\P(X>t)=e^{-\lambda t}$ for $t>0$. For $\rho_1,\rho_2>0$ and  two independent random variables $X^1\sim \mathrm{Exp}(\rho_1)$, $X^2\sim \mathrm{Exp}(\rho_2)$ we write $X\sim \mathrm{Exp}(\rho^2,\rho^1)$, if $X\sim X^2-X^1$. For $x,y\in \R \cup \{\infty,-\infty\}$, we denote by $\llbracket x,y \rrbracket$ the set $\Z \cap [x,y]$. $e_1$ and $e_2$ stand for $(1,0)$ and $(0,1)$ respectively.
\begin{figure}[t]	
	\centering
	\begin{subfigure}{0.55\textwidth}
		\includegraphics[width=0.9 \textwidth]{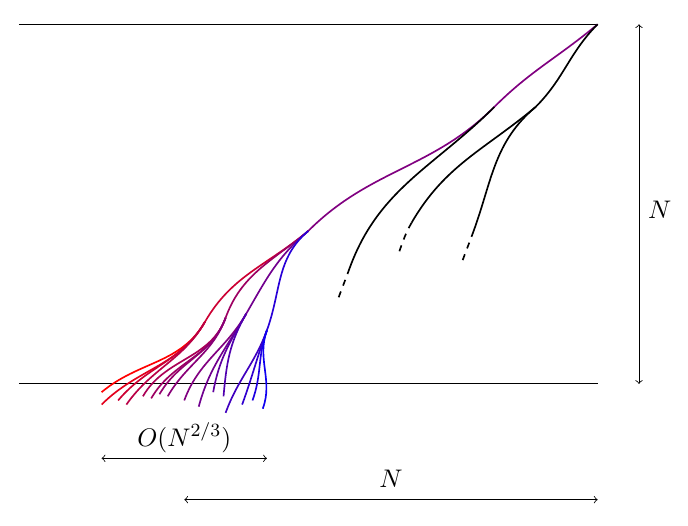} 
		\caption{Dense crossings of geodesics}
		\label{sfig:dgeo}
	\end{subfigure}%
	\begin{subfigure}{0.55\textwidth}
		\includegraphics[width=0.9 \textwidth]{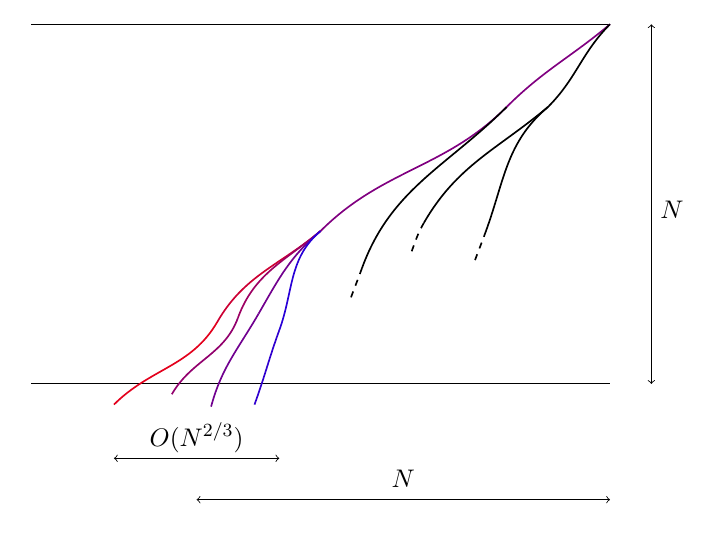}
		\caption{Sparse crossings of geodesics} 
		\label{sfig:dgeo2}
	\end{subfigure}
	
	\caption{Two possible behaviours of infinite geodesics starting from $(N,N)$. From the point $(N,N)$ we consider infinite geodesics  going in the direction $(1-\alpha,1+\alpha)$ where $\alpha=O(N^{-1/3})$. These geodesics will cross the $x$-axis within $O(N^{2/3})$ distance of the origin. Different colours represent different intensities. Figure \ref{sfig:dgeo}) The set of crossing points of the geodesics with the $x$-axis is "dense" in $\R$ i.e.\ within any interval of size $\Theta(N^{2/3})$ one can find a crossing with high probability. Our work seems show that this picture is false. Figure \ref{sfig:dgeo2})  The set of crossing points of the geodesics with the $x$-axis on an interval of size $O(N^{2/3})$ is finite. Our work implies that this picture is the correct one. }
	\label{fig:dgeo}
\end{figure}
\subsection{Main results}
 Let $\cS=C(\R)$ be the set of continuous functions on the real line equipped with the topology of uniform convergence on compact intervals, and let $D(\R,\cS)$ be the space of $\cS$-valued rcll functions  on the real line, equipped with the Skorohord topology (see Subsection \ref{subsec:sh2} for the exact definition). Let $\bm{B}=\{B^\rho\}_{\rho\in(0,1)}$ be the Busemann process as defined in Section \ref{sec:bus}. From  $\bm{B}$ we construct a sequence of processes $G^N=\{G^N_{\mu}\}_{\mu\in\R}$ in $D(\R,\cS)$  as follows. First define $F^N_{\cdot}\in D(\R,\cS)$ through
\begin{align*}
	\text{$F^N_\mu$ is the linear interpolation of the discrete function $\{B^{1/2-\frac{\mu}{4}N^{-1/3}}_{0,ie_1}\}_{i\in \Z}$}, \quad \forall \mu\in \R.
\end{align*}
Then 
\begin{align*}
	G^N_\mu(x)=N^{-1/3}[F^N_\mu(N^{2/3}x)-2N^{2/3}x], \quad \forall \mu\in \R.
\end{align*}
A rigorous definition of $G^N$ is given in \eqref{G}.  One should think of $G^N$ as a process, indexed by $\R$, and whose values are two-sided random walks.  To see where the parametrization $\mu \mapsto 1/2-\frac{\mu}{4}N^{-1/3}$ for $G^N$ comes from, note first that from the properties of the Busemann process (Section \ref{sec:bus}), $B_{\cdot}^{1/2-\frac{\mu}{4}N^{-1/3}}-2\cdot
\,$  is a discrete RW with mean $\mu N^{-1/3}+O(N^{-2/3})$. This in turn means that $G^N_{\mu}$ is a RW with drift $\mu+o(1)$. In other words, the process $G^N$ is parametrized by the drift of the random walk $G^N_{\cdot}$. We are now ready to state our results.

\begin{theorem}\label{thm:sh}
	There exists a  process $G\in D(\R,\cS)$ such that the following limit holds in distribution in $D(\R,\cS)$
	\begin{align*}
		G^N \Rightarrow G  \qquad \text{as $N\rightarrow \infty$}.
	\end{align*} 
\end{theorem}
We call the limiting process $G$ the \textit{Stationary Horizon}. 
Here are some of its properties.
\begin{theorem}\label{thm:prop}
	The Stationary Horizon $G$ satisfies the following:
	\begin{enumerate}
		\item \label{odd} The one-dimensional distribution of $G$ is that of a  Brownian motion with a drift, i.e. 
		\begin{align}
			G_\mu \sim W^\mu \quad \mu\in \R,
		\end{align}
		where $W^\mu$ is a doubly-infinite Brownian motion with drift $\mu$ and diffusivity $2$, conditioned to vanish at the origin.
		\item \label{QM} For every $k\in\N$, there exists a function $\Phi^k:C(\R)^k\rightarrow C(\R)^k$ such that for every $\mu_1<...<\mu_k$ in $\R$
		\begin{align*}
			(G_{\mu_1},...,G_{\mu_k})\sim \Phi^k(W^{\mu_1},...,W^{\mu_k})
		\end{align*}
		where $W^{\mu_1},...,W^{\mu_k}$ are independent two-sided Brownian motions with diffusivity $2$ and drift $\mu_1,...,\mu_k$ respectively,  conditioned to vanish at the origin.
		\item \label{sci}$G$ is scale-invariant with exponents $(1,-2,1)$, i.e.\ for $c>0$
		\begin{align}
			cG_{c\cdot}(c^{-2}\cdot)\sim G.
		\end{align}
		
		\item \label{pc}$G$ is stochastically continuous, i.e. 
		\begin{align*}
			\P( G_\mu-G_{\mu-}=0)=1 \quad \forall \mu\in \R.
		\end{align*}
		\item \label{jp}Let $x_0>0$ and consider the process $G^{x_0}\in D\big(\R,C([-x_0,x_0])\big)$ defined through
		\begin{align}\label{Gx02}
		G^{x_0}_\mu(x)=G_\mu(x)\quad \forall \mu\in\R, x\in[-x_0,x_0].
		\end{align}
		In words,  $G^{x_0}$ retains only the information of $G$ restricted to $[-x_0,x_0]$ (see Figure  \ref{fig:SH}). Then $G^{x_0}$ is a pure jump process in $D\big(\R,C([-x_0,x_0])\big)$, i.e.\ its increments are the sum of its jumps in  $D\big(\R,C([-x_0,x_0])\big)$. Moreover, $G^{x_0}$ has finitely many jumps on any compact interval $[-\mu_0,\mu_0]$, and we have the bound
		\begin{align}\label{ub82}
		\P\big(\#\{\text{$\mu\in[-\mu_0,\mu_0]$: $G^{x_0}$ has a jump at $\mu$ }\}\geq M\big)\leq CM^{-1} \qquad M>0.
		\end{align}
		for some constant $C>0$ depending on $x_0$ and $\mu_0$.
		\item \label{mono}For any real numbers $x_1\leq x_2$, the process 
		\begin{align*}
		\mu \mapsto G_\mu(x_2)-G_\mu(x_1)
		\end{align*}
		is a non-decreasing jump process in $D(\R,\R)$, i.e. a c\`adl\`ag process such that   for any $\mu_1 \leq \mu_2$
		\begin{align}\label{mon2}
		G_{\mu_2}(x_2)-G_{\mu_2}(x_1)\geq G_{\mu_1}(x_2)-G_{\mu_1}(x_1).
		\end{align}
	\end{enumerate}
\end{theorem}
Figure \ref{fig:SH} is an illustration of a typical realization of $G$ on a compact rectangle on the $\mu-x$ plane. Theorem \ref{thm:prop} shows that some of the properties of the microscopic Busemann process carry over in the limit as can be seen by the  similarity between properties \ref{pc}--\ref{mono} of the SH and the fact that $\rho\mapsto B^\rho_{\bm{0},e_1}$ is a Poisson process \cite{FS18}. 
\begin{remark}
	We believe the upper bound \eqref{ub82}, on the number of jumps of the SH on a compact set in $x-\mu$ is far from optimal, and we expect it to be exponential.  
\end{remark}
\begin{remark}
	While we do not provide a proof, we do not expect $G$ to be Markovian.
\end{remark}

Let us now try to give a sketch of the proof of the above result while pointing out the main challenges in it.  In general, in the Skorohord topology on the real line, a sequence of random variables $X^N$ converges to some $X$ if (see \ref{sub:TR})
\begin{enumerate}
	\item  The sequence $X^N$ is tight,\label{i1}
	\item $X^N$ converges in the sense of finite dimensional distributions.\label{i2}
\end{enumerate}
  Proving Item \ref{i1} consists mostly of showing that the  processes in the sequence $\{G^N\}_{N\in\N}$ do not have too many jumps on a compact rectangle on the $\mu-x$ plane, uniformly in $N$. In Section \ref{sec:SoQ}, we develop the technical tools to tackle this problem. We show that the map $\md$ developed in \cite{FS18} can be obtained by the top lines of sorted random walks. These tools are then used in Section \ref{subsec:sh1} to obtain the necessary intermediate result. For Item \ref{i2}, we essentially show  that under diffusive scaling, the discrete queuing map \eqref{q} converges to the so called Brownian queue \eqref{cq}\cite{o2001brownian}. 
  \subsection{Acknowledgements}
  The author would like to thank M\'arton Bal\'azs for reading early parts of the manuscripts and helpful comments, Patrik Ferrari and Offer Kella for helpful discussions. The author also thanks Timo Sepp\"al\"ainen and Evan Sorensen for helpul comments.  
\section{Some results on queues in stationarity}\label{sec:SoQ}
In this section we develop the tools that will be used later to show that the prelimiting sequence $\{G^N\}$ is tight. The main technical result of this section is Proposition \ref{prop:J}. We conclude  this section with discussing our results in the context of melonizing RWs, which although are not stated as a theorem, underlie the  main ideas in the proofs of Section \ref{sec:bus}.
\subsection{Preliminaries on Queues}\label{subsec:prelim}
Let us recall the basic setup of queuing theory. Let $\qs_+=(\R_+)^{\Z}$. Let $\arrv=\{\arr_i\}_{i\in\Z}$ and $\servv=\{\serv_i\}_{i\in\Z}$ be the sequences in $\qs_+$ denoting the inter-arrival times and service times in a queue respectively. We assume 
\begin{align}\label{qc}
	\lim_{m\rightarrow -\infty}\sum_{i=m}^{0}(\serv_{i-1}-\arr_i)=-\infty.
\end{align}
Given two sequences $\arrv,\servv\in\qs_+$ satisfying \eqref{qc}, the waiting time customer $i$ must wait in the queue before being served is
\begin{align*}
	w_i=\sup_{j\leq i} \Big(\sum_{k=j}^is_{k-1}-a_k\Big)^+,
\end{align*}
and we define the \textit{waiting time} mapping $W:\qs_+^2  \rightarrow \qs_+$
\begin{align}\label{W}
	\{w_i\}_{i\in\Z}=W(\servv,\arrv).
\end{align}
\begin{remark}
	Although the maps used in this paper are not defined on all the prescribed domain (only on those points satisfying \eqref{qc}), we shall abuse notations and write $W:\qs_+^2  \rightarrow \qs_+$.
\end{remark}
  Define the queuing mapping  $D(\cdot,\cdot):\qs_+^2\rightarrow \qs_+$  through
\begin{align}\label{d}
	\{d_i\}_{i\in\Z}&=D(\servv,\arrv)\\
	d_i= (w_{i-1}&+s_{i-1}-a_i)^-+s_i.\label{d1}
\end{align}	
The third map is the so-called \textit{sojourn} mapping $S_{oj}(\cdot,\cdot):\qs_+^2\rightarrow \qs_+$, which stands for the time a customer spends in the queue once arriving it, and is given by
\begin{align}\label{souj}
	\{t_i\}_{i\in\Z}&=S_{oj}(\servv,\arrv)\\
	t_i&= w_{i}+s_i.\nonumber
\end{align} 
The fourth map   $R(\cdot,\cdot):\qs_+^2\rightarrow \qs_+$ is given by
\begin{align}\label{r}
	\{r_i\}_{i\in\Z}&=R(\servv,\arrv)\\
	r_i=a_i\wedge t_{i-1}&=a_i-(w_{i-1}+s_{i-1}-a_i)^-\label{r1}.
\end{align}
 Lastly, define the \textit{idle time} mapping to be $Er(\cdot,\cdot):\qs^2\rightarrow \qs$, accounting for the time the sever is idle, defined through
\begin{align}\label{er}
	\{e_i\}_{i\in\Z}&=Er(\servv,\arrv)\\
	e_i= (w_{i-1}&+s_{i-1}-a_i)^-.\nonumber
\end{align}	
Although the queueing mappings defined in \eqref{W}--\eqref{r} are usually defined in the literature on $\qs_+$, it is not hard to see that they are well defined on the larger space $\qs=\R^\Z$ as long as \eqref{qc} holds. From here on, we shall think of the queuing mappings as maps from $\qs^2$ to $\qs$. Let $\rhov=\{\rho_i\}_{i\in\N} \in \R_+^{\N}$ be a strictly decreasing sequence of positive numbers. Let $\{I^i\}_{i\in\N}$ be a sequence of independent r.v.s in $\qs$ i.e. $I^i=\{I^i_k\}_{k\in\Z}\in\qs$ is a random sequence of numbers. We further assume that $I^i=\{I^i_k\}_{k\in\Z}$ is a sequence of  i.i.d.  r.v.s with marginal distribution $I^i_0\sim\text{Exp}(\rho_i)$. Let $\h{I}=(I^1,I^2,...)\in \qs^\N$, we then write $\h{I}\sim \nu^{\rhov}$, i.e. $\nu^{\rhov}$ denotes the product measure of i.i.d sequences of Exponential random variables with intensities $\rhov$.  
\subsection{Stationary measures for queues in tandem}
In this subsection we introduce the construction of the multi-species stationary distribution of queues introduced in \cite{FS18}.  
Fix $n\in\N$. Let $\rhov^n=(\rho_1,...,\rho_n)\in \R_+^n$. Define the state spaces
\begin{align*}
	\cY^n&:=\Big\{(I^1,...,I^n)\in \qs^n:\lim_{m\rightarrow \infty}m^{-1}\sum_{i=-m}^0 I^k_i<\lim_{m\rightarrow \infty}m^{-1}\sum_{i=-m}^0 I^{k+1}_i\quad \text{for $1\leq k\leq n-1$}\Big\}\\
	%\\
%	\cY&:=\Big\{(I^1,I^2,...,)\in \qs^{\Z_+}:\lim_{m\rightarrow \infty}m^{-1}\sum_{i=-m}^0 I^k_i<\lim_{m\rightarrow \infty}m^{-1}\sum_{i=-m}^0 I^{k+1}_i\quad \text{for $k\in \Z_+$}\Big\}\\
	\cY_{\rhov^n}&:=\Big\{(I^1,...,I^n)\in \qs^n:\lim_{m\rightarrow \infty}m^{-1}\sum_{i=-m}^0 I^k_i=(\rho_k)^{-1} \quad \text{for $1\leq k\leq n$}\Big\}.
\end{align*}
Suppose  $\rhov^n=(\rho_1,...,\rho_n)\in \R_+^n$ is a strictly decreasing vector, i.e. $\rho_1>...>\rho_n$. Let $\rho_1<\rho_0$, and let $I^0\in\qs$ such that
\begin{align*}
	\lim_{m\rightarrow \infty}m^{-1}\sum_{i=-m}^0 I^0_i=(\rho_0)^{-1}.
\end{align*}
We define the mapping $\Phi:\cY_{\rhov^n}\rightarrow \qs$ associated with $I^0$  through
\begin{align*}
	\Phi\big((I^1,...,I^n)\big)=(D(I^0,I^1),...,D(I^0,I^n)) \qquad (I^1,...,I^n)\in \cY_{\rhov^n}.
\end{align*}
Now assume that $I^0$ is an i.i.d. sequence with $I^0_0\sim \text{Exp}(\rho_0)$ and that $\h{I}^n=(I^1,...,I^n)\in \cY_{\rhov^n}$ is random and independent of $I^0$. A natural question in the context of queues in tandem is under what distribution $\mu^{\rhov_n}$ of $\h{I}^n$ the mapping $\Phi$ is stationary, i.e.
\begin{align*}
	\Phi\h{I}^n\sim \mu^{\rhov_n},
\end{align*}
or, abusing notation, in terms of distributions 
\begin{align}\label{sm}
	\Phi\mu^{\rhov_n} = \mu^{\rhov_n}.
\end{align}
This question was answered in \cite{FS18} by Fan and Sepp\"al\"ainen. In order to describe their construction of the measure $\mu^{\rhov_n}$, we introduce some notation. First extend the queueing operator $D$ to $D^{(k)}:\cY^k\rightarrow \qs$ through
\begin{align}\label{q}
	D^{(1)}(I^1)&=I^1\nonumber\\\nonumber
	D^{(2)}(I^1,I^2)&=D(I^1,D^{(1)}(I^2))\\\nonumber
	D^{(3)}(I^1,I^2,I^3)&=D(I^1,D^{(2)}(I^2,I^3))\\\nonumber
	&\vdots\\
\text{ and in general }	D^{(k)}(I^1,...,I^k)&=D(I^1,D^{(k-1)}(I^2,...,I^k)) \qquad k\geq 2 .
\end{align}
Note that for $1\leq k <n$
\begin{align}\label{qi}
	D^{(n)}(I^1,...,I^n)=D^{(k+1)}\Big(I^{1},...,I^{k},D^{(n-k)}(I^{k+1},...,I^{n})\Big).
\end{align}
Next  define the map $\md^{(n)}:\cY^n\rightarrow \cY^n$ given by
\begin{align}\label{md}
	[\md^{(n)}(\h{I})]_i=D^{(i)}(I^1,...,I^{i}) \qquad 1\leq i\leq n,
\end{align}
where we used the notation $[\bm{v}]_i=v_i$ for any vector $\bm{v}$ of size $n$ such that $1\leq i\leq n$.
  The following result from \cite[Theorem 5.5]{FS18} is the foundation upon which we build our results.
  \begin{theorem}\cite{FS18}\label{thm:FS2}
  	There exists a unique measure $\mu^{\rhov_n}$ satisfying \eqref{sm}, and that  is defined as the push-forward distribution with respect to the measure $\nu^{\rhov_n}$ and $\md^{(n)}$, i.e.
  	\begin{align}\label{FS}
  		\mu^{\rhov_n}=\nu^{\rhov_n} \circ(\md^{(n)})^{-1}.
  	\end{align}
  \end{theorem}  
 From here to the end of the section we introduce notation and obtain some auxilary results that will be used in the next subsection. For $i,j\in \N$ and $m\in \N\cup \{\infty\}$ such that $m\geq j\geq i$, define $\Pi_{[i,j]}$ to be the projection onto the i'th to j'th components i.e.
\begin{align*}
	\Pi_{[i,j]}(\h{I})=(I^i,...,I^j), \quad \text{ for } \h{I}\in \qs^m.
\end{align*}
We will abbreviate $\Pi_{i}:=\Pi_{[i,i]}$. 
 Let $\rhov=(\rho_1,\rho_2,...)\in\mathbb{R}_+^{\N}$ be a decreasing vector i.e. $\rho_i>\rho_{i+1}$ for $i\in\N$. Define 
 \begin{align*}
 	\cZ:=\Big\{(I^1,I^2,...)\in \qs^{\N}:\text{ the following limit exists }\lim_{m\rightarrow \infty}m^{-1}\sum_{k=-m}^0 I^i_k \quad \text{for all $i\in\N$}\Big\}.
 \end{align*}
 Define the operator $\denp:\cZ\rightarrow \R^{\N}$ through
\begin{align*}
	[\denp(\h{I})]_i=\Bigg[\lim_{m\rightarrow\infty}m^{-1} \sum_{k=-m}^0I^i_k\Bigg]^{-1},\qquad \h{I}\in\cZ,i\in \N. 
\end{align*}
In words, the operator $\denp$ accepts as input a vector of sequences and returns the vector of their densities. Next define, for each $i\in\N$, the operator $\sigma_i:\cZ\rightarrow \cZ$ through
\begin{align}\label{sigma}
	(\sigma_i \h{I})_k=
	\begin{cases}
	\h{I}_k & \text{$k\notin \{i,i+1\}$ or $[\denp(\h{I})]_i\leq [\denp(\h{I})]_{i+1}$}\\
	D(\h{I}_{i},\h{I}_{i+1})   & \text{$k=i$ and $[\denp(\h{I})]_i>[\denp(\h{I})]_{i+1}$} \\
	R(\h{I}_{i},\h{I}_{i+1}) & \text{$k=i+1$ and $[\denp(\h{I})]_i>[\denp(\h{I})]_{i+1}$}. 
	\end{cases}
\end{align}
To see that $\sigma_i$ is well defined and maps into $\cZ$, note that for $i$ such that $[\denp(\h{I})]_i>[\denp(\h{I})]_{i+1}$
\begin{align*}
	[\denp(\sigma_i\h{I})]_k=
	\begin{cases}
	[\denp(\h{I})]_k & k\notin\{i,i+1\}\\
	[\denp(\h{I})]_{i+1} & k=i\\
	[\denp(\h{I})]_i & k=i+1.
	\end{cases}
\end{align*}
We shall also need the following operator: for $m\in \N\cup \{\infty\}$ and $1\leq i \leq m$, let $\sigma^*_i:\R_+^m\rightarrow \R_+^m$
\begin{align*}
	[\sigma^*_i\rhov]_l=
	\begin{cases}
	\rho_l & \text{$l\notin\{i,i+1\}$ or $\rho_i<\rho_{i+1}$}\\
	\rho_{i+1} & \text{$l=i$ and $\rho_i>\rho_{i+1}$}\\
	\rho_{i} & \text{$l=i+1$ and $\rho_i>\rho_{i+1}$}. 
	\end{cases}
\end{align*}
In words, the operator $\sigma^*_i$ takes a vector of length $m$ and switches places between the $i$'th and $i+1$'th elements iff  they are reversely ordered i.e. $\rho_i>\rho_{i+1}$. Note that for $\h{I}\in \qs^m$ we have
\begin{align}\label{sigmai}
	\denp(\sigma_i \h{I})=\sigma^*_i(\denp(\h{I})).
\end{align}
We say the operations $\sigma_i\h{I}$ or $\sigma^*_i\denp(\h{I})$ are {\it{active}} if $[\denp(\h{I})]_i>[\denp(\h{I})]_{i+1}$, otherwise we say they are {\it{repressed}}. Note that the operation $\sigma_i\h{I}$  ($\sigma^*_i\denp(\h{I})$) is active if and only if $\sigma_i\h{I}\neq \h{I}$ ($\sigma^*_i\denp(\h{I})\neq \denp(\h{I})$). The following result establishes the effect of the operation $\sigma_i$ on the distribution $\nu_{\rhov}$. 
\begin{lemma}\label{stat}%\note{checked}
	Suppose $\h{I}\sim \nu_{\rhov}$. Fix $i\in \N$ and consider $\sigma_i$ defined in  \eqref{sigma}.  Then
\begin{align}\label{di}
\sigma_i \h{I} \sim \nu^{\sigma^*_i\rhov}.
\end{align}
\end{lemma}
\begin{proof}
	From the law of large numbers applied to each element of $\h{I}$,  with probability $1$ $\denp(\h{I})=\rhov$. This means there are two cases to check - a) $[\denp(\h{I})]_i<[\denp(\h{I})]_{i+1}$ a.s.- here nothing happens with probability $1$ and so \eqref{di} holds. (b) $[\denp(\h{I})]_i>[\denp(\h{I})]_{i+1}$ a.s. - here with probability $1$, the application of $\sigma_i$ results in 
	\begin{align*}
		\sigma_i\h{I}=\big(\Pi_{[1,i-1]}(\h{I}),D(I^i,I^{i+1}),R(I^i,I^{i+1}),\Pi_{[i+2,\infty)}(\h{I})\big).
	\end{align*}
	By hypothesis $\h{I}\sim \nu_{\rhov}$, the result now follows from  
	 Lemma \ref{lem:ind}  and  \eqref{sigmai}.
\end{proof}
In what follows, if $\h{I}=(I^1,...,I^n)\in\qs^n$, we denote $D^{(n)}(\h{I})=D^{(n)}(I^1,...,I^n)$. Next, define 
\begin{align}
\sigma_{[i,j]}\h{I}= \sigma_i\sigma_{i+1}\cdots\sigma_j \h{I}\label{sigma2}.
\end{align}
%For $m\geq 3$, $2\leq i\leq m$ and $\h{I}\in \qs^m$ for which the operations below are well defined, define
%\begin{align*}
%\eta(\h{I})=\Pi_1(\h{I}) \quad \xi_i(\h{I})=D^{(i-1)}(\Pi_{[2,i]}\h{I}).
%\end{align*}
 For $c\in\R$, we let $\mathbf{c}\in\qs$ denote the constant sequence. i.e. $\mathbf{c}_i=c$ for all $i\in\Z$. We denote  $\mathbf{c}^k=(\underbrace{\mathbf{c},...,\mathbf{c}}_{\text{k times}})$. The following result will be used later in the paper, it shows that translation of the input of the function $\md^{(n)}$ results in translation of the output.
\begin{lemma} \label{lem:const}%\note{checked}
	For every $n\in\N$ and $c\in\R$
	\begin{align}
		\md^{(n)}(\h{I})-\mathbf{c}^n=\md^{(n)}(\h{I}-\mathbf{c}^n) \label{eq11}.
	\end{align}
\end{lemma}
\begin{proof}
	From the definition \eqref{md}, it is enough to show that 
	\begin{align}\label{eq10}
		D^{(i)}(I^1,...,I^i)-\mathbf{c}=D^{(i)}(I^1-\mathbf{c},...,I^i-\mathbf{c}) \quad 1\leq i \leq n.
	\end{align}
	Clearly \eqref{eq10} holds for $i=1$, the proof for $i\geq2$ will follow by induction on $i$. For the base case $i=2$, it is not hard to see that for all $k\in\Z$
	\begin{align*}
		&[D(I^1,I^2)]_k-c=d_k-c\\
		&=\big(W(I^1-\mathbf{c},I^2-\mathbf{c})_{k-1}+(I^1_{k-1}-c)-(I^2_k-c)\big)^-+I^1_k-c=[D(I^1-\mathbf{c},I^2-\mathbf{c})]_k.
	\end{align*}
	Assume the hypothesis holds up to  $i-1$, using \eqref{q}
	\begin{align*}
		&D^{(i)}(I^1,...,I^i)-\mathbf{c}=D\big(I^1,D^{(i-1)}(I^2,...,I^i)\big)-\mathbf{c}\\ &=D\big(I^1-\mathbf{c},D^{(i-1)}(I^2,...,I^i)-\mathbf{c}\big)=D\big(I^1-\mathbf{c},D^{(i-1)}(I^2-\mathbf{c},...,I^i-\mathbf{c})\big)\\
		&=D^{(i)}(I^1-\mathbf{c},...,I^i-\mathbf{c}).
	\end{align*}
	We have thus proven \eqref{eq10} and so \eqref{eq11}.
\end{proof}
\subsection{The construction $\nc^{(n)}$}\label{subsec:nc}
In this subsection we obtain the main technical tool that underlies the development of much of the main results in this paper. Recall from Theorem \ref{thm:FS2} that in order to obtain the distribution $\mu^{\rhov_n}$, one can simply do the following - start with $\h{I}\sim \nu^{\rhov_n}$ and apply to it the map $\md^{(n)}$. In what comes next, we construct the distribution $\mu^{\rhov_n}$ differently. Although the new construction is arguably more convoluted than $\md^{(n)}$, it  presents an important feature that facilitates much of the results to come. Let us try to explain here the advantage of the new construction $\nc^{(n)}$, which equals $\md^{(n)}$. The map $\nc^{(n)}$ takes as input $\h{I}\sim \nu^{\rhov_n}$ and spits out $(f^1,...,f^n)\in\qs^n$. The construction $\nc^{(n)}$ is such that for each  $2\leq i\leq n$, there exists $u^i\in\qs$ such that $f^i=D(f^{i-1},u^i)$. The upshot, and one of the main results of this section (Proposition \ref{prop:J}), is that the r.v.s (recall \eqref{souj}) $\{S_{oj}(f^{i-1},u^i)_m\}_{2\leq i\leq n}$ are mutually independent, for any $m\in\Z$. This property is important since when the difference between $\rho_{i-1}$ and $\rho_i$ is small, $S_{oj}(f^{i-1},u^i)_m$ gets large which then implies that $f_i=f_{i-1}$ on a fixed compact interval with high probability. We will then use the results in this section to show that, on a fixed compact interval, there are only few $i$'s for which $f^i\neq f^{i+1}$, which is key to show relative compactness in the Skorohord space. We note here that although the map $\nc$ is equivalent to the map constructed in \cite[Section 6.2]{FS18},  its derivation as a sequential sorting of RWs, is new and plays an important role in the proof of Lemma \ref{lem:Fb} (see the last part of Subsection \ref{rl}).  

Let $\rhov=(\rho_1,\rho_2,...)$ be a decreasing sequence and let $\h{I}=(I^1,I^2,...)\sim \nu^{\rhov}$. Set $f^1=I^1$ and for $i=2$
\begin{align*}
	\h{I}^2&=\h{I} \\
	 f^2&=D^{(2)}(\Pi_{[1,2]}\h{I}^2)\\
	u^2=I^2 &\qquad v^2=I^1.
\end{align*}
For $i\geq 3$
\begin{align}
	\h{I}^i&=\sigma_{[1,i-2]}\h{I}^{i-1}\label{Ii}\\
	 f^i&=D^{(i)}(\Pi_{[1,i]}\h{I}^i)\nonumber\\
	u^i=D^{(i-1)}(\Pi_{[2,i]}\h{I}^i) &\qquad v^i=\Pi_{1}\h{I}^i.\label{xieta}
\end{align}
Note that from \eqref{qi} with $k=1$ it follows that
\begin{align}\label{id2}
	f^i=D(v^i,u^i).
\end{align}
If $\rhov_n=(\rho_1,...,\rho_n)$ and $\h{I}^{(n)}\sim\nu^{\rhov_n}$, we define the map $\nc^{(n)}:\qs^{n}\rightarrow \qs^n$ by
\begin{align}\label{nc}
	\nc^{(n)}(\h{I}^{(n)})=(f^1,...,f^n).
\end{align}
The following is a simple yet important observation.
\begin{lemma}\label{lem:ac}%\note{checked}
	Fix $i\geq 4$ and  define $X^k=\sigma_{[k+1,i-2]}\h{I}^{i-1}$ for any $0\leq k\leq i-3$ and $X^{i-2}=\h{I}^{i-1}$. Then for every $0\leq k\leq i-2$ the operation $\sigma_k X^k$ is active.
\end{lemma}
\begin{proof}
	First note that from \eqref{Ii}, $\h{I}^{i-1}$ is obtained from $\h{I}$ by applying $\sigma$ operations affecting only the first to $i-2$  elements. As $\Pi_{[1,n]}\h{I}\in\cY^n$ for every $n\in\N$, it follows that
	\begin{align*}
		[\denp(\h{I}^{i-1})]_j>[\denp(\h{I}^{i-1})]_{i-1} \quad \text{ for $1 \leq j<i-1$},
	\end{align*}
	and therefore that
	\begin{align}\label{bc}
		 \text{$\sigma_{i-2}\h{I}^{i-1}$ is active and  $[\denp(\sigma_{i-2}\h{I}^{i-1})]_j>[\denp(\sigma_{i-2}\h{I}^{i-1})]_{i-2} \quad \text{ for $1 \leq j<i-2$}$}.
	\end{align}
	Continuing by induction on $k$ with \eqref{bc} as our base case we conclude the result.
\end{proof}
The next lemma shows that $\nc^{(n)}$ is indeed a representation of $\md^{(n)}$.
\begin{lemma}%\note{checked}
	Fix $n\in\N$. Then
\begin{align*}
	\md^{(n)}(\Pi_{[1,n]}\h{I})=\nc^{(n)}(\Pi_{[1,n]}\h{I})
\end{align*}
\end{lemma}
\begin{proof}
	For $i\in [n]$ we must show that 
	\begin{align*}
		[\cD^{(n)}(\Pi_{[1,n]}\h{I})]_i=f^i,
	\end{align*}
	where $f^i$ was defined in \eqref{Ii}. By definition 
	\begin{align*}
		[\cD^{(n)}(\Pi_{[1,n]}\h{I})]_i=D^{(i)}(\Pi_{[1,i]}\h{I}) \quad 1\leq i \leq n.
	\end{align*}
	Note that $\h{I}^i$ is obtained from $\h{I}$ by a sequence of $\sigma$ operations affecting only the first to $i-1$'th elements of the sequence $\h{I}$ i.e. 
	\begin{align*}
		\h{I}^i_j=\h{I}_j \qquad \forall i\leq j.
	\end{align*}
	we see that 
	\begin{align*}
		f^i=D^{(i)}(\Pi_{[1,i]}\h{I}^i)\stackrel{\eqref{si3}}{=}D^{(i)}(\Pi_{[1,i]}\h{I})\stackrel{\eqref{md}}{=}[\cD^{(n)}(\Pi_{[1,n]}\h{I})]_i.
	\end{align*}
\end{proof}
Recall $v^i$ and $u^i$ from \eqref{xieta}. The following identities will come useful in the sequel. 
\begin{lemma} %\note{checked}
For $i\geq 3$
\begin{align}
	v^{i-1}=\Pi_1\Big(\sigma_{[2,i-2]}\h{I}^{i-1}\Big)\label{etar}\\
	u^{i-1}=\Pi_2\Big(\sigma_{[2,i-2]}\h{I}^{i-1}\Big)\label{xir}
\end{align}
and 
\begin{align}
	D(v^{i-1},u^{i-1})=\Pi_1\Big(\h{I}^{i}\Big)\label{dr}\\
	R(v^{i-1},u^{i-1})=\Pi_2\Big(\h{I}^{i}\Big)\label{rr}.
\end{align}
Moreover, 
\begin{align}\label{vf}
		v^i=D(v^{i-1},u^{i-1})=f^{i-1} \quad \text{ for $i\geq 2$}.
\end{align}
and 
\begin{align}\label{ftf}
	f^i=D(f^{i-1},u^{i}).
\end{align}
\end{lemma}
\begin{proof}
	We begin with showing \eqref{etar}, which is equivalent, from the definition of $v^{i}$ in \eqref{xieta}, to showing that 
	\begin{align}\label{eqetar}
		\Pi_{1}\h{I}^{i-1}=\Pi_1\Big(\sigma_{[2,i-2]}\h{I}^{i-1}\Big).
	\end{align}
	As the operation $\sigma_{[2,i-2]}$ does not affect the first element of $\h{I}^{i-1}$, \eqref{eqetar} follows. Next we show \eqref{xir}. Define
	\begin{align*}
		h_k=\Pi_k\sigma_{[k,i-2]}\h{I}^{i-1} \qquad k\in \{2,...,i-2\}.
	\end{align*}
	Using Lemma \ref{lem:ac}, the following recursive relation holds for  $1\leq k \leq i-3$
	\begin{align}\label{rec1}
		h_k&=\Pi_k\sigma_{[k,i-2]}\h{I}^{i-1}=\Pi_k\sigma_k\sigma_{[k+1,i-2]}\h{I}^{i-1}\\\nonumber
		&=\Pi_1\big(D(\h{I}^{i-1}_k,\Pi_{k+1}\sigma_{[k+1,i-2]}\h{I}^{i-1}),R(\h{I}^{i-1}_k,\Pi_{k+1}\sigma_{[k+1,i-2]}\h{I}^{i-1})\big)\\
		&=D(\h{I}^{i-1}_k,\Pi_{k+1}\sigma_{[k+1,i-2]}\h{I}^{i-1})=D(\h{I}^{i-1}_k,h_{k+1})\nonumber
	\end{align}
	with the boundary condition
	\begin{align}\label{rec2}
		h_{i-2}=D\big(\h{I}^{i-1}_{i-2},\h{I}^{i-1}_{i-1}\big).
	\end{align}
	It is not hard to see that from \eqref{rec1}--\eqref{rec2} we have
	\begin{align*}
		\Pi_2\big(\sigma_{[2,i-2]}\h{I}^{i-1}\big)=h_2=D^{(i-2)}\big(\h{I}^{i-1}_{2},...,\h{I}^{i-1}_{i-1}\big)=D^{(i-2)}(\Pi_{[2,i-1]}\h{I}^{i-1})=u^{i-1},
	\end{align*}
	thus proving \eqref{xir}. Finally we prove \eqref{dr}--\eqref{rr}. 
	\begin{align*}
		\Pi_1\big(\h{I}^i\big)&\stackrel{\eqref{Ii}}{=}\Pi_1\big(\sigma_1 \sigma_{[2,i-2]}\h{I}^{i-1}\big)=D\Big(\Pi_1\Big(\sigma_{[2,i-2]}\h{I}^{i-1}\Big),\Pi_2\Big(\sigma_{[2,i-2]}\h{I}^{i-1}\Big)\Big)=D\big(v^{i-1},u^{i-1}\big)\\
		\Pi_2\big(\h{I}^i\big)&\stackrel{\eqref{Ii}}{=}\Pi_2\big(\sigma_1 \sigma_{[2,i-2]}\h{I}^{i-1}\big)=R\Big(\Pi_1\Big(\sigma_{[2,i-2]}\h{I}^{i-1}\Big),\Pi_2\Big(\sigma_{[2,i-2]}\h{I}^{i-1}\Big)\Big)=R\big(v^{i-1},u^{i-1}\big),
	\end{align*}
	where  in both lines, the second equality follows from the definition of $\sigma_i$, while the third equality follows from \eqref{etar}--\eqref{xir}. From \eqref{xieta}, \eqref{id2} and \eqref{dr} we conclude \eqref{vf}. Finally, \eqref{ftf} follows by plugging \eqref{vf} in \eqref{id2}.
\end{proof}

Let $\rhov\in \R_+^{\N}$ be a decreasing vector. In what comes next we use the construction above with $\h{I} \sim \nu^{\rhov}$. Recall $u^i$ and $v^i$ from \eqref{xieta} and define $J^i=S_{oj}(v^i,u^i)\in \qs$ for $i\geq 2$. 
\begin{lemma}\label{lem:indJ}%\note{checked}
	Fix $i\geq 2$ and $x\in \Z$. Then $\{J^{j}_x\}_{j\geq i+1}$ is independent of $J^i_x=S_{oj}(v^{i},u^{i})_x$.
\end{lemma}
\begin{proof}
	We first claim that it is enough to show that 
	\begin{align}\label{ac}
		\{[\h{I}^j_z]_l\}_{j\geq i+1,l\leq x,z\in\N} \text{ is independent of } J^i_x.
	\end{align}
	 To see why this is true, first note that from \eqref{souj}, for every $j\geq 2$ and $x\in\Z$, $J^j_x$ is a function of $\{u^{j}_l,v^{j}_l\}_{l\leq x}$. Moreover, for $j\geq i+1$, by \eqref{xieta} $\{u^{j}_l,v^{j}_l\}_{l\leq x}$ is a function of $\{[\h{I}^j_z]_l\}_{j\geq i+1,l\leq x,z\in\N}$  and so \eqref{ac}  implies the result. Next we claim that \eqref{ac} can be further weakened to showing that
	 \begin{align}\label{ac2}
	 	\{[\h{I}^{i+1}_z]_l\}_{l\leq x,z\in\N} \text{ is independent of } J^i_x.
	 \end{align}
	 To see why \eqref{ac2} implies \eqref{ac}, first note that the random variables $\{[\h{I}^j_z]_l\}_{j\geq i+1,l\leq x,z\in\Z}$ can be obtained from $\h{I}$ by applying the $\sigma$ operations. Next observe that the operator $\sigma$ does not 'look into the future' i.e. for every $x\in \Z$
	 \begin{align*}
	 	\Big\{\Pi_1\big(\sigma (I^1,I^2)\big)_i,\Pi_2\big(\sigma (I^1,I^2)\big)_i\Big \}_{i\leq x}=\Big\{D(I^1,I^2)_i,R(I^1,I^2)_i\Big\}_{i\leq x} \\
	 	\text{ is measureable with respect to the sigma algebra generated by } \{I^1_i,I^2_i\}_{i\leq x}.
	 \end{align*}
	  This implies that $\{[\h{I}^{j}_z]_l\}_{j\geq i+1,l\leq x,z\in\Z}$  is measurable with respect to the sigma-algebra generated by $\{[\h{I}^{i+1}_z]_l\}_{l\leq x,z\in\Z}$ which shows that \eqref{ac2} implies \eqref{ac}. 
	  
	  We now move on to show \eqref{ac2}. Let $\h{I}'=\sigma_{[2,i-1]}\h{I}^i$. From \eqref{Ii} $\h{I}^{i+1}=\sigma_1\sigma_{[2,i-1]}\h{I}^i=\sigma_1\h{I}'$.  From Lemma \ref{stat} we see that
	  \begin{align*}
	  	\Pi_{[3,\infty]}\h{I}' \text{ is independent of } \Pi_{[1,2]}\h{I}'.
	  \end{align*}
	  From \eqref{etar}--\eqref{xir}
	  \begin{align}\label{id}
	  	\Pi_{[1,2]}\h{I}'=\big(v^i,u^i\big).
	  \end{align}
	  It follows that 
	  \begin{align}\label{ind3}
	  	\Pi_{[3,\infty]}\h{I}' \text{ is independent of } \big(D(v^i,u^i),R(v^i,u^i),S_{oj}(v^i,u^i)\big).
	  \end{align}
	  From \eqref{ind3} and Lemma \ref{lem:ind} \ref{ind1}
	  \begin{align}\label{ind}
	  	\{[D(v^i,u^i)]_j\}_{j\leq x},\{[R(v^i,u^i)]_j\}_{j\leq x},\{\Pi_{[3,\infty]}\h{I}'\}_{j\leq x} \,\,\text{ are independent of }\,\, J^i_x=S_{oj}(v^i,u^i)_x.
	  \end{align}
	  It follows that
	  \begin{align}\label{Ip}
	  	\h{I}^{i+1}&=\sigma_1 \h{I}'=\sigma_1 \Big(\Pi_{[1,2]}\h{I}',\Pi_{[3,\infty]}\h{I}'\Big)= \Big(\sigma_1\Pi_{[1,2]}\h{I}',\Pi_{[3,\infty]}\h{I}'\Big)\\
	  	&\stackrel{\eqref{id}}{=} \Big(\sigma_1\big(v^i,u^i\big),\Pi_{[3,\infty]}\h{I}'\Big)=\Big(D\big(v^i,u^i\big),R\big(v^i,u^i\big),\Pi_{[3,\infty]}\h{I}'\Big)\nonumber.
	  \end{align}
	  From \eqref{ind} and \eqref{Ip} follows \eqref{ac2}. The proof is now complete.
\end{proof}
The following result is the main feature of the construction $\nc^{n}$.
\begin{corollary}\label{cor:in}%\note{checked}
	Fix  $x\in \Z$. Then $\{J^{j}_x\}_{j\geq 1}$ are independent.
\end{corollary}
\begin{proof}
	This follows by induction and Lemma \ref{lem:indJ}.
\end{proof}
Next we would like to determine the distribution of $J^i_x$. In order to do that, we first have to determine the distributions of $v^i$ and $u^i$.
\begin{lemma}\label{lem:int}%\note{checked}
	Fix $i\in \N$. Then $v^i$ and $u^i$ are independent i.i.d. sequences with the following marginals
	\begin{align}
		&v^i_0\sim \mathrm{Exp}(\rho_{i-1})\label{etad}\\
		&u^i_0\sim \mathrm{Exp}(\rho_{i})\label{xid}.
	\end{align}
\end{lemma}
\begin{proof}
	By Lemma \ref{stat} and \eqref{etar}--\eqref{xir}, both $v^i$ and $u^i$ are independent sequences of Exponential random variables. It is left to verify the intensities in \eqref{etad}--\eqref{xid}. Note that from \eqref{Ii}, the $\sigma$ operators affect only the first to $i-1$ elements of $\h{I}$ to obtain $\h{I}^i$. This implies that 
	\begin{align*}
		\Pi_i\h{I}^i=[\h{I}]_i.
	\end{align*}
	and therefore that
	\begin{align*}
		[\denp(\h{I}^i)]_k>[\denp(\h{I}^i)]_i \quad 1 \leq k < i.
	\end{align*}
	It follows that
	\begin{align*}
		[\denp(v^i)]_1\stackrel{\eqref{xieta}}{=}[\denp(\h{I}^i)]_1=[\denp(\sigma_{[1,i-2]}\h{I}^{i-1})]_1\stackrel{\eqref{sigmai}}{=}[\sigma^*_{[1,i-2]}(\denp(\h{I}^{i-1}))]_1=[\denp(\h{I}^{i-1})]_{i-1}=\rho_{i-1},
	\end{align*}
	concluding \eqref{etad}. Finally, \eqref{xid} follows from 
	\begin{align*}
		\denp(u^i)\stackrel{\eqref{xieta}}{=}\denp(D^{(i-1)}(\Pi_{[2,i]}\h{I}^i))=\denp([\h{I}]_i)=\rho_i.
	\end{align*}
\end{proof} 
The following proposition is the main result of this section.
\begin{proposition}\label{prop:J}
	Let $\rhov\in \R_+^{\N}$ be a decreasing vector. Assume the construction above and fix $x\in\Z$. The random variables $\{J^i_x\}_{i\in\N}$ are independent and 
	\begin{align}\label{J}
		J^i_x\sim \mathrm{Exp}(\rho_{i-1}-\rho_i).
	\end{align}
\end{proposition}
\begin{proof}
	Independence of $\{J^i_x\}_{i\in\N}$ is a consequence of Corollary \ref{cor:in}. \eqref{J} follows from Lemma \ref{lem:sourj} and Lemma \ref{lem:int}.
\end{proof}

 \begin{figure}[t]
	\centering
	
	\includegraphics[scale=1]{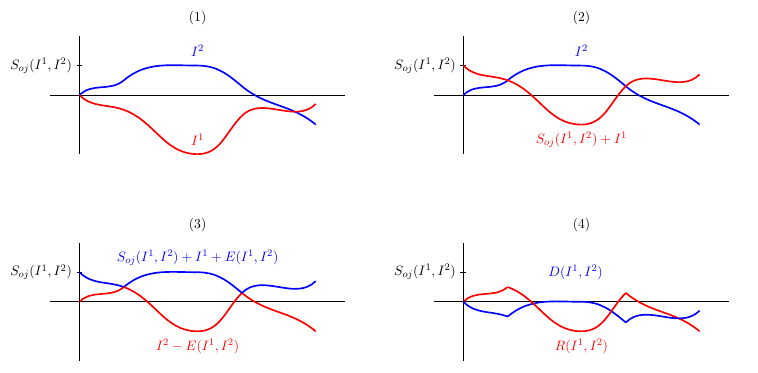}
	\caption{The transformation $(I^1,I^2)\mapsto \big(R(I^1,I^2),D(I^1,I^2)\big)$ restricted to $\N$ as melonization of two random walks. To simplify notation we annotate the curves by their sequence of differences rather than their true value i.e. we write $I^1$ instead of $S^{0,\cdot}(I^1)$. (1) The input of the melonization is two infinite sequences $I^1$ and $I^2$ from which we obtain $S_{oj}(I^1,I^2)$. (2) The curve $I^1$ is raised  by $S_{oj}(I^1,I^2)$. (3) We construct an upper (blue) and lower (red) curves (4) We lower the upper curve by $S_{oj}(I^1,I^2)$ to obtain $D(I^1,I^2)$ (blue). The algorithm holds up to error of $O(1)$ which disappears upon scaling.}
	\label{fig:mel}
\end{figure}    
  \begin{figure}[t]
	\centering
	
	\includegraphics[scale=0.9]{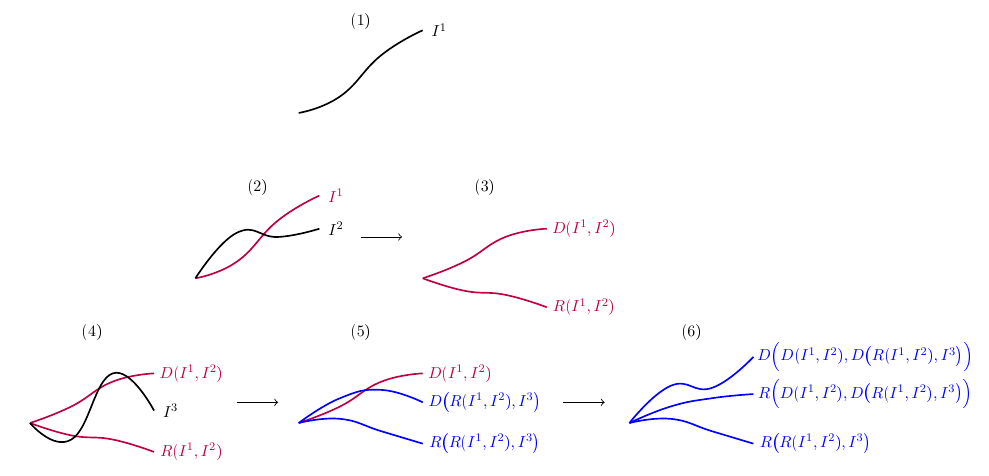}
	\caption{Illustration of the construction of $\nc$ as melonization of random walks. From the top, at each of the three lines of the figure, we construct and element of the vector
		 $\nc(I^1,I^2,I^3)=\Big(I^2,D(I^1,I^2),D\big(D(I^1,I^2),D(R(I^1,I^2),I^3)\big)\Big)$. At each line, the black curve represents a new RW that is independent of all other RW introduced in earlier stages of the algorithm. Curves with the same colour are ordered. In (1) we introduce the first RW, $I^1$, which is also the first element of $\nc$. In (2)-(3) we construct the second element of the $\nc$ as follows. We start with the output of (1), to which we add the RW $I^2$. We then order the two RWs to obtain (3). The top line, $D(I^1,I^2)$ is the second element of $\nc$. In (4)--(6) we construct the third element of $\nc$ in a similar fashion. Use the line ensemble from $(3)$ as an input and add $I^3$ to it. We then start sorting the curves from bottom to top, $I^3$ is first sorted with the bottom line in (5) yielding the two blue curves, the top line of which, is then sorted with respect to the red line to give the ordered ensemble in (6). The top line of this ensemble is the third element of $\nc$.}
	\label{fig:mel2}
\end{figure}    
\subsection{The operator $\nc$ as melonization of random walks}\label{subsec:mel}
	In this subsection we discuss in more detail the connection of the map $\nc$ from Subsection \ref{subsec:nc} to the non-intersecting-RWs-RSK correspondence of O'Connell and Yor from \cite{OcY02}. 
	\subsubsection*{Two types of melonization}
	The connection between queues and LPP goes back to the works of Szczotka and Kelly \cite{SK90}, and Glynn and Whitt \cite{GW91}. In \cite{OcY02}, O'Connel and Yor showed that a sequence of independent Brownian motions (Bm) conditioned not to intersect can be represented as a sequence of queues in tandem. Let us now sketch their idea. Let  $\h{X}=(X^1,X^2,...,X^N)$ be $N\geq 1$ independent Bms on $[0,\infty)$. Let
	\begin{align*}
		E(f,g)(t)=\max_{s\in [0,t]}g(s)-f(s),
	\end{align*}
	and define
	\begin{align}
		\alpha^U(f,g)(t)&=f(t)+E(f,g)(t)\label{m}\\
		\alpha^D(f,g)(t)&=g(t)-E(f,g)(t)\label{m1}.
	\end{align}
	We then define 
	\begin{align*}
		\alpha_i(\h{X})=\big(X_1,...,\alpha^U(X_i,X_{i+1}),\alpha^D(X_i,X_{i+1}),...,X_N\big).
	\end{align*}
	Next we define
	\begin{align*}
		\hat{\alpha}_i=\alpha_1\cdots\alpha_{i-1}\alpha_i  \quad 1\leq i\leq N-1.
	\end{align*}
Finally let
\begin{align*}
	\cW(\h{X})=\hat{\alpha}_{N-1}\hat{\alpha}_{N-2}\cdots\hat{\alpha}_{1}\h{X}.
\end{align*}
It then holds that $\cW$ has the same distribution as that of $\h{X}$ conditioned on $\{X_1(t)\geq X_2(t)\geq...\geq X_N(t)\}\quad \forall t\geq 0$. Confirming to the literature, we refer to $\cW$ as the \textit{packed melon} of $\h{X}$, where we added the word packed to distinguish it from a similar model we shall shortly discuss. We shall often refer to the algorithm behind the melon as \textit{melonization} or \textit{packed melonization}. From this, and the connection between LPP and queueing theory, follows the connection between the extremal values of LPP and the eigenvalues of the GUE random matrices \cite{Bar01,GTW01}. Indeed, by Greene's Theorem and the results of O'Connell in \cite{Occ03}, the sum of the top $k$ values in $\cW(\h{X})$ at time $t=1$ equals the maximal passage time of $k$ non-intersecting paths on $\h{X}$, starting from  $X^N$ at time $0$ and terminating at time $1$ on $X^1$. These paths, looking like the stripes of a watermelon, are where the name melon comes from. This picture is expected to hold more generally, as was shown by Biane, Bougerol and O'Connell in \cite{Occ03,BBO05} where the connection between non-intersecting RWs and the RSK algorithm was studied, and in fact recently used by Dauvergne and Virag in \cite{DV21} to prove the convergence of various models to the Directed Landscape.

Recall the sequences $\arrv,\servv,\bm{d}$ and $\bm{r}$ defined in Subsection \ref{subsec:prelim}, and let us now try to see how our description of the map $\nc$ is related to the melon of RWs. Define
\begin{align*}
	A_n=\sum_{i=1}^{n}a_i, \quad 
	S_n=\sum_{i=1}^{n}s_i \quad   D_n=\sum_{i=1}^{n}d_i \quad \text{and} \quad R_n=\sum_{i=1}^{n}r_i \quad \text{for $n\geq1$}.
\end{align*}
 Recall the sequence $\{e_i\}_{i\in \Z}=Er(\servv,\arrv)$ from \eqref{er}. From \eqref{se}
 \begin{align*}
 	E_n:=&\sum_{i=1}^{n}e_i=\sum_{i=1}^{n}(w_{i-1}+s_{i-1}-a_i)^-=\Big(\inf_{1\leq i \leq n-1}J_{0}+S_{i-1}-A_{i-1}-a_i\Big)^-\\
 	=&\Big(\sup_{1\leq i \leq n-1}A_{i-1}-S_{i-1}+a_i-s_0-w_0\Big)^+
 \end{align*}
From \eqref{d1} and \eqref{r1} we conclude that
\begin{align}
	D_n&=S_n+E_n\\
	R_n&=A_n-E_n.
\end{align}
Now assume there exists a scaling $\stackrel{N}{\rightarrow}$ such that 
\begin{align*}
	A \stackrel{N}{\rightarrow} f_A, \quad	S \stackrel{N}{\rightarrow} f_S, \quad	E \stackrel{N}{\rightarrow} f_E, \quad	w_0 \stackrel{N}{\rightarrow} W, \quad 	D \stackrel{N}{\rightarrow} f_D, \quad \text{and } 	R \stackrel{N}{\rightarrow} f_R \quad \text{ as $N\rightarrow \infty$},
\end{align*}
where $f_A,f_S,f_E,f_D,f_R\in C[0,\infty)$ vanish at $0$ and $W\in\R_+$. Assume further that under $\stackrel{N}{\rightarrow}$ the contribution of $|a_i|$ and $|s_i|$ is negligible. Then, putting aside some (serious) technical issues, it follows that
\begin{align}
	f_D&=f_S+E(f_S+W_0,f_A)^+\label{mm}\\
	f_R&=f_A-E(f_S+W_0,f_A)^+\label{mm2}.
\end{align}
Consider Figure \ref{fig:mel} for an illustration of the sorting mechanism for 2 functions. Comparing  \eqref{m}--\eqref{m1} with \eqref{mm} --\eqref{mm2}, we see that the difference between the 2-functions melonization $(f,g)\mapsto (\alpha_1,\alpha_2)$,  used in \cite{OcY02}, and $(f_S,f_A)\mapsto (f_D,f_R)$ -- the one we are using in this paper, comes from the waiting time $W_0$. Alternatively, the difference in the two melonizations comes from the fact that the sequences $\servv$ and $\arrv$ are bi-infinite rather than just infinite in one direction. To distinguish the $(f_S,f_A)\mapsto (f_D,f_R)$ defined through \eqref{mm} --\eqref{mm2} from the packed melon, we shall refer to the former as the stationary melon(ization).
\subsubsection*{From melonziation to $\nc$}
Now that we have some intuition to the basic stationary melonization operation $(\servv,\arrv)\mapsto \big(D(\servv,\arrv),R(\servv,\arrv)\big)$ as a geometric picture of sorting two functions, let us explain how we obtain the mapping $\nc$ from a sequence of such operations. Let $\h{J}\sim \nu_{\rhov}$ for some decreasing vector $\bm{\rho}=(\rho_1,...,\rho_k)$. Recall the operator $\sigma_i$ from \eqref{sigma} and $\sigma_{[i,j]}$ from \eqref{sigma2} and define 
\begin{align}\label{Wm}
	\cW^{(1)}(\h{J})&=J^1\\
	\cW^{(i)}(\h{J})&=\sigma_{[1,i-1]}\Big(\cW^{(i-1)}(\h{J}),J^i\Big) \quad \text{for $2\leq i\leq k$}.
\end{align}
Note that $\cW^{(i)}$ is a vector of size $i$. In words, we receive the vector $\h{J}$ as input and we sort its elements in the following way; having sorted the elements  $J^1$ through $J^{i-1}$ to obtain $\cW^{(i-1)}(\h{J})$, we sort the new (and independent) element $J^i$ from the bottom ($\cW^{(i-1)}_{i-1}$) up to the top line ($\cW^{{(i-1)}}_{1}$) to obtain $\cW^{i}$ (see Figure \ref{fig:mel2}).  Finally we define
\begin{align*}
	\cW^{\text{top}}(\h{J})=(\cW^{(1)}_1,...,\cW^{(k)}_1),
\end{align*}
the vector registering the top most line in each of the iterations $\{\cW^{(i)}\}_{i=1}^k$. We now claim that 
\begin{align}\label{eqv}
	\nc^{(k)}(\h{J})=\cW^{\text{top}}(\h{J}).
\end{align}
To see that, recall the infinite vector of sequences $\h{I}^i$  from \eqref{Ii}, and note that
\begin{align*}
	\Pi_{[1,i-1]}\h{I}^i=\cW^{(i-1)}(\Pi_{[1,k]}\h{I}) \quad i\geq 2,
\end{align*}
and that
\begin{align*}
	f^i=D^{(i)}\big(\Pi_{[1,i]}\h{I}^i\big)=D^{(i)}\big(\Pi_{[1,i-1]}\h{I}^i,I^i\big)=D^{(i)}\big(\cW^{(i-1)}(\Pi_{[1,i-1]}\h{I}),I^i\big).
\end{align*}
Setting the $k$ first elements of $\h{I}$ to be $\h{J}$
\begin{align}\label{fde}
		f^i=D^{(i)}\big(\cW^{(i-1)}(\h{J}),J^i\big) \quad 1\leq i\leq k.
\end{align}
Next note that the operator $D^{(i)}$ applies the operator $D$ successively, hence sorting the sequence $J^i$ through $\cW^{(i-1)}(\h{J})$ from bottom to top. We have thus showed that
\begin{align}\label{DW}
	D^{(i)}\big(\cW^{(i-1)}(\h{J}),J^i\big)=\cW^{(i)}_1 \quad \text{for $1\leq i\leq k$}.
\end{align}
Plugging \eqref{DW} in \eqref{fde} and using \eqref{nc} implies \eqref{eqv}.
\subsubsection*{The effect of large $W_0$}\label{subsec:W}
From Burke's Theorem \cite{Bur56}, it follows that no matter how large  the size  the vector $\h{J}$ ($k$) is, $\cW^{(k)}_1$ cannot be stochastically larger than $J^k$. This is in contrast to the packed melon where $\cW_1 \sim \sqrt{k}$, i.e. the top line of the packed melon grows like square root of the number of lines in the melon. The fact that the system does not "blow up" is of course expected from a stationary system. What is not so straightforward is the fact that the system is kept "in check" by sticking successive lines to one another. This phenomena can be attributed to the role of $W_0$ in the stationary melonization. If the sequence $J^k$ is kept fixed while the number of sequences ($k$) goes to infinity, then from condition \eqref{qc} it follows that the different elements in $\h{J}$ must be close, which one could then expect would lead to large waiting times in the stationary melonziation. A large waiting time $W_0$ in the operation $\sigma_i$ for some $1\leq i \leq k$ can be thought of as a potential well that prevents the $\sigma_i$ from being carried out. Indeed, setting $W_0$ large enough in \eqref{mm}--\eqref{mm2} would imply that $f_D=f_S$ and $f_R=f_A$, and therefore that $\sigma_i(f_S,f_A)=(f_D,f_R)=(f_S,f_A)$. Rendering many of the $\sigma$ operations in \eqref{W} inactive results in many of the elements of $\cW^{\text{top}}(\h{J})$ (and therefore the elements of $\nc^{(k)}(\h{J})$) being equal. This stickiness feature of the Busemann process has been observed in \cite{FS18} where the authors showed that when varying the density $\rho\in(0,1)$, $\rho \mapsto B^{\rho}_{0,e_1}$ is a Poisson process, a phenomena that can be attributed to the coalescence of geodesics.  Our results show that this feature of the Busemann process carries over to its behaviour around a prescribed density when scaled diffusively.  
\subsubsection*{Relevant literature}\label{rl}
Our representation of $\nc^{(n)}$, developed in Subsection \ref{subsec:nc}, is not new and is equivalent to that in \cite[Section 6.2]{FS18} (we thank Timo Sep\"al\"ainen for pointing this out).  However, the derivation of the representation in \cite{FS18} is different and constructed using two triangular arrays rather than one. In particular, the connection to the packed melon did not appear in \cite{FS18} and seems to be new. This observation plays a crucial role in one of the key results in this paper -  Lemma \ref{lem:Fb}, by controlling the top line of the stationary melon by that of the packed one (see the discussion therein). Moreover,  although we did not pursue such a result, we believe this new picture of the map $\md^{(k)}$ as sorting random walks can be used to improve \eqref{ub82} significantly (although not optimally).

The fact that large $W_0$ results in constant values of the Busemann process around a fixed density $\rho$ was first used in \cite{BBS20} to obtain strong local stationary results of the point to point geodesics. In \cite{BF20}, the authors used  similar ideas to obtain the exact exponent for the probability that the point to point geodesic does not coalesce with the stationary one. In this sense, this work is a natural progression as we consider the constant times of the Busemann process across a continuum of densities around a fixed density. 
\section{The Busemann Process}\label{sec:bus}
	In this section we recall the definition and main properties of the Busesmann process, and obtain the necessary results to be used later in the paper. The Busemann Process is a random process parametrized by $0<\rho<1$ and taking values in the space of real valued functions defined on $\Z^2\times \Z^2$. More specifically, for a fixed $0<\rho<1$, the process $B^{\cdot}$ attains the value $B^{\rho}\in \R^{\Z^2 \times \Z^2}$ which is called the Busemann function (associated with the density $\rho$). This function has proven to be an important tool in studying infinite geodesics in LPP on the lattice (where the ideas extend to positive temperature polymers on the lattice). For example, $B^\rho$ allows one to explicitly construct an infinite geodesic starting from any point on the lattice in the direction associated with the density $\rho$. 
	
	This section is divided into two parts, in Section \ref{subsec:Def} we provide the definition of the Busemann process and give its main properties.  Section \ref{subsec:sh1}  shows that under the diffusive scaling, the Busemman process does not take to many values on a compact interval. This will later imply the modulus of continuity condition (Condition \ref{fdc} of Theorem \ref{thm:tc}) for the prelimiting sequence $G^N$.
	\subsection{Definition and main properties}\label{subsec:Def}
	Below we recall the main properties of the Busemann process.
\begin{theorem}\cite[Theorem 4.2]{S18}
	For each $0<\rho<1$ there exists a function $B^{\rho}_{\cdot,\cdot}(\omega):\Z^2\times\Z^2 \rightarrow \R$ and a family  of random weights $\{\hat{\omega}_x\}_{x\in\Z^2}$ such that the following properties hold:
	\begin{enumerate}
		\item for any $x,y,z\in\Z^2$ 
		\begin{align}\label{cocyc}
			B^{\rho}_{x,y}+B^{\rho}_{y,z}=B^{\rho}_{x,z}.
		\end{align}
		\item $B^{\rho}$ is stationary with marginal distribution
		\begin{align*}
			B_{x,x+e_1}\sim \mathrm{Exp}(\rho) \quad B_{x,x+e_2}\sim \mathrm{Exp}(1-\rho)
		\end{align*}
		\item For any down-right path $\mathcal{Y}=(y_k)_{k\in\Z}$ in $\Z^2$, the random variables 
	\begin{align}
		\{\hat{\omega}_z\}_{z\in \mathcal{G}_-}, \quad \{t(y_{k-1},y_k):k\in \Z\} \quad \text{and} \quad \{\omega_z\}_{z\in\mathcal{G}_+}
	\end{align}
are mutually independent, where
\begin{align*}
	t(e)=
	\begin{cases}
		B^\rho_{x-e_1,x} & e=(x-e_1,x)\\
		B^\rho_{x-e_2,x} & e=(x-e_2,x)
	\end{cases}
\end{align*}
and where 
\begin{align*}
	\mathcal{G}_+=\cY+\Z^2_{>0}, \quad \text{and}	\quad \mathcal{G}_-=\cY+\Z^2_{<0}.
\end{align*}
		\item the following identities hold 
		\begin{align*}
			\hat{\omega}_{x-e_1-e_2}&=B^\rho_{x-e_1-e_2,x-e_2}\wedge B^\rho_{x-e_1-e_2,x-e_1}\\
			B^\rho_{x-e_1,x}&=\omega_x+(B^\rho_{x-e_1-e_2,x-e_2}-B^\rho_{x-e_1-e_2,x-e_1})^-\\
			B^\rho_{x-e_2,x}&=\omega_x+(B^\rho_{x-e_1-e_2,x-e_1}-B^\rho_{x-e_1-e_2,x-e_2})^- 
		\end{align*}
		\item \label{rc}There exists a single event $\Omega_2$ of full probability such that for every $\omega\in\Omega$, all $x\in\Z^2$ and all $\rho_1<\rho_2$
		\begin{align}\label{mon}
			B^{\rho_2}_{x,x+e_1}(\omega)\leq B^{\rho_1}_{x,x+e_1}(\omega), \quad B^{\rho_1}_{x,x+e_2}(\omega)\leq B^{\rho_2}_{x,x+e_2}(\omega),
		\end{align}  
	and for each $x,y\in\Z^2$, the function $\rho\rightarrow B^{\rho}_{x,y}(\omega)$ is right-continuous with left limits.
	\item For each fixed $0<\rho<1$ there exists $\Omega_2^{(\rho)}$ of full probability such that for every $\omega\in \Omega_2^{(\rho)}$ and any sequence $v_n\in \Z^2$ such that $|v_n|_1\rightarrow\infty$ and
	\begin{align*}
		\lim_{n\rightarrow\infty}\frac{v_n}{|v_n|_1}=\Bigg(\frac{\rho^2}{(1-\rho)^2+\rho^2},\frac{(1-\rho)^2}{(1-\rho)^2+\rho^2}\Bigg)
	\end{align*}
	we have the following limit
	\begin{align*}
		B^{\rho}_{x,y}(\omega)=\lim_{n\rightarrow\infty}[G_{x,v_n}(\omega)-G_{y,v_n}(\omega)].
	\end{align*}
	\end{enumerate}
\end{theorem}
 In this work we shall be primarily interested in the distribution of the Busemann process restricted on a horizontal line on the lattice. Hence, for a fixed $0<\rho<1$ define the sequence $I^\rho$ through
\begin{align}\label{Ir}
	I^{\rho}_i=B^{\rho}_{(i-1)e_1,ie_1} \quad i\in \Z.
\end{align}
The following is from \cite{FS18}. It relates the results in Section \ref{sec:SoQ} with the Busemann functions. Recall the measure $\mu^{\rhov^n}$ from Theorem \ref{thm:FS2}.
\begin{theorem}\label{thm:FS}\cite[Theorem 3.2]{FS18}
	Let $\rhov^n=(\rho_1,..,\rho_n)$ be a decreasing vector with elements in $(0,1)$. Then
	\begin{align*}
		(I^{\rho_1},...,I^{\rho_n})\sim \mu^{\rhov^n}.
	\end{align*}
\end{theorem}  
\subsection{Upper bound on the number of jumps of the Busemann process }\label{subsec:sh1}
Recall $\{I^{\rho}\}_{\rho\in(0,1)}$ from \eqref{Ir}. The main result of this subsection is Proposition \ref{prop:ub2} -- on an interval of size $O(N^{2/3})$, the process $\mu\mapsto I^{\mu,N}:=I^{1/2-\frac{\mu}4N^{-1/3}}$ does not see much change i.e. with high probability it will have only a finite number of 'jumps'. It is in this subsection that we use the results of Subsection \ref{subsec:nc} as input.  

  Fix $\rho^*>0$ and let $\en\in\N$. Define the sets
\begin{align*}
\cE(\en)&=[-\rho^*,\rho^*]\cap \{i2^{-\en}\}_{i\in\Z}\\
\Ind(\en)&=2^\en\cE.
\end{align*}
Define the set of intensities
\begin{align*}
\rho_i=1/2-i2^{-M}N^{-1/3} \quad i\in\Ind.
\end{align*}
The vector $\rhov:\Ind\rightarrow \R_+$, defined through $\rhov_i=\rho_i$ is decreasing. Let $\h{I}=\{I^i\}_{i\in\Ind}\in\qs^{\Ind}$ be a vector of independent sequences with exponential distribution with intensities $\rhov$ i.e. $\h{I}\sim \nu^{\rhov}$. We can then define
\begin{align}\label{hf}
\h{f}^{M,N}=(f^1,...f^{|\Ind|})=\md^{(|\Ind|)}(\h{I})=\nc^{(|\Ind|)}(\h{I}),
\end{align}
where the map $\md^{(n)}$ was defined in \eqref{md}. In the following, we omit the superscript and simply write $\h{f}$. In particular, %\note{add reference}
\begin{align*}
f^{j}\geq f^i \quad \text{for $j,i\in\Ind$ such that $j\geq i$}.
\end{align*}
%   We can also use the vector $\h{I}$ in the construction above, so that we can associate to it the vector $\{(u^i,v^i)\}_{i\in\Ind}$ from Subsection \ref{subsec:nc}. 
In what follows we shall make use of the following set
\begin{align*}
\cE^{-i}(M)=\{j\in \Ind: j\leq \sup\{\Ind\}-i \} \quad \text{for $1\leq i \leq |\Ind|-1$}.
\end{align*}
In words, the set $\cE^{-i}(M)$ takes all indices in $\Ind$ except the $i$ largest ones. Let $x_0>0$ and define $\xin^N=\xin N^{2/3}$. Define the following events
\begin{align}\label{Bset}
B^{i,N}=\{f^i_j\neq f^{i+1}_j \text{ for some $j\in[-\xin^N,\xin^N]\cap \Z$}\} \quad \forall i\in	\cE^{-1},N\in \N. 
\end{align}
Define
\begin{align*}
F_{B^i}:=1_{B^{i,N}}+1_{B^{i+1,N}}+1_{B^{i+2,N}} \quad \text{for all $i\in\cE^{-3}$}.
\end{align*}
In words, the function $F_{B^i}$ counts the number of consecutive changes in the four elements of the vector $\h{f}$, sitting between index $i$ to $i+3$.   The following two maps will be useful. For any $k,l\in\Z$ such that $k\leq l$ we define
\begin{align}\label{summ}
S^{k,l}(I)=\sum_{i=k}^{l}I_i.
\end{align}
For $m\in\Z$, define the map $\psi^{m}:\R_+ \times \qs\times \qs \rightarrow \qs^{[m,\infty)}$ through
\begin{align}\label{psi}
\psi^m(J,I^1,I^2)_i=\inf_{m \leq j\leq i}(J+S^{m,j-1}(I^1-I^2)-I^2_j)^-.
\end{align}
The following result shows that it is not typical to see three consecutive changes in the vector $\h{f}$ when restricted to $[-\xin^N,\xin^N]$.
\begin{lemma}\label{lem:Fb}%\note{checked}
	Let $\rho^*\geq 1$,$x_0\geq 1$, $(32)^3\vee 64(\rho^*)^3<N$ and $\frac{1}{2}\log_2\big(x_0\big)+4<M$. Then for any $i\in \Ind^{-3}$ 
	\begin{align}\label{4b}
	\P\big(F_{B^i}\geq2\big)\leq 2^{20}x_02^{-2M}.
	\end{align}
\end{lemma}
\begin{proof}
	Fix  $i\in \Ind^{-3}$. Let $\h{I}=(I^1,I^2,I^3,I^4)\sim \nu_{\rhov_4}$ where $\rhov_4=(\rho_i,\rho_{i+1},\rho_{i+2},\rho_{i+3})$. Let $(\hat{f}^1,\hat{f}^2,\hat{f}^3,\hat{f}^4)=\md^{(4)}(\h{I})=\nc^{(4)}(\h{I})$. Define 
	\begin{align*}
	\hat{B}^{N,l}=\{\hat{f}^l_j\neq \hat{f}^{l+1}_j \text{ for some $j\in[-\xin^N,\xin^N]\cap \Z$}\} \quad  l\in\{1,2,3\}.
	\end{align*}
	From Theorem \ref{thm:FS} $(\hat{B}^{N,1},\hat{B}^{N,2},\hat{B}^{N,3})$ has the same distribution as $(B^{N,i},B^{N,i+1},B^{N,i+2})$ from \eqref{Bset}, so it is enough to show that
	\begin{align}\label{ub5}
	\P\big(F_{\hat{B}}\geq 2\big)\leq 2^{20}x_02^{-2M},
	\end{align}
	where 
	\begin{align*}
	F_{\hat{B}}:=1_{\hat{B}^{N,1}}+1_{\hat{B}^{N,2}}+1_{\hat{B}^{N,3}}.
	\end{align*}
	From \eqref{vf}--\eqref{ftf}, \eqref{d1} and \eqref{er}, 
	\begin{align}\label{feq}
	\hat{f}^l=D(v^l,u^l)=D(\hat{f}^{l-1},u^l)=\hat{f}^{l-1}+Er(\hat{f}^{l-1},u^l) \quad l\in\{2,3,4\}.
	\end{align}
	From \eqref{feq} and the definition of $\hat{B}^{N,l}$, it follows that %\note{add reference to the below}
	\begin{align}\label{eq1}
	\hat{B}^{N,l}=\{S^{-\xin^N,\xin^N}(Er(\hat{f}^{l},u^{l+1}))>0\}\stackrel{\eqref{se}}{=}\{\psi^{-\xin^N}(J^l,\hat{f}^{l},u^{l+1})_{\xin^N}>0\}\quad \forall l\in\{1,2,3\},
	\end{align}
	where 
	\begin{align*}
	J^l=S_{oj}(\hat{f}^{l},u^{l+1})_{x_0^N-1}.%\note{perhaps use a different notation for the suorjourn time}
	\end{align*}
	{\bf{What's ahead?}}\label{disc}\\
		Let us take a pause here to explain the main idea of the rest of the proof, which otherwise  might be lost on the reader for the trees (or displays). A quick inspection of the map $\psi$ from \eqref{psi} and equation \eqref{eq1}, shows that the probability of the event $\hat{B}^{N,l}$ can be bounded from above by the probability that the random walk
	\begin{align*}
		\mathbb{S}^l:=S^{-x_0,\cdot}(u^{l+1},\hat{f}^l)
	\end{align*}
	  does not go above the  value $J^l>0$ - in other words, in order for the event $\hat{B}^{N,l}$ to happen, the random walk $\mathbb{S}^l$ has to go over a potential wall of size $J^l$.  The statistics of $\{J^l\}_{l=1}^3$ are known to us from  Proposition \ref{prop:J} - these are i.i.d.\ exponential random variables with intensity $2^{-M}N^{-1/3}$. This property is crucial, indeed,  in the diffusive scaling the picture translates to three Brownian motions of drifts $\mu_l=\mu_0+l2^{-M+2}$ for some $\mu_0$ depending on the index $i$ from \eqref{4b}, each trying to cross a random barrier of order $2^M$. When $M$ is large, from the independence, we might expect that the probability of two events in $\{\hat{B}^{N,l}\}_{l=1}^3$ to occur will be of order ${2^{-2M}}$. This argument could be valid if it were not for the fact that  the random walks $\{\mathbb{S}^l\}_{l=1}^3$ are not independent of the barriers $\{J^l\}_{l=1}^3$, nor are the random walks themselves independent. 
	
	Our strategy to work around this difficulty is two fold. First, we decouple the random walks from the barriers by showing that they can be controlled by random walks coming from empty queues starting at  $-x_0$,  i.e. random walks coming from packed-melonization (see the discussion of packed melonization vs stationary one in Section \ref{subsec:W}), this technical bit is proved in Section \ref{subsec:CSE}. Let us denote by $\tilde{f}\in \qs^4$ this packed-melon and let
	\begin{align}\label{RW}
		\mathbb{S}^{l,\text{packed}}:=S^{-x_0,\cdot}(\tilde{u}^{l+1},\tilde{f}^l)
	\end{align}
	   be the new random walks. The next step is to find one random walk that bounds all $\{\mathbb{S}^{l,\text{packed}}\}_{l=1}^3$ uniformly in $l$. Indeed, this will relieve us of trying to figure out the statistics of $\{\mathbb{S}^{l,\text{packed}}\}_{l=1}^3$ which seem to be quite complicated. The first thing to note is that as $\tilde{f}$ is increasing, replacing $\tilde{f}^l$ in \eqref{RW} with $\tilde{f}^1$ will result in a random walk that dominates $\mathbb{S}^l$ for any $l\in\{1,2,3\}$. Similarly, the top line of the melon, $\hat{f}^4$, dominate all other lines that appear in the process of the packed-melonization of $\tilde{f}$, in particular $\{\tilde{u}^{l+1}\}_{l=1}^3$. To conclude,  we can dominate the random walk in \eqref{RW} by $\mathbb{S}^{\text{top}}:=S^{-x_0,\cdot}(\tilde{f}^4,\tilde{f}^1)$. As $\tilde{f}^1$ and $\tilde{f}^4$ are not independent, their joint statistics is not trivial, and we use a sub-optimal bound on  $\mathbb{S}^{\text{top}}$ developed in Lemma \ref{lem:bs}. 
	   
	   From the proof, it seems that it should be possible to prove that the probability of seeing $k$ consecutive changes in the vector $\h{f}$ from \eqref{hf} should decay like $2^{-Mk}$. We believe that these ideas could help to improve the bound in \eqref{ub82},  but may not be enough to obtain the optimal bound which we expect to be  exponential in $M$. Indeed, the fact that the size of $\mathbb{S}^{\text{top}}$ grows with $M$  i.e.\, $\mathbb{S}^{\text{top}}\approx \tilde{f}^M \approx 2\sqrt{M}$, could be in the way of  obtaining a sharp bound on the left hand side of \eqref{ub82} using our strategy.

	Let us now get back to the proof. Next we control $\{u^l\}_{l\in\{2,3,4\}}$. From the definition of $u^l$ in \eqref{xieta}
	\begin{align*}
	u^2&=I^2\\
	u^3&=D(R(I^1,I^2),I^3)\\
	u^4&=D^{(3)}\Big(R\big(D(I^1,I^2),D(R(I^1,I^2),I^3)\big),R\big(R(I^1,I^2),I^3\big),I^4\Big).
	\end{align*}
	In the next display we use the map $D_{m,0}$ from \eqref{qm} and \eqref{ode3}. For any $m\in\Z$
	\begin{align}\label{uub}
	\hat{u}^2:=\Pi_{[m,\infty)}(u^2)&=\Pi_{[m,\infty)}(I^2)\\\nonumber
	\hat{u}^3:=\Pi_{[m,\infty)}(u^3)&\leq D_{m,0}(R(I^1,I^2),I^3)\\\nonumber
	\hat{u}^4:=\Pi_{[m,\infty)}(u^4)&\leq D^{(3)}_{m,0}\Big(R\big(D(I^1,I^2),D(R(I^1,I^2),I^3)\big),R\big(R(I^1,I^2),I^3\big),I^4\Big).\nonumber
	\end{align}
	In what follows, we shall use the function $\Phi_{m,n}$ defined in \eqref{mx}. As $\hat{f}^1\leq \hat{f}^l$ for all $l\in\{2,3,4\}$,  
	\begin{align}\label{ch}
	\hat{C}^{N,l}&:=\{J^{l-1}<\Phi_{-\xin^N,\xin^N}(\hat{u}^l-\hat{f}^1)\} \qquad \forall l\in\{2,3,4\}\\
	&\supseteq\{\psi^{-\xin^N}(J^{l-1},\hat{f}^{l-1},\hat{u}^l)_{\xin^N}>0\}= \hat{B}^{N,l}\nonumber.
	\end{align}
	By definition $\hat{f}^1=I^1$  so
	\begin{align*}
	\Phi_{-\xin^N,\xin^N}(\hat{u}^l-\hat{f}^1)&=\Phi_{-\xin^N,\xin^N}(\hat{u}^l-I^1).
	\end{align*}
	In the next display we use the map $D_{m,0}$ from \eqref{qm}, Lemma \ref{lem:qz2} and that by definition $R(I^1,I^2)\leq I^1,I^2$ for $I^1,I^2\in\qs$. From \eqref{uub}, for $m \leq j$ 
	\begin{align}\label{ub71}
	S^{m,j}(\hat{u}^3) &\stackrel{\eqref{ode4}}{\leq} S^{m,j}\big(D_{m,0}(I^2,I^3)\big)\\
	S^{m,j}(\hat{u}^4)&\stackrel{\eqref{ode4}}{\leq}\label{ub72} 	S^{m,j}\Bigg(D^{(3)}_{m,0}\Big(D(R(I^1,I^2),I^3),I^3,I^4\Big)\Bigg)\\\nonumber
	&\stackrel{\eqref{ode3}}{\leq} 	S^{m,j}\Bigg(D^{(3)}_{m,0}\Big(D_{m,0}(R(I^1,I^2),I^3),I^3,I^4\Big)\Bigg)\\\nonumber
	&\stackrel{\eqref{ode4}}{\leq}	S^{m,j}\Bigg(D^{(3)}_{m,0}\Big(D_{m,0}(I^2,I^3),I^3,I^4\Big)\Bigg).
	\end{align}
	From \eqref{ub71}-\eqref{ub72}  with $m=-x_0^N$ and $i=x_0^N$
	\begin{align}\label{ub73}
	\Phi_{-\xin^N,\xin^N}(\hat{u}^l-I^1)\leq 
	\begin{cases}
	\Phi_{-\xin^N,\xin^N}(I^2-I^1) & l=2\\
	\Phi_{-\xin^N,\xin^N}\big(D_{m,0}(I^2,I^3)-I^1\big) & l=3\\
	\Phi_{-\xin^N,\xin^N}\Big(D^{(3)}_{m,0}\Big(D_{m,0}(I^2,I^3),I^3,I^4\Big)-I^1\Big) & l=4.
	\end{cases}
	\end{align}
	Using \eqref{ub73} in Lemma \ref{lem:bs}  with $m=-x_0^N$ and $i=x_0^N$
	\begin{align*}
	\Phi_{-\xin^N,\xin^N}(\hat{u}^l-I^1)\leq 
	\begin{cases}
	\Phi_{-\xin^N,\xin^N}(I^2-I^1) & l=2\\
	\Phi_{-\xin^N,\xin^N}(I^2-I^1)+\Phi_{-\xin^N,\xin^N}(I^3-I^2)+\Mx_{-\xin^N,\xin^N}(I^3) & l=3\\
	\Phi_{-\xin^N,\xin^N}(I^2-I^1)+2\Phi_{-\xin^N,\xin^N}(I^3-I^2)+ \Phi_{-\xin^N,\xin^N}(I^4-I^3) & l=4\\+\Mx_{-\xin^N,\xin^N}(I^3)+\Mx_{-\xin^N,\xin^N}(I^4)+\Mx_{-\xin^N,\xin^N}(I^3+I^4),
	\end{cases}
	\end{align*}
	where the computation for case $l=4$ is as follows
	\begin{align*}
	&\Phi_{-\xin^N,\xin^N}(\hat{u}^4-I^1)\\
	&\leq\Phi_{-\xin^N,\xin^N}(D_{-x_0^N,0}(I^2,I^3)-I^1)+\Phi_{-\xin^N,\xin^N}(I^3-D_{-x_0^N,0}(I^2,I^3))+\Phi_{-\xin^N,\xin^N}(I^4-I^3) \\&+\Mx_{-\xin^N,\xin^N}(I^4)+\Mx_{-\xin^N,\xin^N}(I^3+I^4)\\
	&\stackrel{\eqref{bs}}{\leq} \Phi_{-\xin^N,\xin^N}(I^2-I^1)+\Phi_{-\xin^N,\xin^N}(I^3-I^2)+\Mx_{-\xin^N,\xin^N}(I^3)+ \Phi_{-\xin^N,\xin^N}(I^3-I^2)\\
	&+\Phi_{-\xin^N,\xin^N}(I^4-I^3)+\Mx_{-\xin^N,\xin^N}(I^4) +\Mx_{-\xin^N,\xin^N}(I^3+I^4)\\
	&=\Phi_{-\xin^N,\xin^N}(I^2-I^1)+2\Phi_{-\xin^N,\xin^N}(I^3-I^2)+ \Phi_{-\xin^N,\xin^N}(I^4-I^3) \\
	&+\Mx_{-\xin^N,\xin^N}(I^3)+\Mx_{-\xin^N,\xin^N}(I^4)+\Mx_{-\xin^N,\xin^N}(I^3+I^4),
	\end{align*}
	where we used Lemma \ref{lem:bs} twice and the fact that $D(I^2,I^3)\geq I^2$. We see that for all $l\in \{2,3,4\}$
	\begin{align*}
	&\Phi_{-\xin^N,\xin^N}(\hat{u}^l-I^1)\leq \Phi_{-\xin^N,\xin^N}(I^2-I^1)+2\Phi_{-\xin^N,\xin^N}(I^3-I^2)+ \Phi_{-\xin^N,\xin^N}(I^4-I^3) \\
	&+\Mx_{-\xin^N,\xin^N}(I^3)+\Mx_{-\xin^N,\xin^N}(I^4)+\Mx_{-\xin^N,\xin^N}(I^3+I^4):=X_{max}.
	\end{align*}
	We have
	\begin{align}\label{xm}
	&\P\big(X_{max}>t\big)\\\nonumber
	&\leq \P(\Phi_{-\xin^N,\xin^N}(I^2-I^1)+ 2\Phi_{-\xin^N,\xin^N}(I^3-I^2)+\Phi_{-\xin^N,\xin^N}(I^4-I^3)>t/2)\\\nonumber
	&+\P(\Mx_{-\xin^N,\xin^N}(I^3)+\Mx_{-\xin^N,\xin^N}(I^4)+\Mx_{-\xin^N,\xin^N}(I^3+I^4)>t/2)\\\nonumber
	&\leq 4\P(\Phi_{-\xin^N,\xin^N}(I^4-I^1)>t/8)\\\nonumber
	&+\P(2\Mx_{-\xin^N,\xin^N}(I^3)+2\Mx_{-\xin^N,\xin^N}(I^4)>t/2)\\\nonumber
	& \leq 4\P(\Phi_{-\xin^N,\xin^N}(I^4-I^1)>t/8)\\\nonumber
	&+4\P(\Mx_{-\xin^N,\xin^N}(I^4)>t/8),
	\end{align}
	where we used the fact that $\Phi_{-\xin^N,\xin^N}(I^4-I^1)$ ($I^4$) stochastically dominates all the terms in the second line(fifth line) of \eqref{xm}. As $I^4-I^1 \sim \mathrm{Exp}(\rho_i+3\cdot2^{-M}N^{-1/3},\rho_i)$ for some $\rho_i\in \cE$, by Lemma \ref{lem:rwub},
	\begin{align}\label{ub9}
	&\P(\Phi_{-\xin^N,\xin^N}(I^4-I^1)>t/8)\\
	&\leq \Bigg(1+\frac{(2x_0^N+1)^{-1}}{(\rho_i-\sqrt{(3\cdot2^{-M}N^{-1/3})^2+4(2x_0^N+1)^{-1}})\rho_i}\Bigg)^{2x_0^N+1}e^{-\frac{\sqrt{(3\cdot2^{-M}N^{-1/3})^2+4(2x_0^N+1)^{-1}}-3\cdot2^{-M}N^{-1/3}}{2}{t/8}}\nonumber\\
	&\leq \Bigg(1+\frac{(2x_0^N+1)^{-1}}{[1/2-\rho^* N^{-1/3} -\sqrt{(3\cdot2^{-M}N^{-1/3})^2+4(2x_0^N+1)^{-1}}](1/2-\rho^* N^{-1/3})}\Bigg)^{2x_0^N+1}\nonumber\\
	&\times e^{-\frac{\sqrt{(3\cdot2^{-M}N^{-1/3})^2+4(2x_0^N+1)^{-1}}-3\cdot2^{-M}N^{-1/3}}{2}{t/8}}\nonumber
	\end{align}
	Setting $\rho^*\geq 1$, $x_0\geq 1$, $(32)^3\vee 64(\rho^*)^3<N$ we have the following bounds
	\begin{align}\label{ub8}
	1/2-\rho^* N^{-1/3}>1/4, \quad \sqrt{(3\cdot2^{-M}N^{-1/3})^2+4(2x_0^N+1)^{-1}}<1/8\quad \text{and} \quad 3\cdot2^{-M}N^{-1/3}<\frac1{16}.
	\end{align}
	Setting $\frac{1}{2}\log_2\big(x_0\big)+4<M$ gives
	\begin{align}\label{ub7}
	\frac{\sqrt{(3\cdot2^{-M}N^{-1/3})^2+4(2x_0^N+1)^{-1}}-3\cdot2^{-M}N^{-1/3}}{2}\geq \frac{1}{4}(2x_0+1)^{-1/2}N^{-1/3}
	\end{align}
	Using \eqref{ub8} and \eqref{ub7} in \eqref{ub9} leads to 
	\begin{align}\label{ub}
	&\P(\Phi_{-\xin^N,\xin^N}(I^4-I^1)>t/8)\\
	&\leq\big(1+(32)^{-1}(2x_0^N+1)^{-1}\big)^{2x_0^N+1}e^{-\frac{1}{32}(2x_0+1)^{-1/2}{N^{-1/3}t}}\leq e^{-\frac{1}{32}(2x_0+1)^{-1/2}{N^{-1/3}t}+1},\nonumber
	\end{align}
	where we used $(1+\frac{x}{n})^n\leq e^x$. Using union bound and \eqref{ub8}
	\begin{align}\label{ub3}
	\P(\Mx_{-\xin^N,\xin^N}(I^4)>t/8)\leq (2x_0^N+1)e^{-(1/2-\rho^*N^{-1/3})t/8}\leq (2x_0^N+1)e^{-t/32}.
	\end{align}
	Plugging \eqref{ub} and \eqref{ub3} into \eqref{xm},
	\begin{align}\label{Xm}
	\P\big(X_{max}>t\big)\leq 4e^{-\frac{1}{32}(2x_0+1)^{-1/2}{N^{-1/3}t}+1}+4(2x_0^N+1)e^{-t/32}.
	\end{align}
	Recall the set $\hat{C}^{N,l}$ from \eqref{ch}. Denote
	\begin{align*}
	F_{\hat{C}}:=1_{\hat{C}^{l,N}}+1_{\hat{C}^{l+1,N}}+1_{\hat{C}^{l+2,N}}.
	\end{align*}
	It follows that
	\begin{align}\label{ub11}
	&\P(F_{\hat{C}^i}\geq 2)\leq \\\nonumber
	&\P(\hat{C}^{N,i+2}\cap \hat{C}^{N,i+3})+\P(\hat{C}^{N,i+2}\cap \hat{C}^{N,i+4})+\P(\hat{C}^{N,i+3}\cap \hat{C}^{N,i+4})\\\nonumber
	&\leq 4\P\big(\max\{J^2,J^3\}\leq  X_{max}\big).
	\end{align}
	By Proposition \ref{prop:J}, $\{J^l\}_{l\in\{2,3,4\}}$ are i.i.d. r.v.s and
	\begin{align*}
	J^l \sim \textrm{Exp}(\rho_{i-1}-\rho_i)=\textrm{Exp}(2^{-M}N^{-1/3}) \quad \text{for $l\in\{1,2,3\}$}.
	\end{align*}
	Using the joint independence of $J^2,J^3$ and $X_{max}$
	\begin{align}\label{ub6}
	\P\big(\max\{J^2,J^3\}\leq X_{max}\big)=2\int_0^\infty 2^{-M}e^{-2^{-M}t}(1-e^{-2^{-M}t})\P(X_{max}>tN^{1/3})dt.
	\end{align}
	Plugging \eqref{Xm} in \eqref{ub6}
	\begin{align}\label{ub10}
	\P\big(\max\{J^2,J^3\}\leq X_{max}\big)\leq 2^{18}x_02^{-2M},
	\end{align}
	where the full computation can be found in \eqref{poub10}.
	Plugging \eqref{ub10} in \eqref{ub11} gives
	\begin{align}\label{ub4}
	\P(F_{\hat{C}^i}\geq 2)\leq 2^{20}x_02^{-2M}
	\end{align}
	From \eqref{ch} $\P(F_{\hat{B}^i}\geq 2)\leq \P(F_{\hat{C}^i}\geq 2)$, so that \eqref{ub4} implies \eqref{ub5}. The proof is complete.
\end{proof}
\begin{corollary}\label{cor:Fb}%\note{checked}
	Let $\rho^*\geq 1$, $x_0\geq 1$, $(32)^3\vee 64(\rho^*)^3<N$ and $\frac{1}{2}\log_2\big(x_0\big)+4<M$. 
	\begin{align}\label{3b}
	\P\big(F_{B^i}\geq2 \text{ for some $i\in\cE^{-3}$}\big)\leq 2^{20}x_02^{-M}.
	\end{align}
\end{corollary}
\begin{proof}
	As $|\cE^{-2}|= 2^M-3$, using union bound and Lemma \ref{lem:Fb} we obtain the result.
\end{proof}

Recall $I^{\rho}$ from \eqref{Ir}. Let $x_0>0$ and $N_0>(\rho^*)^3$. For $N>N_0$  define the family of functions $\{F^{N,x_0}_\mu\}_{\mu\in [-\rho^*,\rho^*]}$ through
\begin{align}\label{F}
F^{N,x_0}_\mu(i)=I^{1/2-\frac{\mu}4N^{-1/3}}_i \quad \forall i\in  \llbracket-x_0^N,x_0^N\rrbracket
\end{align}
When the interval is clear from the context we shall omit the superscript $x_0$.
%i.e. $F^{N,x_0}_\rho(x)$ is the restriction of  the function $G^{N}_\rho(x)$ to the interval $[-x_0,x_0]$. When the interval is clear from the context we shall omit the superscript $x_0$. From the Busemann function property we see that $F^{N}_\rho(x)$  is rcll function from the interval $[-\rho^*,\rho^*]$ to the state space $C([-x_0,x_0])$ i.e. $F^{N}_\rho(x)\in D([-\rho^*,\rho^*],C([-x_0,x_0]))$. Given   $f\in D([-\rho^*,\rho^*],C([-x_0,x_0]))$, 
We call $\rho_0\in (-\rho^*,\rho^*)$ an epoch of $	F^{N}$ if for some $i\in [-x_0^N,x_0^N]$, $	F^{N}_{\cdot}(i)$ is not constant on  any open interval containing $\rho_0$.  Our next result shows that with high probability the epochs of $F^{N}_\rho$ are not too close. Let $\delta>0$ and define
\begin{align}\label{A}
\cA^{\delta,N,\rho^*,x_0}:=\{\text{there exist epochs $\rho_0,\rho_1\in [-\rho^*,\rho^*]$ of $F^{N}$ such that $\abs{\rho_1-\rho_0}\leq \delta$}\}.
\end{align}
\begin{proposition}\label{prop:ub2}%\note{checked}
	For  $\rho^*\geq 1$, $x_0\geq 1$, $(32)^3\vee 64(\rho^*)^3<N$ and $\delta< 2^{-4}x_0^{-1/2}$
	\begin{align*}
	\P\big(\cA^{\delta,N,\rho^*,x_0}\big)\leq 2^{20}x_0 \delta.
	\end{align*}
\end{proposition}
\begin{proof}
	Suppose there exist two epochs $\rho_0,\rho_1\in [-\rho^*,\rho^*]$  of $F^{N}$ such that $\abs{\rho_1-\rho_0}=\delta<1$. Let $M\in\N$ be such that 
	\begin{align}\label{sand}
	2^{-M}<\delta \leq 2^{-M+1}.
	\end{align}
	Then there exists $\rho_e\in\cE(M)$ and two possible cases (without loss of generality we assume $\rho_0<\rho_1$)
	\begin{enumerate}[label=Case \arabic*:]
		\item  	The two epochs belong to two adjacent intervals of size $2^{-M}$ i.e.
		\begin{align}\label{xe1}
		\rho_0\in[\rho_e,\rho_e+2^{-M}) \quad \text{and} \quad \rho_1\in[\rho_e+2^{-M},\rho_e+2^{-M+1}).
		\end{align}
		\item $\rho_0$ and $\rho_1$ belong to the leftmost and  rightmost intervals  of three consecutive intervals of size $2^{-M}$ i.e.
		\begin{align}\label{xe}
		\rho_0\in[\rho_e,\rho_e+2^{-M}) \quad \text{and} \quad \rho_1\in[\rho_e+2^{-M+1},\rho_e+2^{-M+2}).
		\end{align}
	\end{enumerate}
	Form \eqref{mon}, if $\rho_0$ and $\rho_1$ are epochs of $F^{N}$ satisfying either \eqref{xe1} or \eqref{xe}, then
	\begin{enumerate}[label=Case \arabic*:]
		\item  	\eqref{xe1} implies
		\begin{align}\label{xe2}
		I^{\rho_e,N}< I^{\rho_e+2^{-M},N}< I^{\rho_e+2^{-M+1},N}.
		\end{align}
		\item \eqref{xe} implies
		\begin{align}\label{xe3}
		I^{\rho_e,N}< I^{\rho_e+2^{-M},N}< I^{\rho_e+2^{-M+2},N}.
		\end{align}
	\end{enumerate}
	The inequalities in \eqref{xe2}--\eqref{xe3} should be understood element-wise, where strict inequality $<$ means that the vector to its right has at least one element that is strictly larger than the corresponding element of the vector to the left of it.   
	Denote
	\begin{align*}
	&\cB^{M,N}:=\{\text{there exists epochs $\rho_0,\rho_1\in [-\rho^*,\rho^*]$ of $F^{N}$ such that $2^{-M}<|\rho_1-\rho_0|\leq 2^{-M+1}$}\}\\
	&\cC^{M,N}:=\{\text{$\exists \rho_e\in \cE$ such that either \eqref{xe2} or \eqref{xe3} hold} \}.
	\end{align*}
	We have thus shown that
	\begin{align*}
	\cB^{M,N}\subseteq \cC^{M,N}.
	\end{align*}
	It follows from Corollary \ref{cor:Fb} that 
	\begin{align}\label{B}
	\P(\cC^{M,N})\leq 2^{20}x_02^{-M}.
	\end{align}
	Note that 
	\begin{align*}
	\cD^{M,N}\subseteq \bigcup_{M\leq i}\cB^{i,N}\subseteq \bigcup_{M\leq i}\cC^{i,N},
	\end{align*}
	where 
	\begin{align*}
	\cD^{M,N}:=\{\text{there exists epochs $\rho_0,\rho_1\in [-\rho^*,\rho^*]$ of $F^{N}$ such that $|\rho_1-\rho_0|\leq 2^{-M+1}$}\}
	\end{align*}
	Using \eqref{B} and union bound, 
	\begin{align}\label{Cb}
	\P(\cD^{M,N})\leq \sum_{i=M}^{\infty}\P(\cC^{i,N})\leq 2^{20}x_02^{-M+1}.
	\end{align} 
	%{\color{blue}{Computation: comparing to the integral $t^4e^{-t/2}+\int_t^{\infty}s^4e^{-\log(2)s}ds\leq 768t^4e^{-t/2}$
	%\begin{align*}
	%	4\sum_{i=M}^{\infty}i^4e^{-i/2}\leq 4\times 768M^4e^{-M/2}.
	%\end{align*}
	%For $x_0\geq 4$, $M\geq 32x_0$ $M-\frac{1}{32}(x_0)^{-1/2}{M^2}<-M^{3/2}/2$, and 
	%\begin{align*}
	%	4e\sum_{i=M}^{\infty} e^{-i^{3/2}/2}\leq 4e2e^{-M^{3/2}/2}
	%\end{align*}
	%Similarly, for For $x_0\geq 4$, $M\geq 32x_0$ $M-\frac{1}{32}(x_0)^{-1/2}{M^2}<-M^{3/2}/2$, and $N>216$ $\log(x_0^N)+M-\frac{1}{32}N^{1/3}{M^2}<-M^{3/2}/2$ and so
	%
	%	\begin{align*}
	%		8\sum_{i=M}^{\infty} e^{-i^{3/2}/2}\leq 82e^{-M^{3/2}/2}
	%	\end{align*}
	%}}
	Finally, using \eqref{sand} in \eqref{Cb}
	\begin{align*}
	&\P\big(\text{$\exists$ epochs $\rho_0,\rho_1\in [-\rho^*,\rho^*]$ of $F^{N}$ such that $|\rho_1-\rho_0|\leq \delta$}\big)\leq \P(\cD^{M,N})\\
	&\leq 2^{21}x_0 \delta.
	\end{align*}
	
\end{proof}
%\begin{corollary}\label{cor:ub}
%	\begin{align*}
%		\P\Big([\sup_{1\leq i \leq n-1}J^i_x-w^i_x]>t\Big)\leq (n-1)e^{-\rho_{max}t}
%	\end{align*}
%where $\rho_{min}=\inf_{0\leq i \leq n}\rho_{i}$.
%\end{corollary}
%\begin{proof}
%	From \eqref{Jdec}, \eqref{dis2} and union bound
%	\begin{align*}
%		\P\Big(\big[\sup_{1\leq i \leq n-1}J^i_x-w^i_x\big]>t\Big)\leq \sum_{i=1}^{n-1} \P(s^i_x>t) \leq \sum_{i=1}^{n-1} e^{-\rho_{i-1}t}\leq (n-1)e^{-\rho_{min}t}.
%	\end{align*}
%\end{proof}
%\section*{Convergence}
%Let $\xin>0$. Denote by $\cX=C[-\xin,\xin]$ the space of continuous functions equipped with the uniform topology. Let $\rhoi>0$ and let $D\big([-\rhoi,\rhoi]\big)$ be the space of rcll functions form $[-\rhoi,\rhoi]$ to $\cX$. Let $F^N:[-\rhoi,\rhoi]\rightarrow \cX$ be defined by \note{define $B^{\rho}$}
%\begin{align*}
%	F^N(\mu)=f^N_\mu(\cdot)
%\end{align*}
%where $f_\mu:[-\xin,\xin]\rightarrow \R$ is a continuous function given by \note{the linear interpolation of the} 
%\begin{align*}
%	f_\mu(x)=N^{-1/3}\sum_{i=\xin N^{2/3}}^{xN^{2/3}}\big(B^{1/2+\mu N^{-1/3}}(ie_1,(i+1)e_1)-2\big) \quad x\in[-\xin,\xin].
%\end{align*}
% As $B^{\cdot}_{ie_1,(i+1)e_1}$ is rcll for every $i\in \Z$, it follows that $F^N\in D\big([-\rhoi,\rhoi],\cX\big)$. Next define
% \begin{align*}
% 	G^N(\mu)=g^N_\mu(\cdot)
% \end{align*}
% where
% \begin{align*}
% 	g^N_\mu(x)=\sum_{i=0}^{2^N}1_{[{a^N_i,a^N_{i+1}})} f^N_\mu(a_i^N)
% \end{align*}
% where
% \begin{align*}
% 	a_i^N=-\xin N^{2/3}+i2^{-N+1}\xin N^{2/3}.
% \end{align*}
\section{Convergence to the Stationary Horizon}
In this section we prove the main results of this work, i.e. we show that the Busemann process on a horizontal line, converges to a non-trivial limit under diffusive scaling - Theorem \ref{thm:sh} and Theorem \ref{thm:prop}. The sense in which the convergence takes place will be defined rigorously in Subsection \ref{subsec:sh2}, here we only try to give a general idea of our intentions. Recall $\{I^{\mu}\}_{\mu\in\R}$ from \eqref{Ir}, a family of sequences sampling the value of the Busemann process on the horizontal edges that belong to the $x$-axis. As we have an eye to use $\{I^{\mu}\}_{\mu\in\R}$ as initial conditions for our LPP model in future work,  a problem to which the  relevant scaling is the $1:2:3$-KPZ scaling, the scaling for the initial condition is the diffusive one - $1:2$. For a fixed $\mu\in \R$ we interpolate between the points of $I^{\mu,N}=I^{1/2-\frac{\mu}4N^{-1/3}}$ and scale diffusively to obtain a continuous function $G^N_{\mu}$ on the real line. One should think of $G^N$ as a process that takes values in the space of continuous functions and is indexed by its drift $\mu\in\R$. The goal of this section is to show that $G^N$ converges in some sense to some non-trivial limit $G$ which we refer to as  the \textit{Stationary Horizon}. The processes $G^N$ are rcll, inherited from Property \eqref{rc} of the Busemann process, hence the appropriate and perhaps natural sense of convergence is that of the Skorohod topology. 

This section contains four subsections, intended to establish  the two necessary and sufficient conditions for convergence in the Skorohord topology (see Theorem \ref{thm:tc}).  Section \ref{subsec:sh2} introduces the prelimiting sequence $\{G^N\}_{N\in\N}$, the space in which it lives and its associated topology. Here we also prove Condition \ref{fdc} of Theorem \ref{thm:tc}.  Section \ref{subsec:fdd}  shows that  the finite dimensional distributions of $G^N$ converge weakly, and so verifying Condition \ref{mocc} of Theorem \ref{thm:tc}. Finally, Section \ref{subsec:pet} combines the results of the preceding Sections to complete the proof of Theorem \ref{thm:sh} and Theorem \ref{thm:prop}.

\subsection{The space $D(\R,\cS)$}\label{subsec:sh2}
This subsection is dedicated to defining the processes $G^N$, the space $\cS$ where they take values, and the space $D(\R,\cS)$ in which they 'live'. We will also prove some of the properties of these processes in this subsection. Here we also prove one of the main result of this section -- Proposition \ref{prop:rc}, essentially tightness for the sequence $G^N$

Let $I\subset\R$ be an interval and let $C(I,\R)$ be the set of real valued continuous functions on $I$. We will use the abbreviation $C(I)$. For $n\in\N$ define the pseudo-metric $d_n$ on $C(\R,\R)$ 
\begin{align*}
	d_n(f,g)=\sup_{x\in[-n,n]}|f(x)-g(x)|.
\end{align*} 
Let $\cS=C(\R,\R)$ be the set of continuous functions on the real line equipped with the following metric
\begin{align*}
	d(f,g)=\sum_{i=1}^{\infty}2^{-i}\frac{d_i(f,g)}{1+d_i(f,g)}.
\end{align*} 
The space $\cS$ is a complete separable metric space (Lemma \ref{lem:cs}). Note that since for any $f,g\in\cS$ it holds that  $d_i(f,g)\leq d_{i+1}(f,g)$, we have the following bound which we will make use of often in the sequel
\begin{align}\label{maprox}
	d(f,g)\leq d_n(f,g)+2^{-n} \quad \forall n\in\N.
\end{align}
Let $I=\{I_i\}_{i\in\Z}\in \qs$. Define the linear map $\cm:\qs\rightarrow \cS$ through
\begin{align}\label{cm}
	\cm(I)(t)=S^I_{\floor{t}}+(t-\floor{t})(S^I_{\floor{t}+1}-S^I_{\floor{t}})  \quad t\in\R
	%	\begin{cases}
	%		S^I_{\floor{t}}+(t-\floor{t})(S^I_{\floor{t}+1}-S^I_{\floor{t}}) & \quad t\geq 0\\
	%			S^I_{\floor{t}}+(t-\floor{t})(S^I_{\floor{t}}-S^I_{\floor{t}+1}) & \quad t< 0
	%	\end{cases}
\end{align}
where
\begin{align}\label{s}
	S^I_i=
	\begin{cases}
		S^{1,i}(I) & 0<i\\
		0 & i=0\\
		-S^{i+1,0}(I) & i<0,
	\end{cases}
\end{align}
and where the map $S^{k,l}$ was defined in \eqref{summ}.
In words, the map $\cm$ constructs a function $\cm(I)$ in the following way; it considers its input, say $I$, as the increments of the piecewise linear function $\cm(I)$ on the integers i.e..   
\begin{align}\label{eq17}
	\cm(I)(t)-\cm(I)(t-1)=I_t \quad t\in \Z.
\end{align}
As such continuous functions are only unique up to an additive constant, $\cm(I)$ is chosen such that it vanishes at the origin.   
Display \eqref{eq17} implies that for integers $s\leq t$
\begin{align}\label{sum}
	\cm(I)(t)-\cm(I)(s)=\sum_{s+1}^tI_i. 
\end{align}
%Here are a few facts about the map $\cm$ we shall make use of later. Suppose $I\in\qs$ is such that 
%\begin{align*}
%	\lim_{m\rightarrow -\infty}S^{m,i}(I)<0.
%\end{align*}
%Define the sequences
%\begin{align*}
%	I^{sup}_i&=\sup_{j \leq i} S^{j,i}I_j\\
%	I^{inf}_i&=
%	\begin{cases}
%		\inf_{i_0 < j \leq i} S^{i_0,j}(I_j) &  i> i_0 \qquad \text{for $i_0\in\Z$}\\
%		\inf_{i < j \leq i_0} S^{j,i_0}(I_j) &  i\leq  i_0
%	\end{cases}
%	\\
%	I^-&=(I_i)^-.
%\end{align*} 
%It is then not hard to verify that 
%\begin{align}\label{sup}
%	\cm(I^{sup})(t)&=\sup_{s \leq t} \cm(I)(t)-\cm(I)(s)\\[0.5ex]
%	\cm(I^{inf})(t)&=
%		\begin{cases}
%		\inf_{i_0 < j \leq i} \cm(I)(t)-\cm(I)(i_0) &  t\geq i_0 \label{inf}\\
%		\inf_{i < j \leq i_0} \cm(I)(i_0)-\cm(I)(t) &  t< i_0.
%	\end{cases}
%	\\[1.3ex]
%	d_N\big(\cm(I^-)&,(\cm(I))^{-}\big)\leq \sup_{i\in[-N,N]}|I_i| \label{minus}
%\end{align}
Recall the sequences $I^{\rho}$ from \eqref{Ir}. Let $G^N:\R\rightarrow\cS$
% \note{should change everything to $\rho^*$} 
be defined through
  \begin{align}
 	\tilde{G}^{N}_\mu(x)=&
 	\begin{cases}
 		N^{-1/3}\cm(I^{1/2-\frac\mu4N^{-1/3}}-\bm{2})(xN^{2/3})  & -N^{1/3}\leq \mu<N^{1/3}\\
 		N^{-1/3}\cm(I^{1/4}-\bm{2})(xN^{2/3}) &\qquad N^{1/3} \leq \mu \\
 		N^{-1/3}\cm(I^{3/4}-\bm{2})(xN^{2/3}) &\qquad \mu<-N^{1/3}
 	\end{cases}\nonumber \\
 \intertext{and,}
 &G^N_\mu(x)=\tilde{G}^N_{\mu+}(x)\label{G}.	
 \end{align}
 The reason for defining $G^N$ as in \eqref{G}, is so that $G^N$ is a rcll process. Indeed, from the fact that $\mu\mapsto I^\mu$ is a rcll process  $\tilde{G}^N$ is lcrl process.   For fixed $N$, we will consider $G^{N}$ as a function $\mu \mapsto \cS$. For $x_0>0$, we let $F^{N,x_0}:\R\rightarrow C[-x_0,x_0]$ be the restriction of $G^{N}$ to $[-x_0,x_0]$ i.e.
  \begin{align*}
  	F^{N,x_0}(x)=G^{N}(x) \quad x\in [-x_0,x_0].
  \end{align*}
  Recall that $D(\R,\cS)$ is the space of rcll functions from $\R$ to $\cS$. The next lemma shows that $\{G^{N}\}_{N\in\N}$ is a sequence of elements in  $D(\R,\cS)$.
\begin{lemma}%\note{checked}
	 $G^{N}\in D(\R,\cS)$ for every $N\in \N$.
\end{lemma}
\begin{proof}
	Clearly $G^{N}_\mu\in \cS$ for every $\mu\in\R$, so we must show that for every $N\in\N$ and $\mu\in\R$
	\begin{align}\label{rc2}
		\lim_{h\downarrow 0} d(G^N_\mu,G^N_{\mu+h})=0,
	\end{align}
and the following limit exists in $\cS$
\begin{align}\label{el}
	\lim_{h\uparrow 0} G^N_{\mu+h}.
\end{align}
	Note that from \eqref{G}, we need only show \eqref{rc2}--\eqref{el} for $-N^{1/3}\leq \mu<N^{1/3}$. From the rcll property (\ref{rc}) of the Busemann process, 
	\begin{align}\label{dF}
			\lim_{h\downarrow 0} d_n(F^{N,n}_\mu,F^{N,n}_{\mu+h})=0\qquad -N^{1/3}\leq \mu<N^{1/3}, n\in \N.
	\end{align}
	Now, fix $\mu\in[-N^{1/3},N^{1/3})$ 	and let $M\in\N$. From \eqref{dF}, there exists $h>0$ small enough such that
	\begin{align}\label{ub12}
		d(G^N_\mu,G^N_{\mu+h})<\sum_{i=1}^M2^{-i}\frac{d_i(F^{N,i}_\mu,F^{N,i}_{\mu+h})}{1+d_i(F^{N,i}_\mu,F^{N,i}_{\mu+h})}+2^{-M-1}<2^{-M}.
	\end{align}
	As \eqref{ub12} holds for every $M\in\N$, \eqref{rc2} follows. Next we show \eqref{el}. Once again, from  the rcll property (\ref{rc}) of the Busemann process, the following limit exists in $(d_n,C([-n,n],\R))$ for every $n\in\N$
	\begin{align}\label{lim}
		\hat{F}^{N,n}_{\mu-}:=\lim_{h\uparrow 0} F^{N,n}_{\mu+h}.
	\end{align}
	For $n\in\N$, construct the function
	\begin{align*}
		\tilde{F}^{N,n}_{\mu-}(x)=
		\begin{cases}
			\hat{F}^{N,n}_{\mu-}(x) & x\in[-n,n]\\
			\hat{F}^{N,n}_{\mu-}(-n) & x<-n\\
			\hat{F}^{N,n}_{\mu-}(n) & x>n.
		\end{cases}
	\end{align*}
	It is not hard to see that $\tilde{F}^{N,n}_{\mu-}\in \cS$ and that $\{\tilde{F}^{N,n}_{\mu-}\}_{n\in\N}$ is a Cauchy sequence in $\cS$, so that there exists $ G^{N}_{\mu-}\in\cS$ such that
	\begin{align}\label{lim2}
	d(\tilde{F}^{N,n}_{\mu-},G^{N}_{\mu-})\rightarrow 0 \qquad n\rightarrow\infty. 
	\end{align}
	For $M\in\N$ and  $h>0$ 
	\begin{align*}
		&d(G^N_{\mu-h},G^N_{\mu-})	\leq \sum_{i=1}^M2^{-i}\frac{d_i(F^{N,i}_{\mu-h},G^N_{\mu-})}{1+d_i(F^{N,i}_{\mu-h},G^N_{\mu-})}+2^{-M-1}\\
		&\leq \sum_{i=1}^M2^{-i}\frac{d_i(\tilde{F}^{N,n}_{\mu-},G^N_{\mu-})+d_i(F^{N,i}_{\mu-h},\tilde{F}^{N,n}_{\mu-})}{1+d_i(F^{N,i}_{\mu-h},G^N_{\mu-})}+2^{-M-1}.
	\end{align*}
From the construction of $\tilde{F}^{N,n}_{\mu-}$ and \eqref{lim2}, for every $i\in\N$ and $n>i$, $d_i(\hat{F}^{N,n}_{\mu-},G^N_{\mu-})=0$. This, with \eqref{lim} implies that for $h$ small enough
\begin{align}\label{ub13}
	d(G^N_{\mu-h},G^N_{\mu-})\leq 2^{-M}.
\end{align}
As \eqref{ub13} holds for every $M\in\N$
\begin{align*}
	\lim_{h\uparrow 0} G^N_{\mu+h}=G^{N}_{\mu-},
\end{align*}
which proves \eqref{el}.
\end{proof}
Next we would like to show that $G^N$ is stochastically continuous. We recall the definition of stochastic continuity. For $f\in  D(\R,\cS)$ we denote  $\triangle f(\mu)= f(\mu)-f(\mu -)$. Note that $\triangle f(\mu)\neq 0$ if and only if  $f$ has a jump at $\mu$. A process in $D(\R,\cS)$ is stochastically continuous if for every $\mu\in\R$ $\P(\triangle f(\mu)=0)=1$. Our next result shows that $G^N$ is stochastically continuous. In fact we prove something stronger.  
\begin{lemma}\label{lem:cond1}%\note{checked}
	 Fix $\mu\in \R$ and $0<\epsilon<1$. Then for $0<\delta<1$ and $N\geq 2(5+|\mu|)^3$
	\begin{align*}
		\P\Big(\sup_{s,t\in (\mu-\delta,\mu]}d(G^{N}_t,G^{N}_{s})>\epsilon\Big)\leq (2^{40}\epsilon^{-20}+1)\delta.
	\end{align*}
\end{lemma}
\begin{proof}
	Without loss of generality, we assume $\mu\in [-N^{1/3}+1,N^{1/3})$, the extension to $\R$ is straightforward as $G^N$ is constant outside $[-N^{1/3}+1,N^{1/3})$. From \eqref{maprox}
	\begin{align}\label{inc}
		\Big\{\sup_{s,t\in (\mu-\delta,\mu]}d(G^{N}_t,G^{N}_{s})>\epsilon\Big\}\subseteq \Big\{\sup_{s,t\in (\mu-\delta,\mu]}d_n(G^{N}_t,G^{N}_{s})>\epsilon-2^{-n}\Big\}
	\end{align}
Recall \eqref{Ir} and observe that
\begin{align}
	&\Big\{\sup_{s,t\in (\mu-\delta,\mu]}d_n(G^{N}_t,G^{N}_{s})>0\Big\}\nonumber\\
	&\subseteq\big\{d_n(G^{N}_{\mu},G^{N}_{\mu-\delta})>0\big\}\nonumber\\
	&=\big\{ I^{1/2-\frac{\mu}{4}N^{-1/3}}_i\neq I^{1/2-\frac{\mu-\delta}{4}N^{-1/3}}_i, \quad \text{ for some } -nN^{2/3}-1 \leq i\leq nN^{2/3}+1\big\}\label{eq15},
\end{align}
where in the second line we used the definition of $G^N$ and \eqref{mon} to conclude that  $d_n(G^N_t,G^N_s)\leq d_n(G^N_\mu,G^N_{\mu-\delta})$ for any $s,t\in(\mu-\delta,\mu]$. Next use \cite{BBS20}[Lemma 5.8] (see also the discussion on p.19 therein) with $\alpha=1/2-\frac{\mu-\delta}{4}N^{-1/3}$, $\beta=1/2-\frac{\mu}{4}N^{-1/3}$, $\theta=N^{-1/3}$ and $N\geq 2(5+|\mu|)^3$ to conclude that
\begin{align}\label{ub65}
	\P(\text{the event in \eqref{eq15}})\leq \delta(1+e^{10n}).
\end{align}
For the details on the derivation of \eqref{ub65} see \eqref{pr}. For $n=\ceil{\log_2(2\epsilon^{-1})}$, using \eqref{inc} 
\begin{align}\label{ub66}
	\P\Big(\sup_{s,t\in (\mu-\delta,\mu]}d(G^{N}_t,G^{N}_{s})>\epsilon\Big)\leq \P\Big(\sup_{s,t\in (\mu-\delta,\mu]}d_n(G^{N}_t,G^{N}_{s})>\epsilon/2\Big).
\end{align}
Using the bound \eqref{ub65} in \eqref{eq15} and then in \eqref{ub66}, that $n\leq \log_2(2\epsilon^{-1})+1$ and that $\log_2(e)<2$ gives the bound
\begin{align*}
	\P\Big(\sup_{s,t\in (\mu-\delta,\mu]}d(G^{N}_t,G^{N}_{s})>\epsilon\Big)\leq \big(1+2^{40}\epsilon^{-20}\big)\delta.
\end{align*}
\end{proof}
Next we would like to show that condition \ref{fdc} of Theorem \ref{thm:tc} holds for the sequence $G^N$. Recall the modulus continuity function $\omega$ for rcll functions from \ref{moc}.
\begin{proposition}\label{prop:rc}%\note{checked}
	Fix  $\mu^*>0$. For every $\epsilon>0$
\begin{align}\label{cond1}
	\lim_{\delta \rightarrow 0} \limsup_{N\rightarrow \infty} \P\Big(\omega(G^N,\mu^*,\delta)>\epsilon\Big)=0.
\end{align}
\end{proposition}
\begin{proof}
	Recall \eqref{theta}--\eqref{moc}. For $m\in\N$ and $a<b$ we define
	\begin{align*}
		\theta^n_X[a,b)=\sup_{s,t\in[a,b)}d_n\big(X(t),X(s)\big) \qquad X\in D(\R,\cS),
	\end{align*}
	and
	\begin{align}\label{mocn}
		\omega^n(X,t,\delta)=\inf\{\max_{1\leq i \leq m} \theta^n_X[t_{i-1},t_i)&:\exists m\geq 1,-t=t_0<t_1<...<t_m=t \\
		&\text{such that $t_i-t_{i-1}>\delta$ for all $i\leq m$}\}\nonumber.
	\end{align}
In words, the difference between $\omega^n$ and $\omega$ is that the former considers only the values of $X$ restricted to the interval $[-n,n]$. From \eqref{maprox} it follows that
\begin{align*}
	\omega(X,\mu^*,\delta)\leq \omega^n(X,\mu^*,\delta)+2^{-n}.
\end{align*}
Setting $n=\ceil{\log_2\epsilon^{-1}}+1$, from the last display it follows that in order to show \eqref{cond1} it is enough to show that
\begin{align}\label{lim4}
	\lim_{\delta \rightarrow 0} \limsup_{N\rightarrow \infty} \P\Big(\omega^n(G^N,\mu^*,\delta)>\frac{\epsilon}{2}\Big)=0.
\end{align}
Recall the event $\cA^{\delta,N,\mu^*,n}$ from \eqref{A}. Note that
\begin{align*}
	\Big(\cA^{2\delta,N,\mu^*,n}\Big)^c\subseteq \{\omega^n(G^N,\mu^*,\delta)=0\}.
\end{align*} 
From Proposition \ref{prop:ub2},  for $(32)^3\vee 64(\rho^*)^3<N$ and $\delta< 2^{-6}n^{-1/2}$
\begin{align*}
	\P\Big(\omega^n(G^N,\mu^*,\delta)>\frac{\epsilon}{2}\Big)\leq \P\big(\omega^n(G^N,\mu^*,\delta)>0\big)\leq \P\big(\cA^{2\delta,N,\mu^*,n}\big)\leq  2^{18}n\delta,
\end{align*} 
which implies \eqref{lim4}, which in turn implies that \eqref{cond1} indeed holds.
\end{proof}
Fix $x_0>0$ and define the sequence of functions $\{G^{x_0,N}\}_{N\in \N}$ in $D(\R,C([-x_0,x_0]))$ through
\begin{align}\label{Gx0}
	G^{x_0,N}_\mu(x)=G^{N}_\mu(x) \qquad \mu\in\R,x\in[-x_0,x_0]. 
\end{align}
In words, $G^{x_0,N}$ retains the information of $G^{N}$ restricted to the interval $[-x_0,x_0]$. As before, we call $\mu\in \R$ an epoch of $G^{x_0,N}$, if $G^{x_0,N}$ is non-constant on any open interval containing $\mu$. The next result says that the process $G^N$, when restricted to $[-x_0,x_0]$, attains finitely many values on $[-\mu_0,\mu_0]$.  
\begin{proposition}\label{prop:fe}
	Fix $x_0\geq1$, $\mu_0\geq1$. For $N,M\in \N$ , define
	\begin{align}\label{Bset2}
		\mathcal{B}^{M,N,\mu_0,x_0}&=\{\text{the process $G^{x_0,N}$ has no less than $M$ epochs in $[-\mu_0,\mu_0]$}\}.
	\end{align}
	Then, for  $M>1+2\mu_0x_0^{1/2}2^4$ and $N>(32)^3\vee 64\mu_0^3$
	\begin{align}\label{ub80}
		\P(\mathcal{B}^{M,N,\mu_0,x_0})\leq 2^{21}\frac{x_0\mu_0}{M-1} .
	\end{align}
	In particular, with probability one and for any $N\geq 1$, the process $G^{x_0,N}$ attains finitely many values on $[-\mu_0,\mu_0]$.
\end{proposition}
\begin{proof}
	From  definition \eqref{G}, it is clear, that for $\mu_0^3<N$ the epochs of $G^{x_0,N}$ are the same as that of $F^{N,x_0}$ from \eqref{F}. More precisely, for $M>0$, define the event
	\begin{align*}
		\mathcal{C}^{M,N,\mu_0,x_0}&=\{\text{the process $F^{N,x_0}$ has no less than $M$ epochs in $[-\mu_0,\mu_0]$}\},
	\end{align*}
	then for $\mu_0^3<N$
	\begin{align*}
		\mathcal{B}^{M,N,\mu_0,x_0}=\mathcal{C}^{M,N,\mu_0,x_0}.
	\end{align*}
	 This implies that 
	\begin{align*}
		\mathcal{B}^{M,N,\mu_0,x_0}=\mathcal{C}^{M,N,\mu_0,x_0}\subseteq \cA^{\frac{2\mu_0}{M-1},N,\mu_0,x_0}.
	\end{align*}
	To see why the last display holds, note that if there are $M\geq 2$ epochs of the process $F^{N,x_0}$ in an interval of size $2\mu_0$, there must be two epochs within distance of $\frac{2\mu_0}{M-1}$ from each other. From Proposition \ref{prop:ub2} it follows that for $M>1+2\mu_0x_0^{1/2}2^4$ and $N>(32)^3\vee 64\mu_0^3$
	\begin{align*}
		\P(\mathcal{B}^{M,N,\mu_0,x_0})=\P(\mathcal{B}^{\ceil{M},N,\mu_0,x_0})\leq 2^{20}x_0\frac{2\mu_0}{M-1},
	\end{align*}
	which implies the result.
\end{proof} 
\subsection{Finite dimensional distributions of the stationary horizon}\label{subsec:fdd}
The proof for the finite dimensional convergence to the SH here is different then the one that appeared in earlier versions of the paper and was communicated to the author by Evan Sorensen. The same proof was used in \cite{busani2022scaling} to show convergence of the TASEP speed process to the the SH in the finite dimensional sense. As the arguments in the following proof are soft, it results in a much shorter and arguably cleaner proof.  

For any continuous functions $f$ and $g$ on $\R$ such $f(0)=g(0)=0$ and 
\begin{equation}
	\sup_{-\infty<s\leq 0}[g(s)-f(s)]<\infty,
\end{equation}
we define the map
\begin{align}\label{cq}
	\cq(g,f)(t)=f(t)+\sup_{-\infty<s\leq t}[g(s)-f(s)]-\sup_{-\infty<s\leq 0}[g(s)-f(s)].
\end{align}
The map in \eqref{cq} is known as a Brownian queue (see \cite{o2001brownian}) and it satisfies the following important property; 
\begin{equation}
	\begin{aligned}
		&\text{If $B^1$ and $B^2$ are Brownian motions with diffusivity $\sigma>0$ and drift $\mu_1$ and $\mu_2$ respectively,} \\
			&\text{then $\cq(B^1,B^2)$ is again a bi-infinite Brownian motion with diffusivity $\sigma$ and drift $\mu_2$,} \\
			&\text{and that vanishes at the origin.}
	\end{aligned}
\end{equation}
 In Corollary \ref{cor:psi} we show that the map $\cq$ can be seen as a continuous version of the map $D$ acting on continuous functions rather than sequences, where the first argument can be thought of as the cumulative service times while the second as cumulative inter-arrival times.

Next we would like to find an analogue of $\md^{(k)}$ to elements in $\cS^k$ for some $k\in\N$. Define the map $\cq^k:\cS^k\rightarrow \cS^k$ as follows

	\begin{enumerate}
		\item $\cq^1(f_1)=f_1$
		\item $\cq^2(f_1,f_2)=[f_1,\cq(f_1,f_2)]$
		\item
		\begin{align}
			\cq^k(f_1,...,f_k)=[f_1,\cq\big(f_1,[\cq^{k-1}(f_2,...,f_k)]_1\big),...,\cq\big(f_1,[\cq^{k-1}(f_2,...,f_k)]_{k-1}\big)] \quad  k\geq 3\label{it}.
		\end{align}
		
	\end{enumerate}
It is helpful to compare the map $\cq^k$ to the map $\md^{(k)}$ form \eqref{md}, as well as the construction in \eqref{it} with that in \eqref{q}. For $1\leq i \leq k$ we denote the $i$'th element of $\cq^k$ by $\cq_i^k$. For $h\in\cS$, define $\bm{h}^k$ to be the element in $\cS^k$ whose every coordinate is the function $h$.  

%\begin{lemma}\label{lem:const2}%\note{checked}
%	Let $k\geq 1$, $c>0$  and $h\in\cS$. For $\h{f}\in \cS^k$
%	\begin{align}\label{const2}
%	\cq^k\big(c(\h{f}-\bm{h}^k)\big)=c(\cq^k(\h{f})-\bm{h}^k)
%	\end{align}
%\end{lemma}
%\begin{proof}
%	For $k=1$ the result is trivial. For $k\geq 2$ continue by induction. The base case, $k=2$, follows easily from considering the map $\cq$ in \eqref{cq} i.e.
%	\begin{align}\label{eq14}
%	\cq(c(f-h),c(g-h))=c\cq(f,g)-ch.
%	\end{align}
%	For the general step $k\geq 3$, assume the hypothesis holds for $k-1$ 
%	\begin{align*}
%	&\cq^k(c(f_1-h),...,c(f_k-h))\\
%	&\stackrel{\eqref{it}}{=}\Big[c(f_1-h),\cq\Big(c(f_1-h),[\cq^{k-1}(c(f_2-h),...,c(f_k-h))]_1\Big)\\
%	&,...,\cq\Big(c(f_1-h),[\cq^{k-1}(c(f_2-h),...,c(f_k-h))]_{k-1}\Big)\Big]\\
%	&=\Big[c(f_1-h),\cq\Big(c(f_1-h),c([\cq^{k-1}(f_2,...,f_k)]_1-h)\Big),...,\cq\Big(c(f_1-h),[c(\cq^{k-1}(f_2,...,f_k)]_{k-1}-h)\Big)\Big]\\
%	&\stackrel{\eqref{eq14}}{=}[cf_1-ch,c\cq\Big(f_1,[\cq^{k-1}(f_2,...,f_k)]_1\Big)-ch,...,c\cq\Big(f_1,[\cq^{k-1}(f_2,...,f_k)]_{k-1}\Big)-ch]\\
%	&=c\cq^k(f_1,...,f_k)-c\bm{h}^k,
%	\end{align*}
%	where in the second equality we used the induction hypothesis for $k-1$.
%\end{proof}

%\subsubsection{Controlling the probability of the events and finite dimensional convergence}\label{sssec:cpe}

For $\tau>0$ define the scaling operator $\cL^\tau:\cS\rightarrow \cS$ through 
\begin{align*}
	\cL^{\tau}(f)(t)=\tau^{-1/3}f(\tau^{2/3} t).
\end{align*}
If $\h{f}\in\cS^k$, we denote $\h{f}^\tau = \cL^\tau(\h{f})=(\cL^\tau(f_1),...,\cL^\tau(f_k))$. The following is the main result of this section.
%\begin{lemma}%\note{checked}
%	For $\tau>0$
%	\begin{align}\label{si}
%		\cL^\tau\big(\cq^k(\h{f})\big)&=\cq^k\big(\h{f}^\tau\big)\\
%		d^{(k)}_N\big(\h{f}^\tau,\h{g}^\tau\big)&=d^{(k)}_{\tau N}\big(\h{f},\h{g}\big)\label{si2}.
%	\end{align}
%
%\end{lemma}
%\begin{proof}
%	The proof of \eqref{si} is by induction on $k$. For $k=2$ the proof follows by observing that $\cL^\tau \big(\cq(f,g)\big)=\cq\big(f^\tau,g^\tau\big)$ which can be verified directly by considering the effect of applying $\cL^\tau $ to the different terms in \eqref{cq}. The case $k\geq 3$ is proved using \eqref{it}. The proof of \eqref{si2} is straightforward. 
%\end{proof}

\begin{proposition}\label{prop:conv}%\note{checked}
	Let $\h{t}\in\R^k$ be an increasing sequence. For each $N\in\N$  let $\rhov^N$ be defined through
	\begin{align*}
		\rho^N_i=1/2-\frac{t_i}{4}N^{-1/3}.
	\end{align*} 
Let $\h{I}^N\sim \nu_{\rhov^N}$. Let $W^{\bar{t}}=(W^{t_1},W^{t_2},...,W^{t_k})$ be a vector of $k$ bi-infinite Brownian motions vanishing at the origin, with diffusivity 2 and such that $W^{t_i}$ has drift $t_i$. Then 
\begin{align}\label{ub64}
	\cL^{N}\big(\cm(\md^{(k)}(\h{I}^N-\bm{2}^k))\big)\Rightarrow \cq\big(W^{\bar{t}}\big),
\end{align}
where the convergence is with respect to the weak topology on $\cS^k$.
\end{proposition}
\begin{proof}
	By Skorohord representation we may assume that  $\cm(\h{I}^N-\bm{2}^k)$ is coupled with $W^{\bar{t}}$ such that 
	\begin{equation}
		\cL^{N}\cm(\h{I}^N-\bm{2}^k)\Rightarrow W^{\bar{t}},
	\end{equation}
where convergence is uniformly on compacts and in probability. We now prove \eqref{ub64} by induction on $k$. From \eqref{q}
\begin{equation}\label{eq23}
	\cm(\md^{(k)}(\h{I}^N-\bm{2}^k))=\cm\Big(D\big(I_1^N,[\md^{(k-1)}(I^N_k,...,I^N_k)]_1\big),...,D\big(I_1^N,[\md^{(k-1)}(I^N_k,...,I^N_k)]_{k-1}\big)\Big).
\end{equation}
Now, from the induction hypothesis we know that each element of the vector  $\md^{(k-1)}(I^N_k,...,I^N_k)$ converges to a Brownian motion uniformly on compact sets and in probability. So, if \eqref{ub64} holds for $k=2$, we can apply it to each element of the vector that is on the right hand side of \eqref{eq23} to conclude the result. The proof of the base case $k=2$ can be found in Lemma \ref{lem:fdcon}.
	\end{proof}
%\begin{definition}\label{def:m}%\note{checked}
%	For each $k\in\N$ and an increasing vector $\h{t}\in\R^k$, we denote by $\pi^c_{\h{t}}\in\cM(\cS^k)$ the distribution of  the r.v. $X_{\h{t}}$ from \eqref{ub64}. 
%\end{definition}
%  

\begin{lemma}\label{lem:fdcon}
	Let $\bm{x}^1,\bm{x}^2$ be two i.i.d.\ sequences such that $D(\bm{x}^1,\bm{x}^2)$ is well defined. Assume that the one dimensional distribution of $\bm{x}^1$ and $\bm{x}^1$ has an exponential tail i.e.\ there exist $C,c>0$ such that 
	\begin{equation}
		\P(|x^i_0|>t)\leq Ce^{-ct}, \qquad i\in\{1,2\}.
	\end{equation}
	 Assume that $\bm{x}^i$ is coupled to a Bm with drift $\mu_i$ for $i\in\{1,2\}$ such that 
	\begin{equation}
		\limsup_{N\rightarrow \infty} \P\Big(\sup_{x\in[-a,a],i\in\{1,2\}}\Big|\cL^{N}\cm[\bm{x}^i](x)-B_i(x)\Big|>\epsilon\Big)=0.
	\end{equation}
	Then for every $a>0$ and  $\epsilon>0$ 
	\begin{equation}\label{eq22}
		\limsup_{N\rightarrow \infty} \P\Big(\sup_{x\in[-a,a]}\Big|\cL^{N}\cm\big[D(I^1,I^2)\big](x)-\cq(B^1,B^2)(x)\Big|>\epsilon\Big)=0.
	\end{equation}
\end{lemma}
\begin{proof}
	 For $a,S>0$, consider the event $E_{N,a,S}$ where the following holds
	\begin{enumerate}
		\item 
		\begin{equation}
			\sup_{-\infty<s\leq -aN^{2/3}}N^{-1/3}\Big(\cm[\bm{x}^1](s)-\cm[\bm{x}^2](s)\Big)=\sup_{-SN^{2/3}<s\leq -aN^{2/3}}N^{-1/3}\Big(\cm[\bm{x}^1](s)-\cm[\bm{x}^2](s)\Big)
		\end{equation}
	\item 
	\begin{equation}
		\sup_{-\infty<s\leq -a}\Big(B^1(s)-B^2(s)\Big)=\sup_{-S<s\leq -a}\Big(B^1(s)-B^2(s)\Big)
	\end{equation}
	\item 
	\begin{equation}
		\sup_{x\in[-a,a]}\Big|\sup_{-SN^{2/3}<s\leq xN^{2/3}}N^{-1/3}\Big(\cm[\bm{x}^1](s)-\cm[\bm{x}^2](s)\Big)-\sup_{-S<s\leq x}\Big(B^1(s)-B^2(s)\Big)\Big|\leq \epsilon/2.
	\end{equation}
\item \label{ub85}
\begin{equation}
	N^{-1/3}4\sup_{i\in[-aN^{2/3},aN^{2/3}]} |x^1_i|+|x^2_i|<\epsilon/2.
\end{equation}
	\end{enumerate}
From the definition of $\cq$ and Corollary \ref{cor:psi} we see that the event in \eqref{eq22} is contained in $E_{N,a,S}^c$. From \cite[Lemma A.3]{busani2022scaling} with $\xi(N)=N^{2/3}$, $\phi(N)=N^{-1/3}$ and $R=1$ we see that the probability of the first three items above  go to zero as $N\rightarrow \infty$. Using a simple union bound to bound item \ref{ub85} we conclude that $\lim_{S\rightarrow \infty}\limsup_{N\rightarrow \infty}\P(E_{N,a,S})\rightarrow 0$ as $N\rightarrow \infty$.  
\end{proof}
\subsection{Putting everything together}\label{subsec:pet}
\begin{figure}[t]	
	\centering
	\begin{subfigure}{0.55\textwidth}
		\includegraphics[width=0.9 \textwidth]{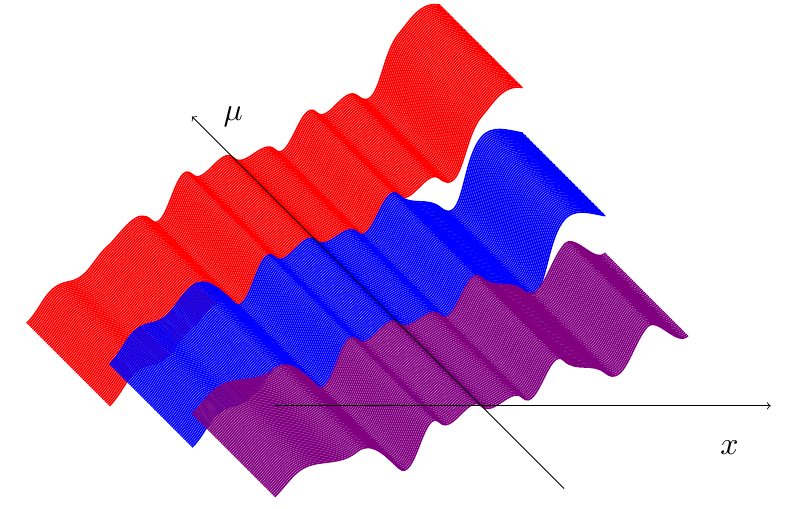} 
		\caption{}
		\label{sfig:SH}
	\end{subfigure}%
	\begin{subfigure}{0.55\textwidth}
		\includegraphics[width=0.9 \textwidth]{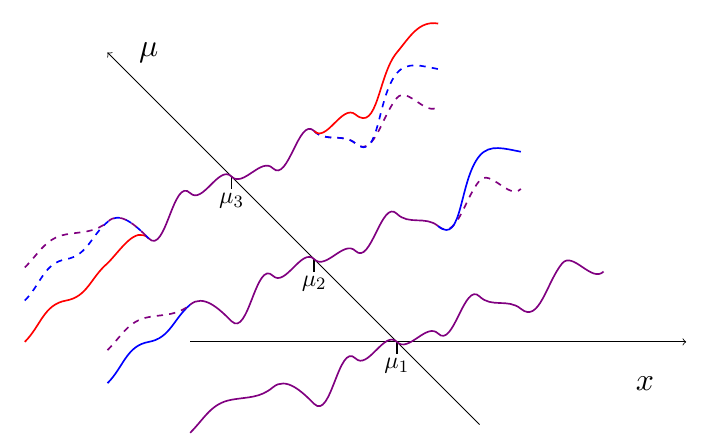}
		\caption{} 
		\label{sfig:SH2}
	\end{subfigure}
	
	\caption{Illustration of a typical realization of the Stationary Horizon. Figure \ref{sfig:SH}) A typical realization of the Stationary Horizon on a compact rectangle on the $x-\mu$ plane. The different colours in the picture represent different constant values (three in total in this picture) of the process along the $\mu$ axis. Figure \ref{sfig:SH2}) A finer depiction of the jump points of $G$. Recall the process $G^{x_0}$ from \eqref{Gx02}.  For $i\in \{1,2\}$, on the interval $[\mu_i,\mu_{i+1})$, $G^{x_0}$ attains the value $G^{x_0}_{\mu_i}\in C\big([-x_0,x_0]\big)$. Typically a change at $\mu_2$ will consist of a (random) compact interval around the origin (on the $x$-axis) where there is no change (solid purple curve), possible change  around  the  interval (solid curve in new colour), where  $G$ will take larger\textbackslash smaller values (on the $x$-axis). To facilitate the comparison between old and new values of $G$,  we keep the old values in dashed lines and in the same colour as when they were solid. The dynamics of the SH can therefore be described by a "sticky" line ensemble, branching  away as we zoom away from the origin or vary $\mu$.}
	\label{fig:SH}
\end{figure}
In this subsection we combine the different results obtained in the previous subsections to prove Theorem \ref{thm:sh} and Theorem \ref{thm:prop}.

\begin{proof}[Proof of Theorem \ref{thm:sh}] %\note{checked}
	The result will follow from Lemma \ref{lem:conv} once we verify all the conditions in its statement for the sequence $\{G^N\}_{N\in\N}$. Condition \ref{ss} holds from Lemma \ref{lem:cond1}. Condition \ref{cfd} holds from Proposition \ref{prop:conv} . Finally Condition \ref{sc} follows from Proposition \ref{prop:rc}.
\end{proof}
\begin{proof}[Proof of Theorem \ref{thm:prop}]%\note{checked}
	Part \ref{odd}. This follows from the fact that the prelimiting random walk $\cm\big(I^{1/2-\frac\mu4N^{-1/3}}\big)$, under diffusive scaling, converges to a Brownian motion with drift $\mu$. Part \ref{QM} is comes from Proposition \ref{prop:conv}. Part  \ref{sci}. From \eqref{G}, it is not hard to verify that i) considering the parameter $N$ to take values in $\R$, we still have that $G^N \rightarrow G$, where $G$ is the stationary horizon. ii) For $c>0$ we have
	\begin{align}\label{eq18}
		G^N_\mu(x)=cG^{c^3N}_{c\mu}(c^{-2}x).
	\end{align}
	Sending $N\rightarrow \infty$ on both sides of \eqref{eq18} we obtain the result.
	Part  \ref{pc} follows from Lemma \ref{lem:cond1} and Lemma \ref{lem:cl}. Part \ref{jp}. Let $x_0,\mu_0>0$. Recall the processes $G^{x_0,N}\in D\big(\R,C([-x_0,x_0])\big)$ defined in \eqref{Gx0}. Clearly, 
	\begin{align}\label{conv2}
		G^{x_0,N}\Rightarrow G^{x_0} \qquad N\rightarrow \infty,
	\end{align}
	where convergence is in $D\big(\R,C([-x_0,x_0])\big)$ equipped with a metric similar to the one in \eqref{met} only with 
%	\begin{align}
%	r^{x_0}(X,Y)=\inf_{\lambda\in\Lambda}\big[\gamma(\lambda)\vee \sup_{t\in [-x_0,x_0]}d\big(X(t),Y(\lambda(t))\big)\big] \qquad X,Y\in D(\R,C([-x_0,x_0])),
%	\end{align}
%	where
%	\begin{align*}
%	\gamma(\lambda)=\sup_{s<t}\abs{\log\frac{\lambda(s)-\lambda(t)}{s-t}}<\infty,
%	\end{align*} 
	\begin{align*}
		d(f,g)=\sup_{x\in[-x_0,x_0]} |f(x)-g(x)| \qquad f,g\in C\big([-x_0,x_0]\big).
	\end{align*}
	  Let
	 \begin{align*}
	 \mathcal{B}^{M,\mu_0,x_0}=\{\text{the process $G^{x_0}$ has no less then $M$ epochs in $[-\mu_0,\mu_0]$}\}.
	 \end{align*}	
	 Recall the event $\mathcal{B}^{M,N,\mu_0,x_0}$ from \eqref{Bset2}, on which the number of epochs of $G^{x_0,N}$ is at least $M$. Without loss of generality, we may assume that the convergence in \eqref{conv2} is an almost sure one. This implies that
	 \begin{align*}
	 	\P\big(\mathcal{B}^{M,\mu_0,x_0}\setminus \,\mathcal{B}^{M,N,2\mu_0,x_0}\big)\rightarrow 0,
	 \end{align*}
	 or that
	  \begin{align}\label{ub81}
	  	\P\big(\mathcal{B}^{M,\mu_0,x_0}\big)\leq \P\big(\mathcal{B}^{M,N,2\mu_0,x_0}\big)+\epsilon_N,
	  \end{align}
	  where $\epsilon_N\rightarrow 0$. Indeed, if the process $G^{x_0}$ has at least $M$ epochs in $[-\mu_0,\mu_0]$, then for  $N>N_0$ for some random $N_0$, the process $G^{x_0,N}$ must have at least $M$ epochs in the interval $[-2\mu_0,2\mu_0]$. Using the bound \eqref{ub80} in \eqref{ub81} and sending $N$ to infinity gives
	  \begin{align*}
	  	\P\big(\mathcal{B}^{M,\mu_0,x_0}\big)\leq 2^{22}\frac{x_0\mu_0}{M-2} \qquad M\geq 3,
	  \end{align*}
	  which is essentially \eqref{ub82}. In particular, with probability 1, the process $G^{x_0}$ has finitely many epochs on any compact interval. Part \ref{mono}. Let $x_1\leq x_2$ and $\mu_1 \leq  \mu_2$. Using \eqref{cocyc} and \eqref{mon} in \eqref{G} implies that for large enough $N$
	  \begin{align*}
	  	G^{N}_{\mu_2}(x_2)-G^{N}_{\mu_2}(x_1)\geq G^{N}_{\mu_1}(x_2)-G^{N}_{\mu_1}(x_1).
	  \end{align*}
	  This implies \eqref{mon2}.
\end{proof}
\appendix
\section{Appendix}
%\subsection{Bernoulli-Exponential distribution}
%Set $0<\lambda_1<\lambda_2$. For $\lambda>0$ we write $X\sim \text{Exp}(\lambda)$ and $Y\sim \text{BerExp}(\lambda_1,\lambda_2)$ if
%\begin{align*}
%\P(X\in dx)&=\lambda e^{-\lambda x}1_{\{x>0\}}dx\\
%\P(Y\in dy)&=\frac{\lambda_1}{\lambda_2}\delta_0(dy)+\frac{(\lambda_2-\lambda_1)}{\lambda_2}\lambda_1e^{-\lambda_1y}1_{\{y>0\}}dy.
%\end{align*}
%For $0<\lambda_1<\lambda_2$, let $X\sim \text{Exp}(\lambda_2)$ and $Y\sim \text{BerExp}(\lambda_1,\lambda_2)$ be two independent r.v.s, then 
%\begin{align*}
%X+Y\sim \text{Exp}(\lambda_1)
%\end{align*}
%Let $0<\lambda_1,\lambda_2<1$. Let $X_1\sim \text{Exp}(\lambda_1)$ and $X_2\sim \text{Exp}(\lambda_2)$. Then
%\begin{align*}
%Z:=(X_1-X_2)^-\sim \text{BerExp}(\lambda_2,\lambda_1+\lambda_2)
%\end{align*}
\subsection{Queues}
The following is a known result in queueing theory.
\begin{lemma}\label{lem:sourj}
	Fix $0<\rho_-<\rho_+$. Let $\arrv\sim \nu^{\rho_-}$ and $\servv\sim \nu^{\rho_+}$. Then
	\begin{align*}
		J_0:=S_{oj}(\servv,\arrv)_0=w_0+s_0,
	\end{align*}
	where $w_0$ and $s_0$ are independent and
	\begin{align*}
		&w_0\sim \mathrm{BerExp}(\rho_+-\rho_-,\rho_+)\\
		&s_0\sim \mathrm{Exp}(\rho_+).
	\end{align*}
	and therefore 
	\begin{align*}
		J_0\sim \mathrm{Exp}(\rho_+-\rho_-).
	\end{align*} 
\end{lemma}
The following lemma is a deterministic result from \cite[Lemma 4.4]{FS18}. 
\begin{lemma}%\note{checked}
	Let $\h{I}\in \cY^3$. Then
	\begin{align}\label{qi3}
		D^{(3)}(\h{I})=D^{(3)}(\sigma_1\h{I}).
	\end{align}
\end{lemma}
The following lemma extends this result.
\begin{lemma}%\note{checked}
	Fix $n\geq 3$. Let $\h{I}\in \cY^n$, then for $1 \leq i  \leq n-2$
	\begin{align}\label{si3}
		D^{(n)}(\h{I})=D^{(n)}(\sigma_i\h{I})
	\end{align}
\end{lemma}
\begin{proof}
	First note that as $1\leq i \leq n-2$ and $\h{I}\in\cY^n$, it follows that  $[\denp(\sigma_i\h{I})]_j>[\denp(\h{I})]_n$ for every $1 \leq j<n$ so that the right hand side of \eqref{si3} is well defined. By definition
	\begin{align}\label{qink2}
		D^{(n)}(I^1,...,I^n)=D^{(i)}\Bigg(I^1,...,I^{i-1},D\Big(I^i,D\big(I^{i+1},D^{(n-(i+2)+1)}(I^{i+2},...,I^n)\big)\Big)\Bigg)
	\end{align}
	From \eqref{qi3}
	\begin{align}\label{qink}
		&D\Big(I^i,D\big(I^{i+1},D^{(n-(i+2)+1)}(I^{i+2},...,I^n)\big)\Big)\\
		&=D\Big(D(I^i,I^{i+1}),D\big(R(I^i,I^{i+1}),D^{(n-(i+2)+1)}(I^{i+2},...,I^n)\big)\Big)\nonumber
	\end{align}
	Plugging \eqref{qink} in \eqref{qink2}
	\begin{align*}
	D^{(n)}(I^1,...,I^n)&=D^{(i)}\Bigg(I^1,...,I^{i-1},D\Big(D(I^i,I^{i+1}),D\big(R(I^i,I^{i+1}),D^{(n-(i+2))}(I^{i+2},...,I^n)\big)\Big)\Bigg)\\
	&\stackrel{\eqref{qi}}{=}D^{(n)}(I^1,...,I^{i-1},D(I^{i+1},I^i),R(I^{i+1},I^i),I^{i+2},...,I^n)=D^{(n)}(\sigma_i\h{I}).
	\end{align*}
\end{proof}
\begin{lemma}[{\cite[Lemma B.2]{FS18}}]\label{lem:ind}%\note{checked}
	Let $\rho_1>\rho_2>0$. Let $I^1$ and $I^2$ be two independent i.i.d. sequences with $I^i_0\sim \rm{Exp}(\rho_i)$.
	\begin{enumerate}[label=(\roman*)]
		\item \label{ind1}For each $k\in \Z$ the random variables
		\begin{align*}
			\{D(I^1,I^2)_i\}_{i\leq k},\{R(I^1,I^2)_i\}_{i\leq k},S_{oj}(I^1,I^2)_k
		\end{align*}
		are mutually independent with marginals
		\begin{align*}
			D(I^1,I^2)_0\sim \text{Exp}(\rho_2) \quad R(I^1,I^2)_0\sim \text{Exp}(\rho_1) \quad S_{oj}(I^1,I^2)_0\sim \text{Exp}\big(\rho_1-\rho_2\big)
		\end{align*}
		\item \label{ind2}$D(I^1,I^2)$ and $R(I^1,I^2)$ are independent.
	\end{enumerate}
\end{lemma}
\begin{lemma}\label{lem:se}%\note{checked}
	Let $\servv,\arrv\in \qs$ and define $S_{oj}(\servv,\arrv)=\{J_i\}_{i\in\Z}$. For every integers $m\leq n$
	\begin{align}\label{se}
		\sum_{i=m}^n(J_{i-1}-a_i)^-=\Big(\inf_{m\leq i \leq n-1}J_{m-1}+\sum_{j=m}^{i-1}(s_j-a_j)+a_i\Big)^-
	\end{align}
\begin{proof}
	This follows from \cite[Lemma A.1.]{BBS20}.
\end{proof}
\end{lemma}
\subsubsection*{Comparing stationary queues to empty ones}\label{subsec:CSE}
Here we show that the output of a stationary queue can be controlled by that of a queue starting with zero waiting time. This tool is an important ingredient in the proof of Lemma \ref{lem:Fb}.

Let $m\in \Z$. Define $D_{m}:\qs\times \qs\rightarrow (\qs)^{[m,\infty)}$ as the restriction of the queueing map on $[m,\infty)$ i.e.
\begin{align*}
	D_m(I^1,I^2)_i=D(I^1,I^2)_i \quad for\quad  m\leq i <\infty.
\end{align*}
Recall \eqref{summ}--\eqref{psi} and note that
\begin{align}\label{sd}
	&S^{m,i}(D(I^1,I^2))=S^{m,i}(I^1)+\psi^m(S_{oj}(I^1,I^2)_{m-1},I^1,I^2)_i\\
	&\leq S^{m,i}(I^1)+\psi^m(0,I^1,I^2)_i\leq S^{m,i}(I^1)+\sup_{m\leq j\leq i}[S^{m,i}(I^2-I^1)]^++\sup_{m\leq j \leq i}|I^2_j|.\nonumber
\end{align}
Next define the map $D_{m,0}:\qs\times \qs\rightarrow \qs^{[m,\infty)}$ in the following way. First define the sequence
\begin{align*}
	J'_{m-1}&=0\\
	J'_i&=(a_i-J'_{i-1})^-+s_i \quad i\geq m,
\end{align*}
and then define
\begin{align}\label{qm}
	D_{m,0}(\servv,\arrv)_i=(J'_{i-1}-a_i)^-+s_i \quad i\geq m.
\end{align}
The map $D_{m,0}$ can also be obtained by setting its input prior to time $m$ to zero. That is, 
\begin{align*}
	D_{m,0}(\servv,\arrv)=D_{m}(\servv',\arrv')
\end{align*}
where
\begin{align*}
	\arrv'_i=
	\begin{cases}
		\arrv_i & i\geq m\\
		0 & i<m.
	\end{cases}
\end{align*} 
and where $\servv'$ is defined similarly.
\begin{lemma}\label{lem:qz}%\note{checked}
	For every $m\in\Z$ the following holds
	\begin{enumerate}
		\item \label{od}	
		\begin{align*}
			D(\servv,\arrv)_i\leq D_{m,0}(\servv,\arrv)_i \quad i\geq m
		\end{align*}
		\item \label{od2}
		The map $S^{m,i}\big(D_{m,0}(\cdot,\cdot)\big)$ is increasing in both variables, i.e.
		\begin{align*}
			S^{m,i}\big(D_{m,0}(\servv^1,\arrv^1)\big)\leq S^{m,i}\big(D_{m,0}(\servv^2,\arrv^2)\big)
		\end{align*}
	whenever
	\begin{align*}
		(\servv^1,\arrv^1)\leq (\servv^2,\arrv^2).
	\end{align*}
	\end{enumerate}
\end{lemma}
\begin{proof}
	The proof of \ref{od}. follows by induction on $i\geq m$. Let $J_{m-1}=S_{oj}(\servv,\arrv)_{m-1}$ and construct $J_i$ via the map
	\begin{align*}
		J_i=(a_i-J_{i-1})^-+s_i \quad i\geq m.
	\end{align*}
Using induction on $i$ we see that
\begin{align*}
	J'_i\leq J_i \quad \text{for all }i\geq m.
\end{align*}
This implies that
\begin{align*}
	D(\servv,\arrv)_i=(J_i-a_i)^-+s_i\leq (J'_i-a_i)^-+s_i \leq  D_{m,0}(\servv,\arrv)_i \quad i\geq m.
\end{align*}
Moving on to the  proof of \ref{od2}, let $\servv'\geq 0$ 
\begin{align}\label{ub70}
		&S^{m,i}\big(D_{m,0}(\servv^1+\servv',\arrv^1)\big)-	S^{m,i}\big(D_{m,0}(\servv^1,\arrv^1)\big)\\
		&= S^{m,i}(\servv')+\psi(0,\servv^1+\servv',\arrv^1)_i-\psi(0,\servv^1,\arrv^1)_i\nonumber.
\end{align}
As 
\begin{align*}
	\psi(0,\servv^1+\servv',\arrv^1)_i-\psi(0,\servv^1,\arrv^1)_i\geq -S^{m,i}(\servv'),
\end{align*}
it follows that the left hand side of \eqref{ub70} is non negative which shows that $S^{m,i}\big(D_{m,0}(\cdot,\cdot)\big)$ is increasing in its first element. Similarly for $\arrv'\geq0$, clearly,
\begin{align*}
		&S^{m,i}\big(D_{m,0}(\servv^1,\arrv^1+\arrv')\big)-	S^{m,i}\big(D_{m,0}(\servv^1,\arrv^1)\big)\\
		&=\psi(0,\servv^1,\arrv^1+\arrv')-\psi(0,\servv^1,\arrv^1)\geq 0,
\end{align*}
so $S^{m,i}\big(D_{m,0}(\cdot,\cdot)\big)$ is increasing in its second element as well, thus proving the claim.
\end{proof}
The map $D_{m,0}$ can be generalized to a map on $k$ variables in the following way. Let $k\geq 2$ and let $I^1,..,I^k$ be sequences. Define another set of sequences $\hat{I}^1,...,\hat{I}^k$ by
\begin{align*}
	\hat{I}^i_l=
	\begin{cases}
		I^i_l & l\geq m \quad \quad i\in \{1,..,k\}\\
		0 & l<m.
	\end{cases}
\end{align*} 
Define the map $D^{(k)}_{m,0}$ through
\begin{align*}
	[D^{(k)}_{m,0}(I^1,...,I^k)]_i=D^{(k)}(\hat{I}^1,...,\hat{I}^k)_i \quad \text{for $i\geq m$}.
\end{align*}
Analogous to Lemma \ref{lem:qz} we have the following.
\begin{lemma}\label{lem:qz2}%\note{checked}
	For every $m\in\Z$ the following holds
	\begin{enumerate}
		\item \label{od3}	
		\begin{align}\label{ode3}
			D^{(k)}(I^1,...,I^k)_i\leq D^{(k)}_{m,0}(I^1,...,I^k)_i \quad i\geq m
		\end{align}
		\item \label{od4}
		The map $S^{m,i}\big(D_{m,0}(\cdot,\cdot)\big)$ is increasing in all its variables, i.e.
		\begin{align}\label{ode4}
			S^{m,i}\big(D_{m,0}(I^1,...,I^k)\big)\leq S^{m,i}\big(D_{m,0}(T^1,...,T^k)\big)
		\end{align}
		whenever
		\begin{align*}
			(I^1,...,I^k)\leq (T^1,...,T^k).
		\end{align*}
	\end{enumerate}
\end{lemma}
\begin{proof}
	The proof of \ref{od3} follows by induction on $k$. The base case is  $k=2$  was proven in Lemma \ref{lem:qz}. Assume the claim holds for $k-1$ . Then for $i\geq m$
	\begin{align*}
			&D^{(k)}(I^1,...,I^k)_i=	D\big(I^1,D^{(k-1)}(I^2,...,I^k)\big)_i\leq D\big(I^1,D^{(k-1)}_{m,0}(I^2,...,I^k)\big)_i\\
			&\stackrel{\eqref{od}}{\leq} D_{m,0}\big(I^1,D^{(k-1)}_{m,0}(I^2,...,I^k)\big)_i=D^{(k)}_{m,0}(I^1,...,I^k)_i,
	\end{align*}
where in the first inequality we used that the map $D(\cdot,\cdot)$ is increasing in its second element. We have thus proven \ref{od3}. Similarly, 
\begin{align*}
	&S^{m,i}\big(D^{(k)}_{m,0}(I^1,...,I^k)\big)=S^{m,i}\Big(D_{m,0}\big(I^1,D^{(k-1)}_{m,0}(I^2,...,I^k)\big)\Big)\\
	&\leq  S^{m,i}(I^1)+\inf_{m \leq j\leq i}\Big(S^{m,j-1}(I^1)-S^{m,j}(D^{(k-1)}_{m,0}(I^2,...,I^k))\Big)^-\\
	& \leq S^{m,i}(I^1)+\inf_{m \leq j\leq i}\Big(S^{m,j-1}(I^1)-S^{m,j}(D^{(k-1)}_{m,0}(T^2,...,T^k))\Big)^-\\
	&=S^{m,i}\Big(D_{m,0}\big(I^1,D^{(k-1)}_{m,0}(T^2,...,T^k)\big)\Big)\stackrel{\eqref{od2}}{\leq} S^{m,i}\Big(D_{m,0}\big(T^1,D^{(k-1)}_{m,0}(T^2,...,T^k)\big)\Big) \\
	&=S^{m,i}\big(D^{(k)}_{m,0}(T^1,...,T^k)\big),
\end{align*}
where in the second inequality we used the induction hypothesis, thus proving  \ref{od4}.
\end{proof}
Next we obtain a simple bound on the queueing map in terms of its input. For integers $m\leq n$ define the function $\Phi_{m,n}:\R^{\llbracket m,\infty\rrbracket}\rightarrow \R_+$ through
\begin{align}\label{mx}
	\Phi_{m,n}(I)=\big[\sup_{m\leq i\leq n} S^{m,i}(I)\big]^+.
\end{align}
Define $\Mx_{m,n}:\qs\rightarrow \R$ through
\begin{align}
	\Mx_{m,n}(I)=\max_{m\leq i\leq n}|I_i|
\end{align}
\begin{lemma}\label{lem:bs}%\note{checked}
	Let $m,n\in \Z$ such that $m\leq n$ and let $I^0\in \qs$. Then
	\begin{align}\label{bs}
		\Phi_{m,n} (D^{(k)}_{m,0}(I^1,...,I^k)-I^0)\leq \sum_{l=1}^{k}\big[\Phi_{m,n}(I^l-I^{l-1})\big]+\sum_{l=2}^{k}\Mx_{m,n}\Big(\sum_{i=l}^{k}I^i\Big).
	\end{align}
\end{lemma}
\begin{proof}
	The proof follows by induction on $k$. The base case is $k=2$. Indeed, from \eqref{sd}
	\begin{align}\label{ub69}
		\Phi_{m,n}(D_{m,0}(I^1,I^2)-I^0)&=\big[\sup_{m\leq i\leq n} S^{m,i}(I^1-I^0)+\psi^m(0,I^1,I^2)_i\big]^+\\
		&\leq \Phi_{m,n}(I^1-I^0)
		+\Phi_{m,n}(I^2-I^1)
		+\Mx_{m,n}(I^2).\nonumber
	\end{align}
Assume the hypothesis holds for $k-1\geq 2$, then
\begin{align}\label{ph}
		&\Phi_{m,n} (D^{(k)}_{m,0}(I^1,...,I^k)-I^0)=\Phi_{m,n} (D^{(k-1)}_{m,0}(I^1,...,D_{m,0}(I^{k-1},I^k))-I^0)\\\nonumber
		& \leq \sum_{l=1}^{k-2}\big[\Phi_{m,n}(I^l-I^{l-1})\big]+\sum_{l=2}^{k-2}\Mx_{m,n}\Big(\sum_{i=l}^{k-2}I^i\Big)+\Phi_{m,n}(D_{m,0}(I^{k-1},I^k)-I^{k-2})\\
		&+\Mx_{m,n}\big(D_{m,0}(I^{k-1},I^k)\big).\nonumber
\end{align}
Note that by the base case $k=2$ in \eqref{ub69}
\begin{align}\label{ph2}
	&\Phi_{m,n}(D_{m,0}(I^{k-1},I^k)-I^{k-2})\\
	&\leq\Phi_{m,n}(I^{k-1}-I^{k-2})+\Phi_{m,n}(I^k-I^{k-1})+\Mx_{m,n}(I^k),\nonumber
\end{align}
and that
\begin{align}\label{M}
	\Mx_{m,n}\big(D_{m,0}(I^{k-1},I^k)\big)=\max_{m\leq i\leq n}|D_{m,0}(I^{k-1},I^k)_i|\leq \max_{m\leq i\leq n}|I^k_i+I^{k-1}_i|=\Mx_{m,n}\big(I^k+I^{k-1}\big).
\end{align}
Plugging \eqref{ph2} and \eqref{M} in \eqref{ph}, we obtain the result.
\end{proof}
\subsection{Results on Random walks}
\begin{lemma}\label{lem:rwub}%\note{Checked}
	Let $\Delta,\rho>0$ and let $\{X_i\}_{i\in \Z}$ be i.i.d. r.v.s such that $X_0\sim \mathrm{Exp}(\rho+\Delta,\rho)$. Assume that $\sqrt{\Delta^2+4m^{-1}}<\rho$, then
	\begin{align*}
		\P(\inf_{0\leq j\leq m}S^{0,j}(X)<-t)\leq  \Big(1+\frac{m^{-1}}{(\rho-\sqrt{\Delta^2+4m^{-1}})\rho}\Big)^me^{-\frac{\sqrt{\Delta^2+4m^{-1}}-\Delta}{2}t}.
	\end{align*}
\end{lemma}
\begin{proof}
	Let $\phi_{\rho}(\theta)$ be the moment generating function of the exponential distribution with intensity $\rho>0$. Using Doob's inequality
	\begin{align}\label{ub2}
		\P(\inf_{0\leq j\leq m}S^{0,j}(X)<-t)=\P(\sup_{0\leq j\leq m}-S^{0,j}(X)>t)\leq \big[\phi_{\rho+\Delta}(-\theta)\phi_{\rho}(\theta)\big]^me^{-\theta t}.
	\end{align}
	Compute
	\begin{align}\label{mgf}
		\phi_{\rho+\Delta}(-\theta)\phi_{\rho}(\theta)=\frac{(\rho+\Delta)\rho}{(\rho+\Delta+\theta)(\rho-\theta)}=1+\frac{\theta^2+\theta\Delta}{(\rho+\Delta+\theta)(\rho-\theta)}.
	\end{align}
	Setting $\theta_0=\frac{-\Delta+\sqrt{\Delta^2+4m^{-1}}}{2}>0$ such that $\theta_0^2+\theta_0\Delta=m^{-1}$, plugging \eqref{mgf} in \eqref{ub2}
	\begin{align*}
		&\P(\inf_{0\leq j\leq m}S^{0,j}(X)<-t)\leq \Big(1+\frac{m^{-1}}{(\rho-\theta)\rho}\Big)^me^{-\frac{\sqrt{\Delta^2+4m^{-1}}-\Delta}{2}t}\\
		&\leq  \Big(1+\frac{m^{-1}}{(\rho-\sqrt{\Delta^2+4m^{-1}})\rho}\Big)^me^{-\frac{\sqrt{\Delta^2+4m^{-1}}-\Delta}{2}t},
	\end{align*}
	which is the required result.
\end{proof}
\begin{lemma}\label{lem:rw}%\note{checked}
	Let $X=\{X_i\}_{i=1}^\infty$ be i.i.d. r.v's such that $X_1\sim \textrm{Exp}(\rho+\Delta)-\textrm{Exp}(\rho)$. Then
	\begin{align}\label{ub14}
	\P\Big(S^{0,m}(X)>t\Big)\leq \Big(1-\frac{\Delta^2}{4(\rho+\Delta/2)^2}\Big)^me^{-\Delta t/2 }\leq e^{-\frac{\Delta^2m}{4(\rho+\frac{\Delta}{2})^2}-\Delta t/2}. 
	\end{align}
	In particular, for $\Delta=\mu N^{-1/3}$ and $m=xN^{2/3}$ where $x>0$, and for large enough $N$,
	\begin{align}\label{ub15}
		\P\Big(S^{0,m}(X)>tN^{1/3}\Big)\leq e^{-\frac\mu 4\big(\frac{\mu x}{\rho^2}+t\big)}
	\end{align} 
\end{lemma}
\begin{proof}
	Using Chernoff's bound with $\theta=\Delta/2$ gives 
	\begin{align*}
			&\P\Big(S^{0,m}(X)>t\Big)\leq[\phi_{\rho+\Delta}(\theta)\phi_{\rho}(-\theta)]^me^{-\theta t}\leq \Big(1-\frac{\Delta^2}{4(\rho+\Delta/2)^2}\Big)^me^{-\Delta t/2 }\\
			&\leq e^{-\frac{\Delta^2m}{4(\rho+\frac{\Delta}{2})^2}-\Delta t/2},
	\end{align*}
	proving \eqref{ub14}. Plugging $\Delta$ and $m$ in \eqref{ub14} and sending $N\rightarrow\infty$ gives \eqref{ub15}.
\end{proof}
\begin{lemma}\label{lem:rw2}%\note{checked}
	Let $X=\{X_i\}_{i=1}^\infty$ be i.i.d. r.v's such that $X_1\sim \textrm{Exp}(\rho+\Delta)-\textrm{Exp}(\rho)$. Then for $t>0$ and any $m\in \N$
	\begin{align}\label{ub16}
	\P\Big(\sup_{1\leq j\leq m} S^{1,j}(X)>t\Big)
	\leq e^{-\Delta t}. 
	\end{align} 
\end{lemma}
\begin{proof}
	For $\theta=\Delta$, $\E(e^{\theta X_1})=1$ which implies that $\{e^{\theta S^{0,j}}\}_{j\geq 1}$ is a martingale. Using Doob's inequality   
	\begin{align*}
		\P\Big(\sup_{0\leq j\leq m} S^{0,j}(X)>t\Big)=\P\Big(\sup_{0\leq j\leq m} e^{\Delta S^{0,j}(X)}>e^{\Delta t}\Big) \leq e^{-\Delta t}.
	\end{align*}
\end{proof}
%\note{
%\begin{lemma}\label{lem:max}
%	Let $X=\{X_i\}_{i=1}^\infty$ be i.i.d. r.v's such that $X_1\sim \textrm{Exp}(\rho + \Delta)-\textrm{Exp}(\rho)$. Then for  $\Delta \leq (\sqrt{\frac{3}{2}}-1)2\rho$ and $M>2\Delta^{-1}$
%	\begin{align*}
%		\P\Big(\sup_{M<i}S^{0,i}(X)>0\Big)\leq 4e^{-\frac{\Delta^2 M}{12\rho^2}}. 
%	\end{align*} 
%\end{lemma}
%\begin{proof}
%	Denote $\tilde{M}=M\rho^2\Delta^{-1}6$ and let 
%	\begin{align*}
%		A_k=\Big\{S^{0,k\tilde{M}}\leq -\frac{kM}{2},\sup_{k\tilde{M}<i\leq (k+1)\tilde{M}}S^{k\tilde{M},i}(X)<\frac{kM}{2}\Big\}.
%	\end{align*}
%	Note that
%	\begin{align*}
%		\P\Big(\sup_{\tilde{M}<i}S^{0,i}(X)<0\Big)\geq \P\Big(\bigcap_{k\geq 1}A_k\Big).
%	\end{align*}
%	It follows that 
%	\begin{align*}
%	\P\Big(\sup_{\tilde{M}<i}S^{0,i}(X)\geq 0\Big)\leq \sum_{k=1}^{\infty }\P\Big(A_k^c\Big)\leq \sum_{k=1}^{\infty}\Big[C_1+C_2 \Big].
%	\end{align*}
%	For $\Delta \leq (\sqrt{\frac{3}{2}}-1)2\rho$
%	\begin{align*}
%		C_1&=\P\big(S^{0,k\tilde{M}}> -\frac{kM}{2}\big)\stackrel{\ref{lem:rw}}{\leq} e^{-\frac{\Delta^2 k\tilde{M}}{6\rho^2}+\Delta kM/4}=e^{-\frac{3\Delta kM}{4}}\\
%		C_2&=\P\big(\sup_{k\tilde{M}<i\leq (k+1)\tilde{M}}S^{k\tilde{M},i}(X)>\frac{kM}{2}\big)\stackrel{\ref{lem:rw2}}{\leq} e^{-\frac{\Delta kM}{2}}  
%	\end{align*}
%	It follows that for $M>2\Delta^{-1}$
%	\begin{align*}
%		\P\Big(\sup_{\tilde{M}<i}S^{0,i}(X)>0\Big)\leq \sum_{k=1}^{\infty}2e^{-\frac{\Delta kM}{2}}\leq 4e^{-\frac{\Delta M}{2}}  .
%	\end{align*}
%\end{proof}
%}
\begin{lemma}\label{lem:max}%\note{checked}
	Let $X=\{X_i\}_{i=1}^\infty$ be i.i.d. r.v's such that $X_1\sim \textrm{Exp}(\rho + \Delta)-\textrm{Exp}(\rho)$. Let $F$ be a random element in $\cS$ such that
	\begin{align*}
		d_N(F,\cm(X))\leq \cC_N \quad \forall N\geq 1
	\end{align*}
where $\{\cC_i\}_{i=1}^{\infty}$ is a sequence of positive random variables.
	 Then for  $\Delta \leq \frac{2\rho}{5}$ and $M\in\N$ such that  $M>24\Delta^{-2}\rho^2$
	 \begin{align}\label{ub40}
	 	\P\Big(\sup_{M<t}F(t)>0\Big)\leq 4e^{-\frac{\Delta^2 M}{12\rho^2}}+\sum_{k=1}^{\infty}\P\Big(\cC_{(k+1)M}\geq M\frac{k\Delta }{\rho^{2}48}\Big). 
	 \end{align}
\end{lemma}
\begin{proof}
	For $M\in \N$, denote 
	\begin{align}\label{Mt}
		\hat{M}=\Delta\rho^{-2}6^{-1}M,
	\end{align}
	 and let 
	\begin{align}\label{eq12}
		A_k&=\Big\{\cm(X)(kM)\leq -\frac{k\hat{M}}{2},\sup_{kM<s\leq (k+1)M}\cm(X)(s)-\cm(X)(kM)<\frac{k\hat{M}}{4}\Big\}\nonumber\\
			&\stackrel{\eqref{sum}}{=}\Big\{S^{1,kM}(X)\leq -\frac{k\hat{M}}{2},\sup_{kM<i\leq (k+1)M}S^{kM,i}(X)<\frac{k\hat{M}}{4}
		\Big\}\\
		B_k&=\Big\{\cC_{(k+1)M}<\frac{k\hat{M}}{8}\Big\},\nonumber
	\end{align}
	where in the second equality we used the fact that linear interpolation attains its maxima on the interpolated points. It is not hard to see that
	\begin{align*}
		A_k\cap B_k\subseteq \Big\{\sup_{kM<s\leq (k+1)M}F(s)<0\Big\}
	\end{align*}
	which implies that
	\begin{align*}
		\P\Big(\sup_{M<s}F(s)<0\Big)\geq \P\Big(\bigcap_{k\geq 1}A_k\cap B_k\Big).
	\end{align*}
	It follows that 
	\begin{align*}
		\P\Big(\sup_{M<s}F(s)\geq 0\Big)\leq \sum_{k=1}^{\infty }\Big[\P\Big(A_k^c\Big)+\P\Big(B_k^c\Big)\Big]\leq \sum_{k=1}^{\infty}\Big[C^1_k+C^2_k +\P\Big(B_k^c\Big)\Big].
	\end{align*}
	For $\Delta \leq (\sqrt{\frac{3}{2}}-1)2\rho$
	\begin{align*}
		C^1_k&\stackrel{\eqref{eq12}}{=}\P\big(S^{0,kM}> -\frac{k\hat{M}}{2}\big)\stackrel{\ref{lem:rw}}{\leq} e^{-\frac{\Delta^2 kM}{6\rho^2}+\Delta k\hat{M}/4}=e^{-\frac{3\Delta k\hat{M}}{4}}\\
		C^2_k&\stackrel{\eqref{eq12}}{=}\P\Big(\sup_{kM<i\leq (k+1)M}S^{kM,i}(X)>\frac{k\hat{M}}{4}\Big)\stackrel{\ref{lem:rw2}}{\leq} e^{-\frac{\Delta k\hat{M}}{4}},
	\end{align*}
	where in the second equality of the first line we used \eqref{Mt}. It follows that for $\hat{M}>4\Delta^{-1}$
	\begin{align*}
		\P\Big(\sup_{M<s}F(s)\geq 0\Big)\leq \sum_{k=1}^{\infty}2e^{-\frac{\Delta k\hat{M}}{4}} +\P\Big(B_k^c\Big)\leq \sum_{k=1}^{\infty}\P\Big(\cC_{(k+1)M}\geq\frac{k\hat{M}}{8}\Big)+4e^{-\frac{\Delta \hat{M}}{4}}  .
	\end{align*}
	The bound in \eqref{ub40} follows from  substituting $\hat{M}\mapsto M\rho^{-2}\Delta 6^{-1}$ in the last display for $M>24\Delta^{-2}\rho^2$.
\end{proof}
\subsection{Topological results} \label{sub:TR}
Recall the space $\cS$ from Subsection \ref{subsec:sh2}. The following result is well known. 
\begin{lemma}\label{lem:cs}%\note{checked}
	The metric space $(d,\cS)$ is complete and separable.
\end{lemma}
%\begin{proof}
%	To see that $\cS$ is separable, let $\{f^n_i\}_{i=1}^\infty$ be a dense set in $C[-n,n]$, then $\bigcup_{i,n\geq 1}f^n_i$ is dense in $\cS$, showing that $\cS$ is separable. Next we show that $\cS$ is complete. Let $\{f_i\}$ be a Cauchy sequence in $\cS$, and let $f^n_i=f_i|_{[-n,n]}$ be the restriction of $f_i$ on the interval $[-n,n]$. As $C([-n,n])$ is complete, for any $n\geq 1$ $\exists g^n\in C([-n,n],\R)$ such that 
%	\begin{align}\label{c}
%		d_n(f^n_i,g^n)\rightarrow 0 \qquad \text{ as } i\rightarrow \infty
%	\end{align}
%	 Define
%	\begin{align*}
%		g(x):=g^{\lceil |x| \rceil}(x) \qquad x\in \R,
%	\end{align*}
%	and note that $g\in\cS$. It follows that for every $n$
%	\begin{align*}
%		d(f_i,g)\leq \sum_{j=1}^{n}2^{-j}\frac{d_j(f_i,g)}{1+d_j(f_i,g)}+2^{-n+1},
%	\end{align*}
%which along with \eqref{c} implies the completeness.
%\end{proof}
From here until the end of the section, we assume $\cS$ is a complete separable space equipped with a metric $d$. Let us denote by $\cM(\cS)$ the set of probability measures on the Borel sigma algebra generated by the metric on $\cS$. Recall the Prohorov metric $d_{\text{Proh}}:\cM\times \cM\rightarrow \R_+$ given by
\begin{align}\label{pm}
	d_{\text{Proh}}(\mu,\nu)=\inf \big\{ \epsilon > 0 :\mu(A)\leq \nu(A^{\epsilon})+\epsilon \text{ and } \nu(A)\leq \mu(A^{\epsilon})+\epsilon \big\}.
\end{align}
The following holds
\begin{lemma}\label{lem:pm}\cite{Bil68}[Theorem 6.8]
Weak convergence in $\cM(\cS)$ is equivalent to convergence with respect to $d_{\text{Proh}}$, $\cM(\cS)$ is a complete and separable metric space.
\end{lemma}
We recall the space $D(\R,\cS)$ of rcll functions defined on $\R$ and taking values in $\cS$. We follow closely \cite{EKbook}[Chapter 3, Sections 5-9]. We note that in \cite{EKbook} the authors considered the space $D([0,\infty),E)$, where $E$ is a complete separable metric space. However, going over the relevant proofs in \cite{EKbook}[Chapter 3] one concludes that the extension to the real line i.e. $D(\R,\cS)$ is straightforward (See for example \cite{FV09}, where it was done for $D(\R,\R)$). Nevertheless, as we use a slightly different space than the one studied in \cite{EKbook} we define the metric on $D(\R,\cS)$ rather than just refer to \cite{EKbook}.

Let $D(\R,\cS)$ be the space of functions defined on $\R$, that are right continuous with left limits in $\cS$. Let $\Lambda'$ be the set of strictly increasing functions $\lambda$ mapping $\R$ to itself. Let $\Lambda$ be the set of Lipschitz continuous functions $\lambda\in\Lambda'$
\begin{align*}
	\gamma(\lambda)=\sup_{s<t}\abs{\log\frac{\lambda(s)-\lambda(t)}{s-t}}<\infty.
\end{align*}
For $x,y\in D(\R,\cS)$, define
\begin{align}\label{met}
	r(x,y)=\inf_{\lambda\in\Lambda}\big[\gamma(\lambda)\vee \int_0^\infty e^{-u}r(x,y,\lambda,u)\big],
\end{align}
where
\begin{align*}
	r(x,y,\lambda,u)=\sup_{0\leq t}d\big(x(t\wedge u),y(\lambda(t)\wedge u)\big)\vee \sup_{0\geq t}d\big(x(t\vee u),y(\lambda(t)\vee u)\big).
\end{align*}
Since $\cS$ is  complete and separable, the space $(D(\R,\cS),r)$ is also a complete and separable space (see \cite{EKbook}[Theorem 5.6]).
We recall the modulus of continuity $\omega:D(\R,\cS)\times \R_+ \times \R_+\rightarrow \R_+$ associated with the space $D(\R,\cS)$. For $X\in D(\R,\cS)$ define
\begin{align}\label{theta}
	\theta_X[a,b)=\sup_{s,t\in[a,b)}d\big(X(t),X(s)\big),
\end{align}
and
\begin{align}\label{moc}
	\omega(X,t,\delta)=\inf\{\max_{1\leq i \leq n} \theta_X[t_{i-1},t_i)&:\exists n\geq 1,-t=t_0<t_1<...<t_n=t \\
	&\text{such that $t_i-t_{i-1}>\delta$ for all $i\leq n$}\}\nonumber.
\end{align}
The following are a well known tightness criteria
\begin{theorem}\label{thm:tc}\cite{EKbook}[Corollary 7.4]%\note{checked}
	Let $\{X^N\}_{N\in\N}$ be a sequence of r.v.s in $D\big(\R,\cS\big)$.  Then $\{X^N\}$ is relatively compact if and only if the following conditions hold:
	\begin{enumerate}
		\item 	\label{mocc} For every $\epsilon>0$ and $t$ rational there exists a compact set $A_{t,\epsilon}$ such that
		\begin{align}\label{fdc2}
		\liminf_{N\rightarrow \infty} \P\Big(X^N(t)\in A_{t,\epsilon}^\epsilon\Big)\geq 1-\epsilon.
		\end{align}		
	\item \label{fdc}For every $\epsilon>0$ and $T>0$
	\begin{align}\label{mocc2}
	\lim_{\delta \rightarrow 0} \limsup_{N\rightarrow \infty} \P\Big(\omega(X^N,T,\delta)>\epsilon\Big)=0.
	\end{align} 
	\end{enumerate} 
\end{theorem}
\begin{lemma}\label{lem:cl}%\note{checked}
	Suppose $\{X^N\}_{N\in\N}$ is a sequence of r.v.s  in $D(\R,\cS)$. Let $t\in \R$. Assume that for every $0<\epsilon<1$ there exists $0<\delta<1$ and $N_1(\epsilon,\delta,t),C(\epsilon)>0$ such that  for $N>N_1$
	\begin{align}\label{ub68}
		\P\Big(\sup_{u,v\in(t-\delta,t+\delta]}d(X^N_u,X^N_v)>\epsilon\Big)<C\delta.
	\end{align}
	Suppose further that
	\begin{align}\label{conv}
		X^N\Rightarrow X,
	\end{align}
	for some $X\in D(\R,\cS)$.Then $X$ is stochastically continuous at $t$, i.e.
	\begin{align*}
	\P(\Delta X_t)=0.
	\end{align*}  
\end{lemma}
\begin{proof}
	From \eqref{conv}, w.l.o.g., we may assume that the convergence in \eqref{conv} is an a.s. one. It then follows that for every $\delta>0$ there exists a sequence $\{\lambda_N\}_{N\geq 1}\subseteq \Lambda$ and $N_0(\delta,\epsilon)>0$ such that for $N>N_0$ (see \cite{EKbook}[p.117])
	\begin{align}\label{ub67}
		\P\Big(\lambda_N(t-\delta,t+\delta]\subseteq [t-2\delta,t+2\delta],\sup_{u\in(t-\delta,t+\delta]}d(X_u,X^N_{\lambda_N(u)})<\frac{\epsilon}{4}\Big)>1-\epsilon/2.
	\end{align}
	By the triangle inequality
	\begin{align}\label{inc2}
		&\Big\{\sup_{u,v\in(t-\delta,t+\delta]}d(X_u,X_v)>\epsilon\Big \}\\
		&\subseteq \Big\{2\sup_{u\in(t-\delta,t+\delta]}d(X_u,X^N_{\lambda_N(u)})>\epsilon/2\Big \}\cap \Big\{\sup_{u,v\in(t-\delta,t+\delta]}d(X^N_{\lambda_N(u)},X^N_{\lambda_N(v)})>\epsilon/2\Big \}\nonumber.
	\end{align}
	 It follows from \eqref{ub67} and \eqref{inc2} that for $N$ large enough
	\begin{align*}
		\P\Big(\sup_{u,v\in(t-\delta,t+\delta]}d(X_u,X_v)>\epsilon\Big)&\leq \frac{\epsilon}{2}+\P\Big(\sup_{u,v\in(t-2\delta,t+2\delta]}d(X^N_{u},X^N_{v})>\epsilon/2\Big)\\
		&\stackrel{\eqref{ub68}}{<}\frac\epsilon 2+2C\delta.
	\end{align*}
	Setting $\delta=\frac{\epsilon}{4C}$
	\begin{align*}
			\P\Big(\sup_{u,v\in(t-\delta,t+\delta]}d(X_u,X_v)>\epsilon\Big)< \epsilon,
	\end{align*}
	which implies the result as $\epsilon$ is arbitrary small.
\end{proof}
\begin{lemma}\label{lem:conv}%\note{checked}
	Let $\{X^N\}_{N\in\N}$ be a sequence of r.v.s in $D\big(\R,\cS\big)$.  Let $T\subseteq \R$ be dense. Suppose that 
	\begin{enumerate}
		\item \label{ss} For every $t\in T$ and $0<\epsilon<1$, there exists $0<\delta<1$ and $N_1(\epsilon,\delta),C(\epsilon)>0$ such that  for $N>N_1$
		\begin{align*}
			\P\Big(\sup_{u,v\in(t-\delta,t+\delta]}d(X^N_u,X^N_v)>\epsilon\Big)<C\delta.
		\end{align*}
		\item \label{cfd}For $k\in\N$ and any $t_1,...,t_k\in T$
		\begin{align*}
			(X^N_{t_1},...,X^N_{t_k})\Rightarrow p_{t_1,...,t_k},
		\end{align*}
		for some $p_{t_1,...,t_k}\in \cM(\cS^k)$.
		\item \label{sc}$\{X^N\}_{N \in \N}$ satisfy \eqref{mocc2}.
	\end{enumerate}
	Then there exists a unique process $X\in D(\R,\cS)$ with finite dimensional distributions $p_{\cdot}$, such that 
	 \begin{align}\label{limit}
	 	X^N\Rightarrow X.
	 \end{align}
\end{lemma}
\begin{proof}
	From Condition \ref{cfd}, for every $t\in T$ the sequence of distributions $\{X^N_t\}$ is convergent. As $\cS$ is complete and separable, this sequence is also tight, which implies that  \eqref{fdc2} in Theorem \ref{thm:tc} holds for $\{X^N_t\}$ which implies that Condition \ref{mocc} holds for $\{X^N\}$ (see \cite{EKbook}[Theorem 2.2, Chapter 3]). Now suppose Condition \ref{sc} holds. As $\{X^N\}$ satisfies both Conditions \ref{mocc} and \ref{fdc} of Theorem \ref{thm:tc}, the sequence $\{X^N\}_{N\in\N}$ is sequentially compact. We can therefore find a subsequence $\{N_m\}_{m\in\N}$ such that $X^{N_m}\Rightarrow X$ as $m$ goes to infinity for some $X\in D(\R,\cS)$. Take $T'\subseteq T$ countable. Next we claim that for $t_1,...,t_k\in T'$
	\begin{align}\label{lim5}
		(X^{N_m}_{t_1},...,X^{N_m}_{t_k})\Rightarrow (X_{t_1},...,X_{t_k}) \quad \text{as $m\rightarrow \infty$}.
	\end{align} 
	In order to show \eqref{lim5} we will invoke the Continuous Mapping Theorem, for which application we will need to show that $X$ is stochastically continuous. But this follows from Condition \ref{ss} and Lemma \ref{lem:cl}. We continue showing\eqref{lim5} by considering the family of projections $\pi_{t_1,...,t_k}:D(\R,\cS)\rightarrow \cS^k$ for any $k\in\N$ and $(t_1,...,t_k)\in\R^k$, given by $\pi_{t_1,...,t_k}Y=(Y_{t_1},...,Y_{t_k})$ for any $Y\in D(\R,\cS)$. The mapping $\pi_{t_1,...,t_k}$ is continuous at functions in $D(\R,\cS)$ that are continuous at the points $t_1,...,t_k$.   From the stochastic continuity of $X$ we conclude that with probability one, $\pi_{t_1,...,t_k}$ is continuous at $X$ for every  $k\in\N$ and $t_1,...,t_k\in T'$. Fix $k\in\N$ and $t_1,...,t_k\in T'$. From the continuous mapping theorem we conclude that the following distributional limit holds
	\begin{align*}
		\lim_{m\rightarrow \infty}\pi_{t_1,...,t_k} X^{N_m} =\pi_{t_1,...,t_k} X.
	\end{align*}
	We have thus proven \ref{lim5}. From Condition \ref{cfd} we conclude that if $X^1$ and $X^2$ are two sub-sequential limits of $X^N$ then 
	\begin{align}\label{ed}
		\pi_{t_1,...,t_k} X^1 \sim \pi_{t_1,...,t_k} X^2 \quad \forall t_1,...,t_k\in T'.
	\end{align} 
	As the set $\cS$ is separable, by \cite{EKbook}[Proposition 7.1] the sigma algebra $\mathfrak{S}$ generated by $D(\R,\cS)$ equals the sigma algebra  $\sigma(\pi_t:t\in T')$ i.e.
	\begin{align}\label{eq16}
		\mathfrak{S}=\sigma(\pi_t:t\in T').
	\end{align}
	From \eqref{eq16} and \eqref{ed} it follows that 
	\begin{align*}
		X^1\sim X^2,
	\end{align*}
	which implies \eqref{limit}.
\end{proof}

\subsection{Proofs of some computations appearing in the text}

\begin{proof}\label{poub10}[Proof of \eqref{ub10}]
	\begin{align*}
		&=2\int_0^\infty 2^{-M}e^{-2^{-M}t}(1-e^{-2^{-M}t})\Big(4e^{-\frac{1}{32}(2x_0+1)^{-1/2}{N^{-1/3}N^{1/3}t}+1}+4(2x_0^N+1)e^{-N^{1/3}t/32}\Big)dt\\
		&=8e\frac{2^{-M}}{2^{-M}+\tfrac{1}{32}(2x_0+1)^{-1/2}}+8\cdot2^{-M}(2x_0^N+1)\frac{1}{2^{-M}+\tfrac{1}{32}N^{1/3}}\\
		&-\Big[8e\frac{2^{-M}}{2^{-M+1}+\tfrac{1}{32}(2x_0+1)^{-1/2}}+8\cdot2^{-M}(2x_0^N+1)\frac{1}{2^{-M+1}+\tfrac{1}{32}N^{1/3}}\Big]\\
		&=8e\frac{2^{-2M}}{(2^{-M}+\tfrac{1}{32}(2x_0+1)^{-1/2})(2^{-M+1}+\tfrac{1}{32}(2x_0+1)^{-1/2})}\\
		&+8(2x_0^N+1)\frac{2^{-2M}}{(2^{-M}+\tfrac{1}{32}N^{1/3})(2^{-M+1}+\tfrac{1}{32}N^{1/3})}\\
		&\leq 2^5(32)^2(2x_0+1)2^{-2M} +2^3(32)^2(2x_0+1)2^{-2M} \leq 2^{17}x_02^{-2M} +2^{15}x_02^{-2M} \leq  2^{18}x_02^{-2M}
	\end{align*}
\end{proof}
\begin{proof}[Proof of \eqref{ub65}]\label{pr}
	From \cite[Lemma 5.8]{BBS20}
	\begin{align}\label{ub74}
		\P(\text{the event in \eqref{eq15}})\leq 1-\frac{\beta}{\alpha}+\int\Big[\frac{\alpha}{\alpha+\theta}\frac{\beta}{\beta-\theta}\Big]^{2nN^{2/3}}e^{-\theta w}\frac{(\alpha-\beta)\beta}{\alpha}e^{-(\alpha-\beta)w}dw.
	\end{align}
	For $N>|\mu-\delta|^3$
	\begin{align}\label{ub75}
		1-\frac{\beta}{\alpha}\leq \frac{\frac{\delta}{4}N^{-1/3}}{1/2-\frac{\mu-\delta}{4}N^{-1/3}}\leq \delta N^{-1/3}.
	\end{align}
	Setting $\alpha=1/2-\frac{\mu-\delta}{4}N^{-1/3},\beta=1/2-\frac{\mu}{4}N^{-1/3}$ and $\theta=N^{-1/3}$
	\begin{align}\label{ub78}
		\Big[\frac{\alpha}{\alpha+\theta}\frac{\beta}{\beta-\theta}\Big]^{2nN^{2/3}}=\Big[1+\frac{\theta^2+\theta(\alpha-\beta)}{\alpha\beta+\theta(\beta-\alpha)-\theta^2}\Big]^{2nN^{2/3}}
	\end{align}
	For $N\geq 2(4+\delta+|\mu|)^3\geq 2(5+|\mu|)^3$
	\begin{align}\label{ub79}
		(\alpha\beta)+[\theta(\beta-\alpha)-\theta^2]\geq (1/4-\mu/4N^{-1/3})-[(1+\delta/4)N^{-1/3}]>1/8
	\end{align}
	Using \eqref{ub79} in \eqref{ub78} to obtain
	\begin{align}\label{ub76}
		\Big[\frac{\alpha}{\alpha+\theta}\frac{\beta}{\beta-\theta}\Big]^{2nN^{2/3}}\leq \big[1+8(1+\delta/4)N^{-2/3}\big]^{2nN^{2/3}}\leq e^{2n(8+2\delta)} \leq e^{10n}.
	\end{align}
	Finally 
	\begin{align}\label{ub77}
		\int e^{-\theta w}\frac{(\alpha-\beta)\beta}{\alpha}e^{-(\alpha-\beta)w}dw\leq \frac{\delta}{4}N^{-1/3}\int e^{-(\alpha-\beta+\theta)w}dw\leq \frac{\delta}{4(1+\delta/4)}.
	\end{align}
	Plugging \eqref{ub75}, \eqref{ub76} and \eqref{ub77} in \eqref{ub74} we see that for $N\geq 2(5+|\mu|)^3$
	\begin{align*}
		\eqref{ub74}\leq \delta N^{-1/3}+e^{2n(8+2\delta)}\frac{\delta}{4(1+\delta/4)}\leq \delta(1+e^{10n}).
	\end{align*}
\end{proof}
For a sequence $\{x_i\}_{i\in \Z}$ and integers $k\leq l$, we denote by $x^{k,l}=\sum_{i=k}^lx_i$. We use the convention $x^{k,l}=0$ whenever $l<k$.
\begin{lemma}\label{lem:depin}
	Let $\arrv$ and $\servv$ be such that $\depav=D(\arrv,\servv)$ is well defined. Let $x_i=s_{i-1}-a_i$. Then
	\begin{equation}
		d^{k,l}=\sup_{-\infty< i\leq l} x^{i,l}-w_{k-1}+a^{k,l}+s_l-s_{k-1}
	\end{equation} 
\end{lemma}
\begin{proof}
	Denote $x_i=s_{i-1}-a_i$. From equations (A.8) and (A.11) in \cite{bala-busa-sepp-20}
	\begin{equation}\label{eq19}
		\begin{aligned}
		d^{k,l}=&\Big(\inf_{k\leq i\leq l}w_{k-1}+x^{k,i}\Big)^-+s^{k,l}\\
		=&\Big(\sup_{k\leq i\leq l}-w_{k-1}-x^{k,i}\Big)^++s^{k,l}\\
		=&\sup_{k-1\leq i\leq l}y^{k-1,i}+s^{k,l}=\sup_{k-1\leq i\leq l}(y^{k-1,l}-y^{i+1,l}+s^{k,l})
		\end{aligned}
	\end{equation}
where 
\begin{equation}
	y_i=
	\begin{cases}
		0& i=k-1\\
		-w_{k-1}-x_i& i=k\\
		-x_i&i>k.
	\end{cases}
\end{equation}
We have
\begin{equation}
	y^{k-1,l}+s^{k,l}=-w_{k-1}+\sum_{i=k}^{l}(a_i-s_{i-1})+\sum_{i=k}^{l}s_i=-w_{k-1}+\sum_{i=k}^{l}a_i+s_l-s_{k-1}
\end{equation}
Plugging the last display into \eqref{eq19}
\begin{equation}\label{eq20}
	\begin{aligned}
		d^{k,l}=&-w_{k-1}+\sum_{i=k}^{l}a_i+s_l-s_{k-1}+\sup_{k-1\leq i\leq l}(-y^{i+1,l})\\
	=&-w_{k-1}+\sum_{i=k}^{l}a_i+s_l-s_{k-1}+\sup_{k-1\leq i\leq l}(-y^{i+1,l})
	\end{aligned}
\end{equation}
Note that 
\begin{equation}
	\sup_{k-1\leq i\leq l}(-y^{i+1,l})=\max\big\{w_{k-1}+x^{k,l},\sup_{k+1\leq i\leq l} x^{i,l}\big\}=\sup_{-\infty< i\leq l} x^{i,l}
\end{equation}
Plugging the last display into \eqref{eq20} we obtain the result.
\end{proof}
In the next result we use the notation $\hat{\servv}:=\{\hat{s}_i\}_{i\in \Z}$ where $\hat{s}_i=s_{i-1}$.
\begin{corollary}\label{cor:psi}
	Let $\servv,\arrv$ and $\depav$ be as in Lemma \ref{lem:depin}. Then for every $t\in\R$
	\begin{equation}
		\cm[\depav](t)=\sup_{-\infty<s\leq t}\Big[\cm[\hat{\servv}](s)-\cm[\arrv](s)\Big]-\sup_{-\infty<s\leq 0}\Big[\cm[\hat{\servv}](s)-\cm[\arrv](s)\Big]+\cm[\bm{a}](t)+\cE(t),
	\end{equation}
where for every $M\in\Z_+$
\begin{equation}\label{eq21}
	\sup_{|s|\leq M}\big|\cE(s)\big|\leq 4\sup_{i\in [-M,M]} |s_i|+|a_i|.
\end{equation}
\end{corollary}
\begin{proof}
	From Lemma \ref{lem:depin}
	\begin{equation}
		\begin{aligned}
		d^{1,l}=&	\sup_{-\infty< i\leq l} x^{i,l}-w_{0}+a^{1,l}+s_{l-1}-s_0 \qquad l\geq 1\\
		d^{l,0}=&w_{l-1}-\sup_{-\infty< i\leq 0} x^{i,0}-a^{l,0}+s_{l-1}-s_{0} \qquad l\leq -1.
		\end{aligned}
	\end{equation}
from which it follows that 
	\begin{equation}
		\cm[\depav](t)=
		\begin{cases}
		\sup_{-\infty< i\leq l} x^{i,l}-\sup_{-\infty< s\leq -1} \cm[\bm{x}](s)+\cm[\arrv](t)+s_0-s_{t-1}& t\in\Z_+\\
		0 & t=0\\
		\sup_{-\infty< s\leq t} \cm[\bm{x}](s)-\sup_{-\infty< i\leq 0} x^{i,0}+\cm[\arrv](t)+s_{t}-s_{0}& t\in\Z_-
		\end{cases}
	\end{equation}
The last display translates to 
\begin{equation}
	\cm[\depav](t)=\sup_{-\infty<s\leq t}\cm[\bm{x}](s)-\sup_{-\infty<s\leq 0}\cm[\bm{x}](s)+\cm[\bm{a}](t)+\cE(t) \qquad t\in \R,
\end{equation}
where $\cE$ satisfies \eqref{eq21}. The result now follows from the linearity of $\cm$ and the definition of $\bm{x}$. 
\end{proof}

\bibliographystyle{plain}
\bibliography{Biblio}

\begin{thebibliography}{10}

\bibitem{BDJ99}
J.~Baik, P.A. Deift, and K.~Johansson.
\newblock On the distribution of the length of the longest increasing
  subsequence of random permutations.
\newblock {\em J. Amer. Math. Soc.}, 12:1119--1178, 1999.

\bibitem{BBS20}
M.~Bal\'azs, O.~Busani, and T.~Sepp{\"a}l{\"a}inen.
\newblock Local stationarity of exponential last passage percolation.
\newblock {\em Probab. Theory Relat. Fields}, 180:113–162, 2021.

\bibitem{BBS20a}
M\'{a}rton Bal\'{a}zs, Ofer Busani, and Timo Sepp\"{a}l\"{a}inen.
\newblock Non-existence of bi-infinite geodesics in the exponential corner
  growth model.
\newblock {\em Forum Math. Sigma}, 8:Paper No. e46, 34, 2020.

\bibitem{bala-busa-sepp-20}
M\'{a}rton Bal\'{a}zs, Ofer Busani, and Timo Sepp\"{a}l\"{a}inen.
\newblock Local stationarity in exponential last-passage percolation.
\newblock {\em Probab. Theory Related Fields}, 180(1-2):113--162, 2021.

\bibitem{Bar01}
Y.~Baryshnikov.
\newblock {GUEs and queues}.
\newblock {\em Probab. Theory Relat. Fields}, 119:256--274, 2001.

\bibitem{BGH19}
Riddhipratim Basu, Shirshendu Ganguly, and Alan Hammond.
\newblock Fractal geometry of {$\rm Airy_2$} processes coupled via the {A}iry
  sheet.
\newblock {\em Ann. Probab.}, 49(1):485--505, 2021.

\bibitem{BHS18}
Riddhipratim Basu, Christopher Hoffman, and Allan Sly.
\newblock Nonexistence of bigeodesics in integrable models of last passage
  percolation.
\newblock 2018.
\newblock {\tt arXiv:1811.04908}.

\bibitem{BSS19}
Riddhipratim Basu, Sourav Sarkar, and Allan Sly.
\newblock Coalescence of geodesics in exactly solvable models of last passage
  percolation.
\newblock {\em J. Math. Phys.}, 60(9):093301, 22, 2019.

\bibitem{BGH21}
Erik Bates, Shirshendu Ganguly, and Alan Hammond.
\newblock Hausdorff dimensions for shared endpoints of disjoint geodesics in
  the directed landscape.
\newblock {\em Electron. J. Probab.}, 27:--, 2022.

\bibitem{BBO05}
Philippe Biane, Philippe Bougerol, and Neil O'Connell.
\newblock Littelmann paths and {B}rownian paths.
\newblock {\em Duke Math. J.}, 130(1):127--167, 2005.

\bibitem{Bil68}
P.~Billingsley.
\newblock {\em {Convergence of Probability Measures}}.
\newblock Wiley ed., New York, 1968.

\bibitem{Bur56}
P.J. Burke.
\newblock The output of a queuing system.
\newblock {\em Operations Res.}, 4:699--704, 1956.

\bibitem{BF20}
Ofer Busani and Patrik Ferrari.
\newblock Universality of the geodesic tree in last passage percolation.
\newblock {\em Ann. Probab.}, to appear - 2020.

\bibitem{busani2022scaling}
Ofer Busani, Timo Sepp{\"a}l{\"a}inen, and Evan Sorensen.
\newblock Scaling limit of the tasep speed process.
\newblock {\em arXiv preprint arXiv:2211.04651}, 2022.

\bibitem{D21}
Duncan Dauvergne.
\newblock Last passage isometries for the directed landscape.
\newblock {\em arXiv preprint arXiv:2106.07566}, 2021.

\bibitem{DOV18}
Duncan Dauvergne, Janosch Ortmann, and B{\'a}lint Vir{\'a}g.
\newblock The directed landscape.
\newblock {\em arXiv preprint arXiv:1812.00309}, 2018.

\bibitem{DV21}
Duncan Dauvergne and B{\'a}lint Vir{\'a}g.
\newblock The scaling limit of the longest increasing subsequence.
\newblock {\em arXiv preprint arXiv:2104.08210}, 2021.

\bibitem{EKbook}
Stewart~N. Ethier and Thomas~G. Kurtz.
\newblock {\em Markov processes}.
\newblock Wiley Series in Probability and Mathematical Statistics: Probability
  and Mathematical Statistics. John Wiley \& Sons, Inc., New York, 1986.
\newblock Characterization and convergence.

\bibitem{FS18}
W.~Fan and T.~Sepp{\"a}l{\"a}inen.
\newblock Joint distribution of {B}usemann functions in the exactly solvable
  corner growth model.
\newblock {\em Prob. Math. Phys.}, 1:55--100, 2018.

\bibitem{FV09}
D.~Ferger and D.~Vogel.
\newblock Weak convergence of the empirical process and the rescaled empirical
  distribution function in the {S}korokhod product space.
\newblock {\em Teor. Veroyatn. Primen.}, 54(4):750--770, 2009.

\bibitem{GH21}
Shirshendu Ganguly and Milind Hegde.
\newblock Local and global comparisons of the airy difference profile to
  brownian local time.
\newblock 2021.

\bibitem{GRAS17}
Nicos Georgiou, Firas Rassoul-Agha, and Timo Sepp\"{a}l\"{a}inen.
\newblock Geodesics and the competition interface for the corner growth model.
\newblock {\em Probab. Theory Related Fields}, 169(1-2):223--255, 2017.

\bibitem{GRAS17b}
Nicos Georgiou, Firas Rassoul-Agha, and Timo Sepp\"{a}l\"{a}inen.
\newblock Stationary cocycles and {B}usemann functions for the corner growth
  model.
\newblock {\em Probab. Theory Related Fields}, 169(1-2):177--222, 2017.

\bibitem{GW91}
Peter~W. Glynn and Ward Whitt.
\newblock Departures from many queues in series.
\newblock {\em Ann. Appl. Probab.}, 1(4):546--572, 1991.

\bibitem{GTW01}
Janko Gravner, Craig~A. Tracy, and Harold Widom.
\newblock Limit theorems for height fluctuations in a class of discrete space
  and time growth models.
\newblock {\em J. Statist. Phys.}, 102(5-6):1085--1132, 2001.

\bibitem{H08}
Christopher Hoffman.
\newblock Geodesics in first passage percolation.
\newblock {\em Ann. Appl. Probab.}, 18(5):1944--1969, 2008.

\bibitem{HN97}
C.~Douglas Howard and Charles~M. Newman.
\newblock Euclidean models of first-passage percolation.
\newblock {\em Probab. Theory Related Fields}, 108(2):153--170, 1997.

\bibitem{HN01}
C.~Douglas Howard and Charles~M. Newman.
\newblock Geodesics and spanning trees for {E}uclidean first-passage
  percolation.
\newblock {\em Ann. Probab.}, 29(2):577--623, 2001.

\bibitem{JRAS19}
Christopher Janjigian, Firas Rassoul-Agha, and Timo Sepp{\"a}l{\"a}inen.
\newblock Geometry of geodesics through busemann measures in directed
  last-passage percolation.
\newblock {\em arXiv preprint arXiv:1908.09040}, 2019.

\bibitem{LN96}
Cristina Licea and Charles~M. Newman.
\newblock Geodesics in two-dimensional first-passage percolation.
\newblock {\em Ann. Probab.}, 24(1):399--410, 1996.

\bibitem{MQR17}
Konstantin Matetski, Jeremy Quastel, and Daniel Remenik.
\newblock The {KPZ} fixed point.
\newblock {\em Acta Math.}, 227(1):115--203, 2021.

\bibitem{N95}
Charles~M. Newman.
\newblock A surface view of first-passage percolation.
\newblock In {\em Proceedings of the {I}nternational {C}ongress of
  {M}athematicians, {V}ol. 1, 2 ({Z}\"{u}rich, 1994)}, pages 1017--1023.
  Birkh\"{a}user, Basel, 1995.

\bibitem{Occ03}
Neil O'Connell.
\newblock A path-transformation for random walks and the {R}obinson-{S}chensted
  correspondence.
\newblock {\em Trans. Amer. Math. Soc.}, 355(9):3669--3697, 2003.

\bibitem{o2001brownian}
Neil O'Connell and Marc Yor.
\newblock Brownian analogues of burke's theorem.
\newblock {\em Stochastic processes and their applications}, 96(2):285--304,
  2001.

\bibitem{OcY02}
Neil O'Connell and Marc Yor.
\newblock A representation for non-colliding random walks.
\newblock {\em Electron. Comm. Probab.}, 7:1--12, 2002.

\bibitem{P16}
Leandro P.~R. Pimentel.
\newblock Duality between coalescence times and exit points in last-passage
  percolation models.
\newblock {\em Ann. Probab.}, 44(5):3187--3206, 2016.

\bibitem{P18}
Leandro P.~R. Pimentel.
\newblock Local behaviour of airy processes.
\newblock {\em J. Stat. Phys.}, 173(6):1614--1638, 2018.

\bibitem{P21}
Leandro P.~R. Pimentel.
\newblock Ergodicity of the {KPZ} fixed point.
\newblock {\em ALEA Lat. Am. J. Probab. Math. Stat.}, 18(1):963--983, 2021.

\bibitem{PS02}
M.~Pr{\"a}hofer and H.~Spohn.
\newblock Scale invariance of the {PNG} droplet and the {A}iry process.
\newblock {\em J. Stat. Phys.}, 108:1071--1106, 2002.

\bibitem{MV21}
Mustazee Rahman and Balint Virag.
\newblock Infinite geodesics, competition interfaces and the second class
  particle in the scaling limit.
\newblock {\em arXiv preprint arXiv:2112.06849}, 2021.

\bibitem{RA18}
Firas Rassoul-Agha.
\newblock Busemann functions, geodesics, and the competition interface for
  directed last-passage percolation.
\newblock In {\em Random growth models}, volume~75 of {\em Proc. Sympos. Appl.
  Math.}, pages 95--132. Amer. Math. Soc., Providence, RI, 2018.

\bibitem{R81}
H.~Rost.
\newblock Non-equilibrium behavior of a many particle system: density profile
  and local equilibrium.
\newblock {\em Z. Wahrsch. Verw. Gebiete}, 58:41--53, 1981.

\bibitem{S18}
Timo Sepp{\"a}l{\"a}inen.
\newblock Variational formulas, busemann functions, and fluctuation exponents
  for the corner growth model with exponential weights.
\newblock {\em arXiv preprint arXiv:1709.05771}, 2017.

\bibitem{SX20}
Timo Sepp\"{a}l\"{a}inen and Xiao Shen.
\newblock Coalescence estimates for the corner growth model with exponential
  weights.
\newblock {\em Electron. J. Probab.}, 25:Paper No. 85, 31, 2020.

\bibitem{SS21a}
Timo Sepp{\"a}l{\"a}inen and Evan Sorensen.
\newblock Busemann process and semi-infinite geodesics in brownian last-passage
  percolation.
\newblock {\em arXiv preprint arXiv:2103.01172}, 2021.

\bibitem{SS21b}
Timo Sepp{\"a}l{\"a}inen and Evan Sorensen.
\newblock Global structure of semi-infinite geodesics and competition
  interfaces in brownian last-passage percolation.
\newblock {\em arXiv preprint arXiv:2112.10729}, 2021.

\bibitem{SK90}
W.~Szczotka and F.~P. Kelly.
\newblock Asymptotic stationarity of queues in series and the heavy traffic
  approximation.
\newblock {\em Ann. Probab.}, 18(3):1232--1248, 1990.

\end{thebibliography}

\end{document}